\documentclass{scrartcl}
\pdfoutput=1

\title{Computations in the unstable homology of moduli spaces of Riemann surfaces}
\author{Carl-Friedrich Bödigheimer\and Felix Boes\and Florian Kranhold}
\date{}

\KOMAoptions{headings=standardclasses,numbers=enddot}%

\usepackage[utf8]{inputenc}
\usepackage[T1]{fontenc}

\usepackage[british]{babel}
\usepackage{csquotes}                         %
\usepackage[UKenglish]{isodate}
\cleanlookdateon                              %

\usepackage[pdfusetitle,%
            bookmarksnumbered=true,%
            unicode,%
            hidelinks]{hyperref}
\newcommand{\mail}[1]{\upshape\href{mailto:#1}{\texttt{#1}}}

 \usepackage{xcolor}
 \hypersetup{
   colorlinks,
   linkcolor={black!45!red},
   citecolor={black!55!green},
   urlcolor ={black!45!blue}
 }

\makeatletter
\def\Hy@UseMaketitleString#1{%
  \ltx@IfUndefined{Hy@#1}{}{%
    \begingroup
      \def\and{, }
      \let\Hy@saved@hook\pdfstringdefPreHook
      \pdfstringdefDisableCommands{%
        \expandafter\let\expandafter\\\csname Hy@newline@#1\endcsname
        \let\newline\\%
      }%
      \expandafter\ifx\csname @pdf#1\endcsname\@empty
        \expandafter\pdfstringdef\csname @pdf#1\endcsname{%
          \csname Hy@#1\endcsname\@empty
        }%
      \fi
      \global\let\pdfstringdefPreHook\Hy@saved@hook
    \endgroup
  }%
}
\makeatother

\usepackage{lmodern}

\usepackage{amssymb,amsmath}           %
\usepackage{mathtools,letltxmacro}     %
\usepackage{stmaryrd}                  %

\usepackage[osf,sc]{mathpazo}

\usepackage{setspace}
\usepackage{ellipsis}

\usepackage[babel,tracking=true]{microtype}
\makeatletter
\@ifpackagelater{microtype}{2020/12/08}
  {\microtypesetup{nopatch=footnote}}  %
  {}                                   %
\makeatother

\UseMicrotypeSet[tracking]{smallcaps}
\SetTracking{encoding=*,shape=sc}{30}

\makeatletter
  \widowpenalty=\@M
\makeatother

\usepackage{scalerel}

\usepackage{relsize}                                   %

\usepackage[shortlabels]{enumitem}

\usepackage{graphicx}
\usepackage[margin    =20pt,
            font      ={footnotesize,stretch=1.06},
            format    =plain,
            labelsep  =space,
            labelfont =bf]{caption}

\usepackage{booktabs}

\usepackage{amsthm}

\newcommand{\tref}[2]{#1~\ref{#2}}

\deffootnote[1.8em]{1em}{1.8em}{%
   \textsuperscript{\thefootnotemark}}

\addtokomafont{pagehead}{\fontfamily{ppl}\selectfont}

\newcommand{\pazofrac}[2]%
  {\leavevmode\kern.1em\raise.65ex\hbox{\scriptsize #1}%
   \kern-.125em/\kern-.125em\lower-.05ex\hbox{\scriptsize #2}}

\usepackage{etoolbox}
\ifundef{\abstract}{}{
  \patchcmd{\abstract}{\quotation}{\quotation\noindent\ignorespaces}{}{}}

\makeatletter
\newcommand{\authBlkOn}{
  \date{}
  \addtokomafont{date}{\vspace*{-\baselineskip}}

  \def\and{, }
  \global\let\oldauthor\@author 
  \author{}
  \hypersetup{pdfauthor={\oldauthor}}

  \usepackage{authblk}

  \renewcommand\Affilfont{\itshape\small}
  \renewcommand\Authands{, and~}
  \renewcommand{\affilsep}{\baselineskip}
}
\makeatother

\recalctypearea 

\usepackage{tikz-cd}
\usepackage{tikz}
\usetikzlibrary{
  arrows,
  calc
}
\tikzset{
  }

\let \le       \leqslant
\let \ge       \geqslant

\let \geq      \geqslant
\let \setminus \smallsetminus

\let \phi      \varphi
\let \epsilon  \varepsilon
\let \theta    \vartheta
\let \binom    \tbinom
\let \del      \partial  

\LetLtxMacro\orgvdots\vdots
\LetLtxMacro\orgddots\ddots
\makeatletter
\DeclareRobustCommand\vdots{%
  \mathpalette\@vdots{}%
}
\newcommand*{\@vdots}[2]{%
  \sbox0{$#1\cdotp\cdotp\cdotp\m@th$}%
  \sbox2{$#1.\m@th$}%
  \vbox{%
    \dimen@=\wd0 %
    \advance\dimen@ -3\ht2 %
    \kern.5\dimen@
    \dimen@=\wd2 %
    \advance\dimen@ -\ht2 %
    \dimen2=\wd0 %
    \advance\dimen2 -\dimen@
    \vbox to \dimen2{%
      \offinterlineskip
      \copy2 \vfill\copy2 \vfill\copy2 %
    }%
  }%
}
\makeatother

\newcommand{\bm}  [1]{\mathbold{#1}}
\newcommand{\quot}[2]{\left.\raise0.5ex\hbox{$#1$} \right/\hspace*{-2px}\lower0.5ex\hbox{$#2$}}

\newcommand{\on}  [1]{\mathrm{#1}}

\usepackage[outline]{contour}
\usepackage[normalem]{ulem}

\contourlength{1pt}

\newcommand{\pa}     [1]{\mathopen{}\left(#1\right)\mathclose{}}                    %
\newcommand{\set}    [1]{\mathopen{}\left\{#1\right\}\mathclose{}}                  %
\makeatletter
\newcommand\dirlim{\mathop{\mathpalette\varlim@{\rightarrowfill@\scriptscriptstyle}}\nmlimits@}
\newcommand\invlim{\mathop{\mathpalette\varlim@{\leftarrowfill@\scriptscriptstyle}}\nmlimits@}
\makeatother

\newcommand{\map}{\on{map}}

\newcommand{\swr}{\kern.6px\wr\kern.6px}

\DeclareSymbolFont{eulargesymbols}{U}{zeuex}{m}{n}
\DeclareMathSymbol{\intop}{\mathop}{eulargesymbols}{"52}

\def\centerarc[#1](#2)(#3:#4:#5)%
    { \draw[#1] ($(#2)+({#5*cos(#3)},{#5*sin(#3)})$) arc (#3:#4:#5); }

\usepackage[scr=boondoxo, scrscaled=1]{mathalfa}       %

\renewcommand{\a}   {\mathfrak{a}}
\renewcommand{\b}   {\mathfrak{b}}
\ifdefined\C\renewcommand{\C}{\mathbb{C}}\else\newcommand{\C}{\mathbb{C}}\fi
\newcommand  {\E}   {\mathcal{E}}
\newcommand  {\e}   {\mathbf{e}}
\renewcommand{\i}   {\mathbf{i}}
\newcommand  {\K}   {\mathbb{K}}
\renewcommand{\L}   {\mathbb{L}}
\newcommand  {\N}   {\mathbb{N}}
\newcommand  {\n}   {\mathfrak{n}}
\renewcommand{\P}   {\mathcal{P}}
\newcommand  {\p}   {\mathfrak{p}}
\newcommand  {\Q}   {\mathbb{Q}}
\newcommand  {\q}   {\mathfrak{q}}
\newcommand  {\R}   {\mathbb{R}}
\renewcommand{\S}   {\mathbb{S}}
\newcommand  {\V}   {\mathscr{V}}
\newcommand  {\X}   {\mathbb{X}}
\newcommand  {\Z}   {\mathbb{Z}}

\newcommand{\bbA}{\mathbb{A}}

\newcommand{\bbF}{\mathbb{F}}

\newcommand{\bA}{\mathbb{A}}

\newcommand{\bF}{\mathbb{F}}

\newcommand{\bR}{\mathbb{R}}

\newcommand{\bZ}{\mathbb{Z}}

\newcommand{\caF}{\mathcal{F}}

\newcommand{\caH}{\mathcal{H}}

\newcommand{\caK}{\mathcal{K}}

\newcommand{\caP}{\mathcal{P}}
\newcommand{\caQ}{\mathcal{Q}}

\newcommand{\scC}{\mathscr{C}}

\newcommand{\scO}{\mathscr{O}}

\newcommand{\scV}{\mathscr{V}}

\newcommand{\bfa}{\mathbf{a}}
\newcommand{\bfb}{\mathbf{b}}
\newcommand{\bfc}{\mathbf{c}}
\newcommand{\bfd}{\mathbf{d}}
\newcommand{\bfe}{\mathbf{e}}
\newcommand{\bff}{\mathbf{f}}

\newcommand{\bfs}{\mathbf{s}}

\newcommand{\bfv}{\mathbf{v}}
\newcommand{\bfw}{\mathbf{w}}

\newcommand{\frH}{\mathfrak{H}}

\newcommand{\frM}{\mathfrak{M}}

\newcommand{\frP}{\mathfrak{P}}

\newcommand{\frS}{\mathfrak{S}}

\ifdefined\frq\renewcommand{\frq}{\mathfrak{q}}\else \newcommand{\frq}{\mathfrak{q}}\fi

\newcommand{\gen}[1]{\langle #1 \rangle}
\newcommand{\cyc}[1]{\langle #1 \rangle}

\newcommand{\re}{\operatorname{Re}}

\newcommand{\cF}{\mathcal{F}}

\newcommand{\cK}{\mathcal{K}}

\newcommand{\cP}{\mathcal{P}}
\newcommand{\cQ}{\mathcal{Q}}

\newcommand{\fM}{\mathfrak{M}}

\newcommand{\fS}{\mathfrak{S}}
\newcommand{\fT}{\mathfrak{T}}

\newcommand{\Par}{\mathfrak{P}}
\definecolor{ddgreen}{rgb}{.20,.60,.20}

\definecolor{dgrey}  {rgb}{.45,.45,.45}
\definecolor{dyellow}{rgb}{.84,.69,.00}
\definecolor{dred}   {rgb}{.90,.30,.30}
\definecolor{dblue}  {rgb}{.28,.32,.80}
\definecolor{dgreen} {rgb}{.17,.65,.16}

\definecolor{bgrey}  {rgb}{.85,.85,.85}
\definecolor{byellow}{rgb}{.95,.93,.40}
\definecolor{bred}   {rgb}{.95,.60,.60}
\definecolor{bblue}  {rgb}{.50,.50,.90}
\definecolor{bgreen} {rgb}{.70,.90,.60}
\definecolor{blila}  {rgb}{.96,.84,.96}

\definecolor{bbyellow}{rgb}{.95,.92,.60}
\definecolor{bbred}   {rgb}{.98,.85,.85}
\definecolor{bbblue}  {rgb}{.70,.70,.96}
\definecolor{bbgreen} {rgb}{.80,.96,.70}
\definecolor{bblila}  {rgb}{.94,.84,.97}

\definecolor{dlila}  {rgb}{.70,.30,.70}
\definecolor{lila}   {rgb}{.80,.60,.90}
\definecolor{turq}   {rgb}{.3,.7,.7}
\definecolor{grey}   {rgb}{.65,.65,.65}
\definecolor{dorange}   {rgb}{.9,.4,.2}

\newcommand{\bldots}{\raisebox{.5ex}{\fbox{$\ldots$}}}

\newcommand{\cDot}{\!\cdot\!}

\numberwithin{equation}{section}

\newcommand{\Br}{\mathfrak{B\kern-.5px r}}

\usepackage{xpatch}
\xpatchcmd\swappedhead{~}{.~}{}{}

\theoremstyle{plain}
\newtheorem{theo}   {Theorem}[section]
\newtheorem{lem}    [theo]{Lemma}
\newtheorem{prop}   [theo]{Proposition}

\theoremstyle{definition}
\newtheorem{defi}   [theo]{Definition}

\newtheorem{rem}    [theo]{Remark}

\newtheorem{constr} [theo]{Construction}
\newtheorem{nota}   [theo]{Notation}

\swapnumbers %
\newcommand{\thistheoremname}{}
\theoremstyle{definition}
\newtheorem{genericthm}[theo]{\thistheoremname}

\newlength{\fixmidfigure}
\newenvironment{tFigure}{
  \setlength{\fixmidfigure}{\lastskip}\addvspace{-\lastskip}
  \begin{figure}[t]
  }{
  \end{figure}
  \addvspace{\fixmidfigure}
}

\newenvironment{tTable}{
  \setlength{\fixmidfigure}{\lastskip}\addvspace{-\lastskip}  
  \begin{table}[t!]
  \centering
  \small
  
  \setlength\tabcolsep{2.5pt}
  \setlength{\belowrulesep}{0pt}
  }{
  \end{table}
  \addvspace{\fixmidfigure}
}

\newenvironment{hTable}{
  \begin{table}[h]
  \centering
  \footnotesize
  
  \setlength\tabcolsep{2.5pt}
  \setlength{\belowrulesep}{0pt}
  }{
  \end{table}
}

\usepackage[%
            backend      = biber,
            style        = alphabetic,
            maxbibnames  = 9,
            maxcitenames = 9,
            bibencoding  = utf8,
            datamodel    = ext-eprint]{biblatex}

\DefineBibliographyExtras{british}{%
  \renewcommand{\finalandcomma}{,}
}

\DefineBibliographyStrings{english}{
  backrefpage={↑},
  backrefpages={↑}
}

\DeclareSourcemap{
  \maps[datatype=bibtex]{
    \map{
      \step[fieldsource=pmid, fieldtarget=pubmed]
    }
  }
}

\DeclareFieldFormat{urn}{%
  \textsc{urn}\addcolon\space
  \ifhyperref
  {\href{https://nbn-resolving.org/#1}{\nolinkurl{#1}}}
  {\nolinkurl{#1}}}

\DeclareFieldFormat{sdr}{%
  \textsc{sdr}\addcolon\space
  \ifhyperref
  {\href{https://purl.stanford.edu/#1}{\nolinkurl{#1}}}
  {\nolinkurl{#1}}}

\DeclareFieldFormat{zbl}{%
  \textsc{zbl}\addcolon\space
  \ifhyperref
  {\href{https://zbmath.org/#1}{\nolinkurl{#1}}}
  {\nolinkurl{#1}}}

\DeclareFieldFormat{hdl}{%
  \textsc{hdl}\addcolon\space
  \ifhyperref
  {\href{https://hdl.handle.net/#1}{\nolinkurl{#1}}}
  {\nolinkurl{#1}}} 

\renewbibmacro*{eprint}{%
  \printfield{zbl}%
  \newunit\newblock
  \printfield{urn}%
  \newunit\newblock
  \printfield{sdr}%
  \newunit\newblock
  \printfield{hdl}%
  \newunit\newblock
  \iffieldundef{eprinttype}
  {\printfield{eprint}}
  {\printfield[eprint:\strfield{eprinttype}]{eprint}}}

\DeclareFieldFormat[misc]{title}{‘#1’}                  %
\AtBeginBibliography{\small}

\bibliography{bibliography.bib}

\usepackage{scrlayer-scrpage}
\lohead{Unstable homology of moduli spaces}
\cohead{}
\rohead{C.-F.\ Bödigheimer, F.\ Boes, F.\ Kranhold}
\begin{document}

\maketitle

\begin{abstract}
  In this article we give a survey of homology computations for moduli spaces
  $\fM_{g,1}^m$ of Riemann surfaces with genus $g\ge 0$, one boundary curve, and
  $m\ge 0$ punctures. While rationally and stably this question has a
  satisfying answer by the Madsen–Weiss theorem, the unstable homology remains
  notoriously complicated. We discuss calculations with integral, mod-$2$, and
  rational coefficients. Furthermore, we determine, in most cases, explicit
  generators using homology operations.\looseness-1
\end{abstract}

\section{Introduction and overview}
\label{sec:introduction}
Let $\fM_{g,n}^m$ denote the moduli space of Riemann surfaces of genus
$g \geq 0$ with $n \geq 1$ boundary curves and with $m \geq 0$ permutable
interior punctures. The boundary curves are numbered and parametrised; an
equivalence is a biholomorphic map between two surfaces that respects the
numbering and the parametrisations of the boundary curves; the punctures may be
permuted. This space is the quotient of the corresponding Teichmüller space
$\fT_{g,n}^m$ by the proper action of the mapping class group
$\Gamma_{g,n}^m$. Since any self-equivalence must fix the boundary curves
pointwise, this action is free, and since the Teichmüller space is homeomorphic
to a ball, the quotient $\fM_{g,n}^m$ is a manifold and has the homotopy type of
the classifying space $B\Gamma_{g,n}^m$.

Some of these moduli spaces have well-understood homotopy types: for $g=0$, the
spaces $\fM_{0,1}^m$ are homotopy equivalent to the unordered configuration
spaces $C_m(\mathring D^2)$ of $m$ points in the interior of a disc, and thus,
the mapping class group $\Gamma_{0,1}^m$ is isomorphic to the braid group
$\mathfrak{B\kern-.7px r}_m$ on $m$ strings. In the case of $g=1$, we consider the moduli space
of bounded $2$-tori (or, equivalently, of ‘directed’ elliptic curves).
The moduli space $\fM_{1,1}$ is equivalent to the complement of the trefoil knot
in the $3$-sphere and the mapping class group $\Gamma_{1,1}$ is isomorphic to the
third braid group $\mathfrak{B\kern-.7px r}_3$.\looseness-1

We are interested in the homology groups $H_\bullet(\frM_{g,n}^m)$. While it is
a famous result that increasing the genus is homologically stable
\cite{Harer-84}, and while the Madsen–Weiss theorem \cite{Madsen-Weiss} gives
a complete description of the stable homology with rational and mod-$p$ coefficients
(for the latter see \cite{Galatius}), the homology outside this
stable range remains notoriously complicated. We focus on the case of $n=1$ and
consider homology with coefficients in $\Z$, $\Q$, and $\bbF_2$. Many explicit
results concerning these homology groups are well-known, and other results have
been achieved over the last decades, partially by the authors and in several
bachelors’, masters’, and PhD theses.  The purpose of this work is two-fold: on
the one hand, we give a detailed overview of what is known so far, and on the
other hand, we make some new contributions.

\subsection{Homology operations}
\noindent There are various geometric constructions that relate the homology
groups $H_\bullet(\frM_{g,1}^m)$ via homology operations. Even though we will
study them extensively in §\,\ref{sec:operations}, let us mention the most
important ones already here.

First of all, it is a classical observation that the boundary connected sum
turns the collection $\coprod_{g,m}\frM_{\smash{g,1}}^m$ into an $H$-space, and thus
endows its homology with the structure of a (graded) Pontrjagin algebra, whose
product we denote by
\[
  -\cdot - \colon H_i(\fM_{g_1,1}^{m_1}) \otimes H_j(\fM_{g_2,1}^{m_1})
  \longrightarrow  H_{i+j}(\fM_{g_1+g_2,1}^{m_1+m_2}).
\]
Even better, this product is part of an $E_2$-structure (this was observed
in \cite{Miller,CFB-90II}).  As a consequence, we additionally have a Browder
bracket $[-,-]$, and, for homology modulo $2$ or for even-dimensional classes, 
a Dyer–Lashof square $Q$. 
There are plenty of relations that involve these operations and which
hold for each $E_2$-algebra, see \cite[§\,\textsc{iii}]{Cohen-Lada-May}. We give
a short summary of them in §\,\ref{subsec:c2}.

In §\,\ref{subsubsec:T}, we introduce a further operation
$T\colon H_i(\frM_{g,1}^m)\to H_{i+1}(\frM_{g+1,1}^{m-1})$ that, visually
speaking, picks one of the punctures and a tangential direction, and turns it
into an extra handle. This $T$-operation already appeared in the computations of
\cite{Mehner} and has proved useful to describe non-trivial homology classes.

We recall several other homology operations in §\,\ref{sec:operations}, for
example the Segal–Tillmann map \cite{Segal-Tillmann} and various operations
involving multiple boundary curves \cite{Kranhold}.

\subsection{Generators and homology tables}
In §\,\ref{sec:generators}, we describe several explicit homology classes,
called $\bfa$, $\bfb$, $\bfc$, $\bfd$, $\bfe$, $\bff$, $\bfs$, and $\bfv$.
Starting with them and using the above homology operations, we can characterise
many generators: these calculations are the main result of this paper. (In
these tables, an empty entry means that the corresponding homology group is
trivial; moreover, all homology groups in degrees higher than the rows shown are
trivial.)

\clearpage

\paragraph{Genus 0}
We start with the well-known case of genus $0$, which means we are looking at
unordered configuration spaces of particles in the plane.

\begin{quote}
  \textsc{Theorem \ref{thm:tableG0}.} \itshape 
  The homology groups $H_\bullet(\frM_{0,1}^m;\Z)$ and their generators
   are, for $m =0, \ldots, 5$, as shown in the following table: %
  \begin{hTable}
    \footnotesize
    \setlength\tabcolsep{4pt}
\begin{tabular}{rllllll}    
  \toprule
  ~    & $\fM^{\vphantom 1}_{0,1}$   & $\fM_{0,1}^1$    & $\fM_{0,1}^2$      & $\fM_{0,1}^3$        & $\fM_{0,1}^4$          & $\fM_{0,1}^5$\\
  \hline
  $0$ & $\bZ\gen{1}$  & $\bZ \gen{\bfa}$ & $\bZ \gen{\bfa^2}$ & $\bZ \gen{\bfa^3}$   & $\bZ \gen{\bfa^4}$     & $\bZ \gen{\bfa^5}$\\
  $1$ &               &                  & $\bZ \gen{\bfb}$   & $\bZ \gen{\bfa\bfb}$ & $\bZ \gen{\bfa^2\bfb}$ & $\bZ \gen{\bfa^3\bfb}$\\
  $2$ &               &                  &                    &                      & $\bZ_2 \gen{\bfb^2}  $ &  $\bZ_2 \gen{\bfa\bfb^2}$\\
  \bottomrule
\end{tabular}

    \caption{Homology groups for $g=0$ and $m = 0,\dotsc,5$}\label{tab:g0}
  \end{hTable}
\end{quote}
While $\bfa$ is a ground class, $\bfb=Q\bfa$ is a Dyer–Lashof square.
In fact, the $\bbF_2$-homology of these moduli spaces 
is the free Dyer–Lashof algebra generated by $\bfa$,
or, in other words, the polynomial algebra over $\bbF_2$ generated by all $\bfa_k$ 
with $\bfa_0\coloneqq \bfa$ and $\bfa_{k+1} \coloneqq Q\bfa_k$.

\paragraph{Genus 1}
The table for genus $1$ already has an interesting first column, since for
$m=0$ the mapping class group $\Gamma_{1,1}$ is isomorphic to the third braid
group.

\begin{quote}
  \textsc{Theorem \ref{thm:tableG1}.} \itshape 
  The homology groups $H_\bullet(\frM_{1,1}^m;\Z)$ and their generators are,
  for $m = 0, \ldots, 4$, as shown in the following table:
  \begin{hTable}
    \footnotesize
    \begin{tabular}{rlllll} 
  \toprule
  ~    & $\fM^{\vphantom 1}_{1,1}$ & $\fM_{1,1}^1$ & $\fM_{1,1}^2$                            & $\fM_{1,1}^3$                                      & $\fM_{1,1}^4$                                                \\
  \hline
  $0$~ & $\Z\gen{\bfc}$ & $\Z\gen{\bfa\bfc}$ & $\Z\gen{\bfa^2\bfc}$                           & $\Z\gen{\bfa^3\bfc}$                               & $\Z\gen{\bfa^4\bfc}$                                         \\
  $1$~ & $\Z\gen{\bfd}$ & $\Z\gen{\bfa\bfd}$ & $\Z\gen{\bfa^2\bfd} \oplus \Z_2\gen{\bfb\bfc}$ & $\Z\gen{\bfa^3\bfd} \oplus \Z_2\gen{\bfa\bfb\bfc}$ & $\Z\gen{\bfa^4\bfd} \oplus \Z_2\gen{\bfa^2\bfb\bfc}$         \\
  $2$~ &                & $\Z_2\gen{\bfe}$   & $\Z_2\gen{\bfa\bfe,\bfb\bfd}$                  & $\Z_2\gen{\bfa^2\bfe, \bfa\bfb\bfd}$               & $\Z_2\gen{\bfa^3\bfe, \bfa^2\bfb\bfd, \bfb^2\bfc}$           \\
  $3$~ &                &                    & $\Z_2\gen{\bff}$                               & $\Z \oplus \Z_2\gen{\bfa\bff, \bfb\bfe}$           & $\Z^2 \oplus \Z_2\gen{\bfa^2\bff, \bfa\bfb\bfe, \bfb^2\bfd}$ \\
  $4$~ &                &                    &                                                & $\Z^2$                                             & $\Z^3 \oplus \Z_2\gen{\bfb\bff}\oplus \Z_2$                  \\
  $5$~ &                &                    &                                                & $\Z$                                               & $\Z^2 \oplus \Z_2$                                           \\
  $6$~ &                &                    &                                                &                                                    & $\Z$                                                         \\
  \bottomrule
\end{tabular}

    \caption{Homology groups for $g=1$ and $m = 0, \dotsc, 4$}\label{tab:g1}
  \end{hTable}
\end{quote}
Again, while $\bfc$ is a ground class, the $T$-operation occurs here for the
first time in $\bfd=T\bfa$, which is (regarded as a loop in the classifying
space $B\Gamma_{1,1}$) a Dehn twist. Similarly, the class $\bfe$ can be written
as $E(\bfa^2)$ where $E$ is the operation from §\,\ref{subsubsec:E}. The class
$\bff$ has a more complicated description, see §\,\ref{subsec:f}. Note that
there are homology groups for which we cannot find generators in terms of
operations applied to known classes.

\paragraph{Genus 2}
The case of $g=2$ and $m=0$ was the first problem for which the simplicial model
from §\,\ref{sec:models} lead to new computational results
\cite{Ehrenfried,Abhau-CFB-Ehrenfried}. They have also been discovered
independently \cite{Godin} by graph-theoretical methods.

\begin{quote}
  \textsc{Theorem \ref{thm:tableG2}.} \itshape The homology groups $H_\bullet(\frM_{\smash{2,1}}^m;\Z)$
  and their generators are, for $m = 0, 1, 2$, as in following
  table (where $\lambda=\frac1\mu$ for some natural number $\mu$):
  \begin{hTable}
    \footnotesize
    \begin{tabular}{rlll}
  \toprule
  ~    & $\fM^{\vphantom 1}_{2,1}$                          & $\fM_{2,1}^1$                                                                    & $\fM_{2,1}^2$                                                                                           \\
  \hline
  $0$~ & $\Z\gen{\bfc^2}$                                   & $\Z\gen{\bfa\bfc^2}$                                                             & $\Z\gen{\bfa^2\bfc^2}$                                                                                  \\
  $1$~ & $\Z_{10}\gen{\bfc\bfd}$                            & $\Z_{10}\gen{\bfa\bfc\bfd}$                                                      & $\Z_{10}\gen{\bfa^2\bfc\bfd} \oplus \Z_2\gen{\bfb\bfc^2}$                                               \\
  $2$~ & $\Z_{2}\gen{\bfd^2}$                               & $\Z \oplus \Z_2\gen{\bfa\bfd^2}$                                                 & $\Z \oplus \Z_2\gen{\bfa^2\bfd^2, \bfb\bfc\bfd}$                                                        \\
  $3$~ & $\Z \gen{\lambda\bfs}\hspace*{-1px}\oplus \Z_{2} \gen{T\bfe}$\!\hspace*{2.3px} & $\Z \gen{\lambda\bfa\bfs}\hspace*{-1px}\oplus\Z \oplus \Z_2 \gen{\bfa\cDot T\bfe}\hspace*{-1px}\oplus\Z_2$\!\hspace*{2.3px} & $\Z \gen{\lambda\bfa^2\bfs}\hspace*{-1px} \oplus\Z^2 \oplus \Z_{2} \gen{\bfa^2\cDot T\bfe,\bfb\bfd^2}\hspace*{-1px}\oplus \Z_2^2$\! \\
  $4$~ & $\Z_2 \oplus \Z_3\gen{\bfv}$                       & $\Z_2^2 \oplus \Z_3 \gen{\bfa\bfv}\oplus\Z_3$                                    & $\Z \oplus \Z_2\gen{\bfb\cDot T\bfe}\oplus \Z_2^4 \oplus \Z_3 \gen{\bfa^2\bfv}\oplus \Z_3^2$            \\
  $5$~ &                                                    & $\Z$                                                                             & $\Z^2 \oplus\Z_2^4 \oplus \Z_3$                                                                         \\
  $6$~ &                                                    & $\Z$                                                                             & $\Z^2 \oplus \Z_2^3$                                                                                    \\
  $7$~ &                                                    &                                                                                  & $\Z_2$                                                                                                  \\
  \bottomrule
\end{tabular}

    \caption{Homology groups for $g=2$ and $m = 0,1,2$}\label{tab:g2}
  \end{hTable}
\end{quote}

\noindent Although nearly all entries in these three tables can easily be
exhibited as generators, two classes
need a more subtle treatment and hence are separated from the aforementioned
theorem; both appear in the PhD thesis of the second author \cite{Boes}:\looseness-1
\begin{quote}
  \hspace*{-.5px}\textsc{Theorem \ref{thm:s}.} \itshape 
  The class $\bfs$ in $H_3(\frM_{2,1};\Z)$ is a rational generator.  \\[.5\baselineskip]
  \textsc{Theorem \ref{thm:v}.} \itshape 
  The class $\bfv$ generates the $\Z_3$-summand in $H_4(\frM_{2,1};\Z)$.
\end{quote}

\noindent While the mere computations of the homology groups for $g\le 2$ were done in
\cite{Harer-91, Ehrenfried, Godin, Abhau, Mehner, Wang, Boes-Hermann}, our
contribution is the identification of generators.  Some of them have already
been found in \cite{Mehner, Boes-Hermann}, others occur here for the first
time. In particular, we give generators for the entire integral homology of
$\fM_{2,1}$, except for one $\Z_2$-summand in degree $4$. There are
similar tables over $\bbF_2$ which we discuss in Appendix \ref{sec:mod2} and
which relie on a recent result by Bianchi \cite{Bianchi}.

\paragraph{Genus 3}
The case $g=3$ is already much more complicated. Here we restrict ourselves to
$m=0$, that is: we only care about $H_\bullet(\frM_{3,1})$.

Let us start by considering the first and the second homology groups: it is a
famous result \cite[Thm.\,1]{Powell} that for $g\ge 3$, the mapping class
group $\Gamma_{g,1}$ is perfect, in other words $H_1(\frM_{g,1})=0$.  In
\cite[Thm.\,1.2]{Korkmaz-Stipsicz}, it has been shown that $H_2(\frM_{3,1})$ is
either $\Z$ or $\Z\oplus \Z_2$, and it has been shown in
\cite[Thm.\,4.9]{Sakasai} that the second of these two cases holds. The
$\Z_2$-summand is generated by $\bfc\bfd^2$, see the proof of
\cite[Lem.\,3.6]{Galatius-Kupers-RW} for details. We call the\footnote{The class
  $\bfw$ is uniquely determined up to sign; \cite{Galatius-Kupers-RW} uses the
  convention that $\kappa_1(\bfw)=12$ holds for the first Mumford–Miller–Morita
  class $\kappa_1$.} generator of the free part $\bfw$: it plays a prominent
role in the study of \emph{secondary} homological stability, see
\cite{Galatius-Kupers-RW}.

Furthermore, we have $H_i(\frM_{3,1})=0$ for $i\ge 10$; this can be seen for
example by inspecting the simplicial complex from §\,\ref{subsec:hilbert}. The
rational Betti numbers of $\fM_{3,1}$ have been calculated in
\cite[§\,6.5.1.4]{Boes-Hermann}; they are given by
\begin{align}\label{eq:BettiM31}
  \dim_\Q H_i(\frM_{3,1};\Q) 
  = 1,0,1,1,0,1,1,0,0,1
\end{align}
In her PhD thesis \cite{Wang}, Wang performed many computer-aided calculations
for $\fM_{3,1}$ in prime characteristic. Combining her results with the Betti
numbers (\ref{eq:BettiM31}), the integral homology of $\fM_{3,1}$ is determined
up to possible direct summands of the form $\Z_{\smash{p^k}}$, where $p\ge 29$ is a
prime:\footnote{To be precise, Wang states that also summands of the form
  $\Z_{\smash{p^k}}$ with $p<29$ a prime and $k$ a \emph{large} exponent can
  occur. However, she also calculates the $\bF_p$-Betti numbers
  \cite[p.\,68]{Wang} for $p<29$, and they are, for each such $p$, exactly the
  sum of the rational ones and the ones coming from the $p$-torsion summands
  that we have already found—this excludes these further summands.} 
  these possible summands are symbolised by $\bldots$\kern.5px.

\begin{quote}
  \textsc{Summary.} \itshape The homology groups $H_\bullet(\frM_{3,1}; \Z)$ and
  some of their generators are as shown in the following table:
  \begin{hTable}
    \footnotesize
    \begin{tabular}{rl}
  \toprule
  ~    & $\fM_{3,1}$                                                                                           \\
  \hline
  $0$~ & $\Z\gen{\bfc^3}$                                                                                      \\
  $1$~ & ~                                                                                                     \\
  $2$~ & $\Z\gen{\bfw} \oplus \gen{\bfc\bfd^2}$                                                                \\
  $3$~ & $\bldots \oplus \Z \oplus \Z_2 \oplus \Z_3 \oplus \Z_4 \oplus \Z_7$                                   \\
  $4$~ & $\bldots \oplus \Z_2^2 \oplus \Z_3^2$                                                                 \\
  $5$~ & $\bldots \oplus \Z \oplus \Z_2 \oplus \Z_3$                                                           \\
  $6$~ & $\bldots \oplus \Z \oplus \Z_2^3$                                                                     \\
  $7$~ & $\bldots \oplus \Z_2$                                                                                 \\
  $8$~ & $\bldots$                                                                                             \\
  $9$~ & $\bldots \oplus \Z$                                                                                   \\
  \bottomrule
\end{tabular}

    \caption{Homology groups for $g=3$ and $m = 0$}\label{tab:g3}
  \end{hTable}
\end{quote}

\paragraph{Genus 4 and higher}
Rationally, Harer’s stability theorem
\cite{Harer-84,Ivanov,Boldsen,RandalWilliams} and the Madsen–Weiss theorem
\cite{Madsen-Weiss} show that
\begin{align}\label{eq:MW}
  H_i(\frM_{g,1};\Q)\cong H^i(\frM_{g,1};\Q)\cong
  \Q[\kappa_1,\kappa_2,\dotsc]_{\text{(degree $i$)}}
\end{align}
for $i\le \frac23\hspace*{1px}g-\frac23$. Here $\kappa_i$ is the
$i$\textsuperscript{th} Mumford–Miller–Morita class: it has degree
$2i$. In addition to that, we can say something about $H_i$ for
$i\le 3$: we clearly have \mbox{$H_0(\frM_{g,1})\cong \Z\gen{\bfc^g}$}, and
$H_1(\frM_{g,1})=0$ for $g\ge 3$, as already noted. For the homological degrees $2$
and $3$, we collect the following two results from the literature:

\begin{itemize}
\item \emph{$H_2(\frM_{g,1}) \cong \Z\gen{\bfc^{g-3}\bfw}$ for $g\ge 4$.}\\[.25\baselineskip]
  From \cite[Thm.\,3.9]{Korkmaz-Stipsicz}, we know that
  $H_2(\frM_{4,1})\cong \Z$, in particular \mbox{$\bfc^2\bfd^2=0$}, as it
  has order $2$. Harer’s stability theorem implies that
  $\bfc\cDot-\colon H_2(\frM_{3,1})\longrightarrow H_2(\frM_{4,1})$ is surjective, and hence
  $\bfc\bfw$ is a free generator. Using Harer’s stability theorem once again,
  $\bfc^{g-4}\cDot -\colon H_2(\frM_{4,1})\longrightarrow H_2(\frM_{g,1})$ is an isomorphism
  for $g\ge 4$.
\item \emph{$H_3(\frM_{g,1};\Q)=0$ for $g\ge 4$.}\\[.25\baselineskip]
  For $g\ge 6$, we are already in the stable range and can invoke (\ref{eq:MW}). The case $g=4$ is
  shown in \cite[Thm.\,6.1]{Galatius-Kupers-RW}, based on
  \cite[Thm.\,1.4]{Tommasi}, and the case $g=5$ follows from
  \cite[Cor.\,5.7]{Galatius-Kupers-RW}: one shows that
  \mbox{$H_4(\frM_{6,1};\Q)\longrightarrow H_4(\frM_{6,1},\frM_{5,1};\Q)$} is epic, hence
  $H_3(\frM_{5,1};\Q)\longrightarrow H_3(\frM_{6,1};\Q)=0$ is injective.\looseness-1
\end{itemize}
Summarising these results, we end up with \tref{Table}{tab:g0-5}. Note that for
homological degrees at most $2$, it is clear what the stabilisation $\bfc\cDot -$
looks like. In degree $3$, however, it remains unknown: for example we do not know
if $\bfc\cDot\lambda\bfs$ is non-trivial or even of infinite order; similarly,
we do not know if $\bfc\cDot T\bfe$ is non-trivial.

\begin{hTable}
  \footnotesize
  \setlength\tabcolsep{4pt}
\begin{tabular}{cllllll}
  \toprule
  ~ & $\fM_{0,1}$   & $\fM_{1,1}$    & $\fM_{2,1}$                                  & $\fM_{3,1}$                              & $\fM_{4,1}$        & $\fM_{5,1}$          \\
  \hline
  $0$ & $\Z\gen{1}$ & $\Z\gen{\bfc}$ & $\Z\gen{\bfc^2}$                             & $\Z\gen{\bfc^3}$                         & $\Z\gen{\bfc^4}$   & $\Z\gen{\bfc^5}$     \\
  $1$ & ~           & $\Z\gen{\bfd}$ & $\Z_{10}\gen{\bfc\bfd}$                      & ~                                        & ~                  & ~                    \\
  $2$ & ~           & ~              & $\Z_2\gen{\bfd^2}$                           & $\Z_2\gen{\bfc\bfd^2}\oplus\Z\gen{\bfw}$ & $\Z\gen{\bfc\bfw}$ & $\Z\gen{\bfc^2\bfw}$ \\
  $3$ & ~           & ~              & $\Z\gen{\lambda \bfs}\oplus \Z_2\gen{T\bfe}$ & $\Z\oplus\text{[torsion]}$               & [torsion]          & [torsion]            \\
  \bottomrule
\end{tabular}

  \caption{$H_i(\frM_{g,1})$ for $0\le i\le 3$ and $g\le 5$}\label{tab:g0-5}
\end{hTable}

\enlargethispage{\baselineskip}
\subsection{Properties and relations}
We also study how the above generators behave when combining them: for example,
each class $x\in H_\bullet(\frM_{g,1}^m)$ gives rise to a (graded) homology
operation by multiplying with it. If $x$ is one of our first generators, we can
say the following:
\begin{enumerate}
\item multiplication with $\bfa$ is
  split injective by \cite[Thm.\,1.3]{CFB-Tillmann-01};
\item multiplication with $\bfb$ is injective modulo $2$,
  as we see in §\,\ref{subsec:b} using \cite{Bianchi}.
\item multiplication with $\bfc$ is the classical genus-stabilisation. Harer’s
  stability theorem tells us
  that this map\footnote{To be precise, the optimal slope from
    \cite{RandalWilliams} is formulated without punctures; however, the
    punctured case follows from the unpunctured one by a spectral sequence
    argument \cite{Hanbury}.} is surjective if
  $\bullet\le \frac23\hspace*{1px}g$ and an isomorphism if
  $\bullet \le \frac23\hspace*{1px}g-\frac23$.
\end{enumerate}

Additionally, we want to describe what relations hold between the above
generators and operations. Some of them have already been derived: for example,
the relation $Q\bfc=3\cdot \bfc\bfd$ can be found \cite[Ex.\,6]{Godin} and appears
in \cite[§\,1.2]{Mehner}. In particular,
\mbox{$[\bfc,\bfc]=2\cdot Q\bfc=6\cdot\bfc\bfd\ne 0$}, showing that the $E_2$-structure
on $\coprod_g\frM_{g,1}$ cannot be enhanced to an $E_3$-structure
\cite[Thm.\,2.5]{Fiedorowicz-Song}, although the group completion
$\Omega B\coprod_g \frM_{g,1}$ has the homotopy type of an \emph{infinite} loop
space \cite[Thm.\,\textsc{a}]{Tillmann-97}. We contribute to this collection of
relations in §\,\ref{sec:relations}. For example, we show the following
stabilisation property of the Browder bracket:

\begin{quote}
  \textsc{Proposition \ref{prop:ckill}.} \itshape
  For two classes $x \in H_\bullet(\frM_{g,1}^m)$, $x' \in H_\bullet(\frM_{g',1}^{m'})$,
  the Browder bracket vanishes after a single stabilisation step: \mbox{$\bfc \cdot [x,x'] = 0$.}
\end{quote}

\noindent Let us point out that it follows from abstract considerations that $[x,x']$
vanishes after \emph{finitely many} stabilisation steps: the group completion
\mbox{$\coprod_g\frM_{g,1}\to \frM_{\infty,1}\times\Z$} is given by iterated
stabilisations and respects the $E_2$-structure on both sides, and the right
side has the homology of an infinite loop space, so its $E_2$-Browder bracket
vanishes. Another result shows that many Browder brackets are trivial:

\begin{quote}
  \hspace*{-.2px}\mbox{\textsc{Proposition \ref{prop:bdiv}.}} \itshape
  For each $x \in H_\bullet(\frM_{g,1}^m)$, the Browder brackets
  $[\bfc,x]$ and $[\bfd,x]$ are divisible by $2$,
  and $[\bfe,x]=0$.\\[.5\baselineskip]
  \mbox{\textsc{Proposition \ref{prop:ac0}.}} \itshape
  $[\bfa,\bfc]=0$ and $[\bfd,\bfd]=0$.
\end{quote}

\noindent Finally, we show that the above $T$-operation often behaves
like a differential.

\begin{quote}
  \hspace*{-.5px}\textsc{Proposition \ref{prop:TT}.} \itshape For each class
  $x\in H_\bullet(\frM_{g,1}^m)$, the class $(T\circ T)(x)$ is divisible by $2$
  and of order $2$, i.e.\ with coefficients in $\bbF_2$ or $\Q$, we have
  $T \circ T = 0$.
\end{quote}

\paragraph{Acknowledgements} 
The results presented in this survey have been accumulated over many years.
Some computations are contained in the theses of several students as quoted in the text, 
other results are recent and new. 
We have benefited from discussions with many students and colleagues: 
Andrea Bianchi (who additionally made many useful comments on the almost final draft of this paper), 
Daniela Egas Santander, 
Johannes Ebert, 
Domenico Marasco, 
Meinard Müller, 
Alexander Kupers, 
Oscar Randal-Williams, 
Ulrike Tillmann, and 
Nathalie Wahl.
The combinatorial insights of Balász Visy
and the programming expertise of
Ralf Ehrenfried, 
Jochen Abhau, 
Stefan Mehner,
Rui Wang, and 
Anna Hermann 
have been central for our computations.
Moreover, we mention those who studied generalisations or special cases of the classical moduli spaces or of our methods:
Tobias Fleckenstein,
Niklas Hellmer,
Alexander Heß,
Annika Kiefner,
Oliver Kienast,
Franca Lippert,
Viktoriya Ozornova, 
Moritz Rodenhausen,
and
Luba Stein:
all their influence and help is gratefully acknowledged.\looseness-1

And finally: many of our computations were executed on equipment at
the Institute for Numerical Simulation and
the Institute for Discrete Mathematics of the University of Bonn, for which we are grateful.

\section{Models for moduli spaces}
\label{sec:models}
Many of our results rely on a finite combinatorial model for the moduli space
$\frM_{g,n}^m$, namely the space of \emph{parallel slit domains} $\frP_{g,n}^m$:
this is based on an old work of Hilbert \cite{Hilbert}, has been described
in the context of moduli spaces by Bödigheimer \cite{CFB-90I}, and admits
a (relative) multisimplicial description
\cite{Abhau-CFB-Ehrenfried,Boes-Hermann}.%
\subsection{Moduli spaces of surfaces with a direction}
\label{subsec:directedSurf}

We start with a description of the moduli spaces $\fM_{g,n}^m$ that is better
suited for our purposes: instead of a Riemann surface with $n$ parametrised
boundary curves, we consider a closed Riemann surface $F$ and specify distinct
points $Q_1, \dotsc, Q_n$, together with a non-zero tangent vector $X_i$ at
$Q_i$ for each $1\le i\le n$, and we demand all $Q_i$ to be different from the
punctures $P_1, \dotsc, P_m$. We call the $Q_i$ \emph{dipole points} and the
$P_j$ \emph{sinks}. To simplify notation, we write
$\cQ = ((Q_1, X_1), \ldots, (Q_n, X_n))$ for the ordered tuple of dipole points
with their tangent vectors and we write $\cP = \{ P_1, \ldots, P_m \}$ for the
set of sinks. We call $(F,\cQ,\cP)$ a \emph{directed surface of type $(g,n,m)$}.\looseness-1

A \emph{conformal equivalence} between two directed surfaces $(F, \cQ, \cP)$ and
$(F', \cQ', \cP')$ of type $(g,n,m)$ is a biholomorphic (conformal) mapping
$F \to F'$ sending $Q_{\smash i}^{\vphantom\prime}$ to $Q'_{\smash i}$, with its
derivative at $Q_{\smash i}^{\vphantom\prime}$ sending
$X_{\smash i}^{\vphantom\prime}$ to $X'_{\smash i}$, and sending $\caP$ to $\caP'$.  We
denote by $\cF = [F, \caQ, \caP]$ the conformal equivalence class and let
$\fM_{g,n}^m$ be the \emph{moduli space} of directed surfaces of type $(g,n,m)$,
with the topology induced by the Teichmüller metric of the corresponding
Teichmüller space: it is equivalent to the moduli space from the
introduction. There is an $n!$-sheeted covering map
$\widetilde{\frM}_{g,n}^m\to \frM_{g,n}^m$ where the set of sinks is
ordered.

It is a classical result that the corresponding Teichmüller space,
and hence also $\frM_{g,n}^m$ itself, has (real) dimension $6g-6+4n-2m$,
with one exception: $\dim(\frM_{0,1})=0$.
\subsection{Potential functions}
\enlargethispage{\baselineskip}
We follow \cite[§\,3.1]{CFB-90I}: given a
directed surface $\cF=(F,\caQ,\caP)$ of type $(g,n,m)$, a \emph{potential
  function} is a map \mbox{$u \colon F \to \overline{\bR} = \bR \cup \{\infty\}$} that
satisfies the following conditions:\vspace*{-3px}

\begin{enumerate}
\item $u$ is harmonic away from its singularities $Q_1, \ldots, Q_n$ and $P_1, \ldots, P_m$.\vspace*{-3px}
\item $u$ has at each $Q_i$ a \emph{dipole} in the direction $X_i$:
  for a local parameter $z$ around $Q_i$ with $z(Q_i) = 0$ and
  $T_{Q_i}z(X_i) = \del_x$, there is a real number $A_i$ and a holomorphic
  function $f_i$ satisfying locally\looseness-1
  \[u(z) = \re\pa{\tfrac{1}{z}} - A_i \cdot \log|z| + \re\pa{f_i(z)}.\]
\item $u$ has at each $P_j$ a \emph{logarithmic sink}: for a local parameter $z$
  around $P_j$ with $z(P_j) = 0$, there is a positive real number $B_j$ and
  a holomorphic function $g_j$ satisfying locally
  \[u(z) = B_j \cdot \log|z|  + \re\pa{g_j(z)}.\]
\end{enumerate}
It then follows from the residue theorem that $\sum_i A_i = \sum_j B_j$. For
each such collection $A_1,\dotsc, A_n, B_1,\dotsc,B_m$ of constants, there is,
up to a \emph{single} further additive constant $C \in \R$, exactly one such
$u$: while the existence of $u$ follows from the Dirichlet principle, the
uniqueness is a consequence of the maximum principle. The space of all potential
functions is therefore parametrised by an affine subspace
$\bA_{n}^m\subset\bR^{n+m+1}$, which is contractible and of dimension
$n+m$. Furthermore, we have a left action of $\fS_m$ on $\bA_{n}^m$ by permuting
the constants $B_1,\dotsc,B_m$ (but not the $A_1, \dotsc, A_n$) and we obtain a
bundle
\[{^\R\hspace*{-.8px}\pi} \colon {^\R\hspace*{-.8px}\frH_{g,n}^m}  \coloneqq 
  \widetilde{\frM}_{g,n}^m \times_{\fS_m} \bA_{n}^m \longrightarrow \frM_{g,n}^m\]
over $\frM_{g,n}^m$ with fibre $\bA_{n}^m$. Thus, the bundle map
${^\R\hspace*{-.8px}\pi}$ is a homotopy equivalence. The elements of the total
space ${^\R\hspace*{-.8px}\frH_{g,n}^m}$ are classes $[\caF,u]$, where $\caF$ is
a directed surface of type $(g,n,m)$ and $u$ is a potential function as above.

\subsection{The critical graph}

Let $u$ be a potential function on a directed surface $\cF=(F,\caQ,\caP)$ of
type $(g,n,m)$. We choose a metric on $F$ that is compatible with the conformal
structure, and consider the gradient vector field $\phi \coloneqq {-\nabla u}$
of steepest descent: it is defined away from $\caQ\cup\caP$. Then each $Q_i$ is
a pole of order $2$ and each $P_j$ is a pole of order $1$. If we denote the
critical points of $\phi$ by $S_1,\dotsc,S_l\in F$ and let $h_k$ be the index of
$S_k$, then it follows from the Poincaré–Hopf index theorem that
$h \coloneqq h_1 + \dotsb + h_l$ coincides with $2g-2+2n+m$. The left side
of \tref{Figure}{fig:graphSlit} shows a small part of such a gradient field.\looseness-1

The \emph{critical graph} $\cK \subseteq F$ is now declared as in \cite[§\,3.2]{CFB-90II}: it is
the embedded graph consisting of all $Q_i$, all $P_j$, and all critical points
$S_k$ as vertices; the edges are all flow lines running from some $S_{\smash k}$ to
either another $S_{\smash{k'}}$ or to some $Q_i$ or some $P_j$. The graph is directed
and without loops, and only the parametrisation of the flow lines depends on the
above choice of metric. The complement $F \setminus \cK$ has $n$ contractible
components: following the flow of $\phi$ backwards, we obtain a retraction of
$F \setminus \cK$ onto the open set where $u$ is larger than any critical value;
this set, however, consists of $n$ contractible open sets, each lying ‘in front
of’ some $Q_i$ with respect to the direction $X_i$. We call the component of
$F\setminus\caK$ that contracts towards $Q_i$ the \emph{basin} of $Q_i$, denoted $F_i$.

It follows that $u$ is on $F \setminus \cK$ the real part of a holomorphic
function $w = u + \bm{i}v$, and the imaginary part $v$ is on each component
$F_i$ uniquely determined up to an additive constant $D_i \in \R$.  Altogether,
$w$ is determined by the constants $A_1, \dotsc, A_n$, $B_1, \dotsc, B_m$, $C$,
and $D_1, \dotsc, D_n$. Note that, even though $v$ is not globally defined, we
can speak of its \emph{critical values}: they are the limits when approaching a
critical point $S_k$ at the boundary of a basin.  We can regard $w$ as a map
$F\setminus \cK \to \C \times \set{1,\dotsc,n}$. Its image is a collection of
$n$ \emph{slitted} planes: %
this will be essential for our reformulation in §\,\ref{subsec:Par}.\looseness-1

Keeping track of the additive constants $D_1,\dotsc,D_n$, we obtain
a trivial $n$-dimensional vector bundle $\frH_{g,n}^m\longrightarrow {^\R\frH_{g,n}^m}$:
elements in $\frH_{g,n}^m$ are classes $[\cF,u,w]$ where \mbox{$\cF = (F,\caQ,\caP)$}
is a directed surface of type $(g,n,m)$, $u$ is a potential function on $\cF$,
and $w$ is a holomorphic extension defined on $F \setminus \cK$.
The composition of bundle projections\looseness-1
\[\pi \colon \frH_{g,n}^m \longrightarrow {^\R\hspace*{-1px}\frH_{g,n}^m} \longrightarrow \frM_{g,n}^m\]
is a bundle with fibre $\bbA_n^m \times \R^n$, and hence an equivalence.
The fibre of the bundle has dimension $2n+m$, so together with our
observation from §\,\ref{subsec:directedSurf} that $\frM_{g,n}^m$ has
dimension\footnote{For the exceptional case, we have $\dim(\frH_{0,1})=2$.}
$6g-6+4n+2m$, it follows that $\frH_{g,n}^m$ has dimension $3h$ for
$h\coloneqq 2g-2+2n+m$.

\begin{tFigure}
  \centering
  \def\svgwidth{.68\columnwidth}
\begingroup%
  \makeatletter%
  \providecommand\color[2][]{%
    \errmessage{(Inkscape) Color is used for the text in Inkscape, but the package 'color.sty' is not loaded}%
    \renewcommand\color[2][]{}%
  }%
  \providecommand\transparent[1]{%
    \errmessage{(Inkscape) Transparency is used (non-zero) for the text in Inkscape, but the package 'transparent.sty' is not loaded}%
    \renewcommand\transparent[1]{}%
  }%
  \providecommand\rotatebox[2]{#2}%
  \ifx\svgwidth\undefined%
    \setlength{\unitlength}{525.15255145bp}%
    \ifx\svgscale\undefined%
      \relax%
    \else%
      \setlength{\unitlength}{\unitlength * \real{\svgscale}}%
    \fi%
  \else%
    \setlength{\unitlength}{\svgwidth}%
  \fi%
  \global\let\svgwidth\undefined%
  \global\let\svgscale\undefined%
  \makeatother%
  \begin{picture}(1,0.33350391)%
    \put(0,0){\includegraphics[width=\unitlength,page=1]{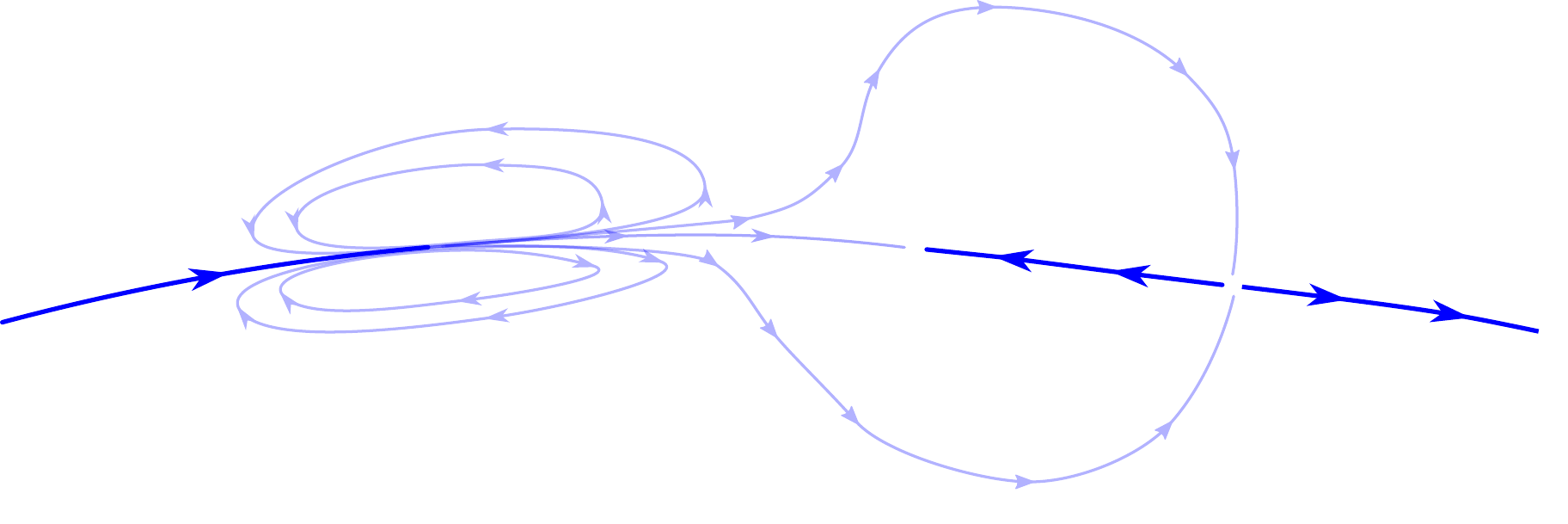}}%
    \put(0.80515602,0.15245142){\color[rgb]{0,0,0}\makebox(0,0)[lb]{\smash{$S$}}}%
    \put(0.59214989,0.19174533){\color[rgb]{0,0,0}\makebox(0,0)[lb]{\smash{$P$}}}%
    \put(0.25775727,0.19207291){\color[rgb]{0,0,0}\makebox(0,0)[lb]{\smash{$Q$}}}%
    \put(0,0){\includegraphics[width=\unitlength,page=2]{figures/graphSlit/flow.pdf}}%
    \put(0.73794227,0.16422972){\color[rgb]{0,0,0}\makebox(0,0)[lb]{\smash{$\scriptscriptstyle\beta$}}}%
    \put(0.7347107,0.13187055){\color[rgb]{0,0,0}\makebox(0,0)[lb]{\smash{$\scriptscriptstyle\beta'$}}}%
    \put(0.93244151,0.13637121){\color[rgb]{0,0,0}\makebox(0,0)[lb]{\smash{$\scriptscriptstyle\alpha$}}}%
    \put(0.92678775,0.10573006){\color[rgb]{0,0,0}\makebox(0,0)[lb]{\smash{$\scriptscriptstyle\alpha'$}}}%
    \put(0,0){\includegraphics[width=\unitlength,page=3]{figures/graphSlit/flow.pdf}}%
  \end{picture}%
\endgroup%

  \quad
  \raisebox{15px}{
    \def\svgwidth{.23\columnwidth}
\begingroup%
  \makeatletter%
  \providecommand\color[2][]{%
    \errmessage{(Inkscape) Color is used for the text in Inkscape, but the package 'color.sty' is not loaded}%
    \renewcommand\color[2][]{}%
  }%
  \providecommand\transparent[1]{%
    \errmessage{(Inkscape) Transparency is used (non-zero) for the text in Inkscape, but the package 'transparent.sty' is not loaded}%
    \renewcommand\transparent[1]{}%
  }%
  \providecommand\rotatebox[2]{#2}%
  \ifx\svgwidth\undefined%
    \setlength{\unitlength}{337.79781958bp}%
    \ifx\svgscale\undefined%
      \relax%
    \else%
      \setlength{\unitlength}{\unitlength * \real{\svgscale}}%
    \fi%
  \else%
    \setlength{\unitlength}{\svgwidth}%
  \fi%
  \global\let\svgwidth\undefined%
  \global\let\svgscale\undefined%
  \makeatother%
  \begin{picture}(1,0.68017274)%
    \put(0,0){\includegraphics[width=\unitlength,page=1]{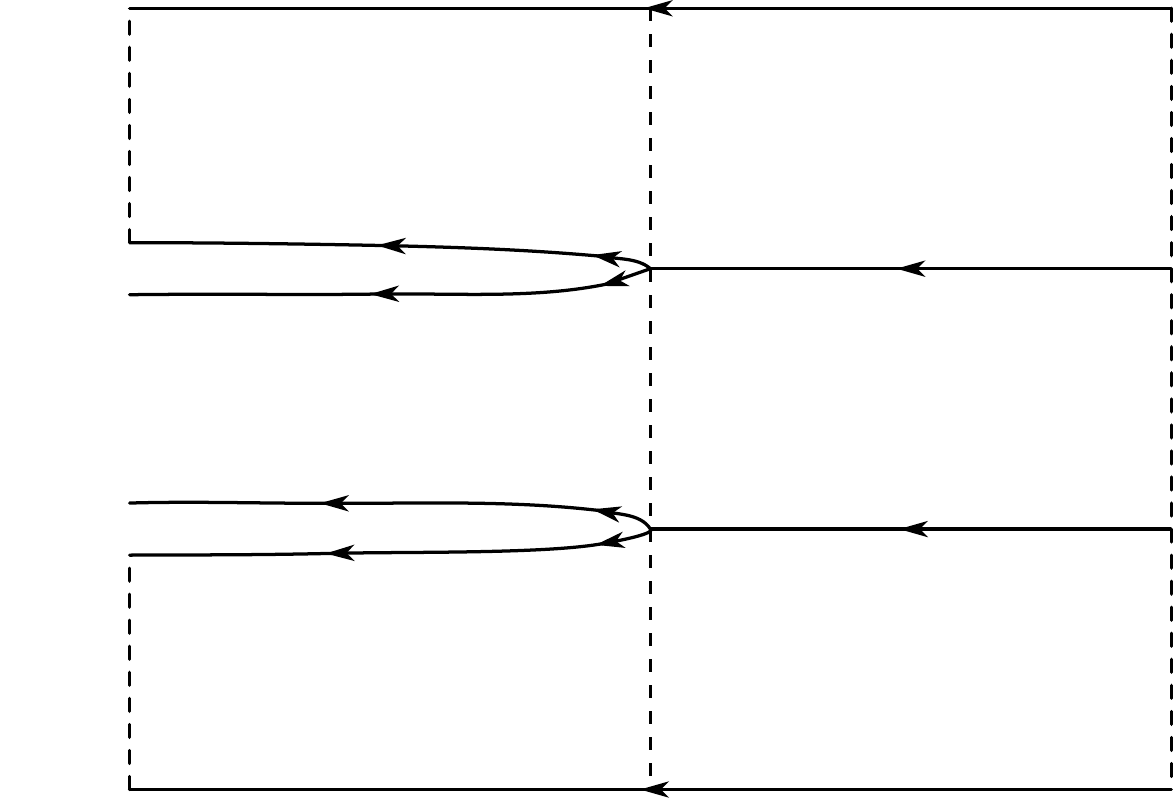}}%
    \put(0.56359387,0.46757912){\color[rgb]{0,0,0}\makebox(0,0)[lb]{\smash{$S$}}}%
    \put(0.56347245,0.24554144){\color[rgb]{0,0,0}\makebox(0,0)[lb]{\smash{$S$}}}%
    \put(-0.00550729,0.30718982){\color[rgb]{0,0,0}\makebox(0,0)[lb]{\smash{$P$}}}%
    \put(0.26009999,0.37259856){\color[rgb]{0,0,0}\makebox(0,0)[lb]{\smash{$\scriptscriptstyle\beta$}}}%
    \put(0.26354054,0.27306988){\color[rgb]{0,0,0}\makebox(0,0)[lb]{\smash{$\scriptscriptstyle\beta'$}}}%
    \put(0.26065947,0.48419315){\color[rgb]{0,0,0}\makebox(0,0)[lb]{\smash{$\scriptscriptstyle\alpha$}}}%
    \put(0.26307797,0.14084094){\color[rgb]{0,0,0}\makebox(0,0)[lb]{\smash{$\scriptscriptstyle\alpha'$}}}%
    \put(0,0){\includegraphics[width=\unitlength,page=2]{figures/graphSlit/slit.pdf}}%
  \end{picture}%
\endgroup%
}
  \caption{The gradient field and the critical graph of a potential function
    with one dipole, one sink, and one critical point; and the image of
    the complement $F\setminus\cK$ under $w$.}\label{fig:graphSlit}
\end{tFigure}

\subsection{Parallel slit domains}
\label{subsec:Par}

In this subsection, we construct a (relative) multisimplicial complex
$\frP_{g,n}^m$, which parametrises all possible slitted planes as in
\tref{Figure}{fig:graphSlit}, together with the extra information how to
‘reglue’ the slits, and which will turn out to be homeomorphic to
$\frH_{g,n}^m$. We will be very brief on the combinatorics; more details can be
found in \cite{Abhau-CFB-Ehrenfried} for the case $n=1$ and in
\cite[§\,2.3]{Boes-Hermann} for the general case.

\begin{nota}
  Let $X$ be a finite set.
  \begin{enumerate}
  \item Each permutation $\sigma$ on $X$ can be decomposed into its cycles,
    which are denoted by $\cyc{x_1,\dotsc,x_r}$ and read from left to right,
    i.e.\ $\sigma(x_i)=x_{i+1}$ for $1\le i<r$ and $\sigma(x_r)=x_1$. Let
    $C(\sigma)$ be the number of cycles of $\sigma$, including fixed points.
  \item Let $N(\sigma)$ be the word length norm of $\sigma$ with respect to the
    generating set of all transpositions. As an $r$-cycle has norm $r-1$, we get
    $N(\sigma) = \# X - C(\sigma)$.
  \end{enumerate}
\end{nota}

\begin{constr}
  For $p_1,\ldots, p_n\ge 0$, we consider the \emph{tableau}
  \[[p_1,\ldots,p_n]\coloneqq \{(i,j);\,1\le i\le n\text{ and }0\le j\le p_i\}\]
  as an index set. Let $\frS[p_1,\dotsc,p_n]$ be the group of \emph{all}
  permutations of $[p_1,\dotsc,p_n]$; it is, up to reindexing, isomorphic to the
  symmetric group on $n+p_1+\dotsb+p_n$ elements. For each
  $(i,j)\in [p_1,\dotsc,p_n]$ we have a \emph{deletion map}, which is a function of sets,
  \[D_j^i \colon \fS[p_1,\dotsc,p_n] \to \fS[p_1,\dotsc,p_i-1,\dotsc,p_n],\]
  where $\smash{D_{\smash j}^i(\sigma)}$ is obtained by skipping the symbol $(i,j)$ from the
  cycle representation of $\sigma$ and shifting down all $j'$ in $(i,j')$ with
  $j<j'$ by one. These functions $D_{\smash j}^i$ satisfy the $n$-semisimplicial
  identities, but they are no group homomorphisms.

  Without going into detail, let us point out that for $n=1$, the collection of
  deletion maps $D_j\colon \frS[p]\to \frS[p-1]$, together with adequate
  degeneracies, is closely related to the notion of a
  crossed simplicial group in the sense of Krasauskas \cite{Krasauskas}
  and Fiedorowicz–Loday \cite{Fiedorowicz-Loday}, \cite[§\,6]{Loday}.
\end{constr}

\begin{defi}
  Let $g, m \ge 0$ and $n \ge 1$. We put $h\coloneqq 2g-2+2n+m$ as before and define a
  $(1+n)$-semisimplicial complex $P_{g,n}^m$ as follows: its
  $(q, p_1,\ldots, p_n)$-simplices are given by tuples
  $\Sigma = (\sigma_q :\dotsc: \sigma_0)$ with
  $\sigma_k \in \fS[p_1,\dotsc,p_n]$ such that
  \begin{enumerate}
  \item $\sum_{k=1}^q N(\sigma_{\smash k\vphantom{a}}^{\vphantom{-1}}
    \sigma_{\smash{k-1}\vphantom{a}}^{-1})\le h$ and
  \item $C(\sigma_q) \le m+n$.
  \end{enumerate}
  (The notation for the tuple is chosen to be reminiscent of the homogeneous
  notation for the bar complex of a group.) We denote the $0$\textsuperscript{th}
  face operator by $\partial'$, and for $1\le i\le n$, the $i$\textsuperscript{th} face
  operator by $\partial^i$: this means we have $\partial'_0,\dotsc,\partial'_q$ and
  $\partial^i_0,\dotsc,\partial^i_{p_i}$ for each $i$. For
  $\Sigma=(\sigma_q:\dotsc:\sigma_0)$, these face operators are given by
  \begin{align*}
    \partial_k'\Sigma  &\coloneqq (\sigma_q : \dotsc : \widehat{\sigma_k} : \dotsc : \sigma_0),\\
    \partial^i_j\Sigma &\coloneqq (D_{\smash j}^i(\sigma_q) : \dotsc : D_{\smash j}^i(\sigma_0))
  \end{align*}  
  A cell $\Sigma = (\sigma_q:\dotsc:\sigma_0)$ of the complex $P_{g,n}^m$ is
  called \emph{non-degenerate} if it satisfies the following properties 
  (see \cite[§\,4.3]{Abhau-CFB-Ehrenfried} and \cite[Def.\,2.3.3]{Boes-Hermann}):
  \begin{enumerate}[itemsep=1pt]
  \item[\textsc{s1.}]
    $\sigma_0= \prod_{i=1}^{n} \cyc{(i,0),\dotsc,(i,p_i)}$ for
    $1 \le i \le n$.
  \item[\textsc{s2.}] 
    $\sigma_k(i,p_i)=(i,0)$ for each $0\le k\le q$ and $1 \le i \le n$.
  \item[\textsc{s3.}]  No cycle of any $\sigma_k$ contains two different symbols
    of the form $(i,0)$.
  \item[\textsc{s4.}] 
    $C(\sigma_q)=n+m$.
  \item[\textsc{s5.}]  
  $\sum_{k=1}^q N(\sigma_{\smash k\vphantom{a}}^{\vphantom{-1}}
    \sigma_{\smash{k-1}\vphantom{a}}^{-1}) = h$.
  \item[\textsc{s6.}] 
  We have $\sigma_k\ne \sigma_{k-1}$, and there is no
    $(i,j)$ such that $\sigma_k(i,j)=(i,j+1)$ for \emph{all} $k$.
  \item[\textsc{s7.}]  Under the equivalence relation on the set
    $\{ 1,\dotsc,n\}$, generated by $i \sim i'$ if there are $k,j,j'$ with
    $\sigma_k(i,j)=(i',j')$, all elements are equivalent.
  \end{enumerate}
  The collection $P^{\prime \, m}_{g,n}\subseteq P_{g,n}^m$ of \emph{degenerate}
  cells forms a subcomplex, and we define the space of \emph{parallel slit
    domains} as the complement
  \[\frP_{g,n}^m \coloneqq |P_{g,n}^m|\setminus |P^{\prime \, m}_{g,n}|.\]
  For the top-dimensional non-degenerate cells we have
  $q=h$, and all $\sigma_{\smash k}^{\vphantom{-1}}\sigma_{\smash{k-1}}^{-1}$ are
  disjoint transpositions, implying that $p_1 + p_2 + \dotsb + p_n = 2h$; thus,
  the dimension of $\frP_{g,n}^m$ is $3h$.
\end{defi}

\begin{constr}\label{constr:gluingRec}
  Using ordered simplex coordinates, each slit domain in $\frP_{g,n}^m$ is
  represented by a tuple $(\Sigma;a,b_1,\dotsc,b_n)$ where:\vspace*{-4px}
  \begin{enumerate}[itemsep=1pt]
  \item $\Sigma$ is a non-degenerate cell,\vspace*{-2px}
  \item $a$ is a $q$-tuple of real numbers $-\infty < a_q < \dotsb < a_1 < \infty$, and\vspace*{-2px}
  \item each $b_i$ is a $p_i$-tuple of real numbers
    $-\infty < b_{i,1} < \dotsb < b_{i,p_i} < \infty$.\vspace*{-2px}
  \end{enumerate}
  This data is used in a ‘gluing recipe’ to obtain an element
  $[\caF,u,w]\in \frH_{g,n}^m$ as in \cite[§\,5.2]{CFB-90I}, see
  \tref{Figure}{fig:SlitPicture-as-Cell}: we subdivide $n$ complex planes
  $\C \times \{1,\dotsc, n\}$ into rectangles\looseness-1
  \[R_{k,i,j} \coloneqq [a_{k+1}, a_k]\times[b_{i,j}, b_{i,j+1}] \times \set{i}\]
  for $0\le k\le q$, $1\le i\le n$ and $0\le j\le p_i$, where we put
  $a_0 \coloneqq b_{i, p_{i}+1} \coloneqq \infty$ and
  $a_{q+1} \coloneqq b_{i,0} \coloneqq -\infty$, skipping boundary at
  $\pm \infty$. Then we glue the left edge of $R_{k,i,j}$ to the right edge of
  $R_{k+1,i,j}$, and the top edge of $R_{k,i,j}$ to the bottom edge of
  $R_{k,\sigma_k(i,j)}$. (These identified top/bottom edges are called a
  \emph{slit}, if $\sigma_{k'}(i,j) \ne (i,j+1)$ for some $k'\ge k$.)
  
  The resulting surface has $n+m$ ends. For each $1 \le i \le n$, there is one
  end that corresponds to the ‘far right’ of the $i$\textsuperscript{th}
  plane; we close such an end by adding a point $Q_i$. The remaining ends are
  closed by adding points $P_1, \dotsc, P_m$. Then property
  \textsc{s7} ensures that $F$ is connected, while property
  \textsc{s5} ensures that $F$ has the correct Euler characteristic.

  The complex structure on $F$ is declared by the following atlas: each point in
  the interior of a rectangle has this rectangle as a coordinate neighbourhood;
  each point in the interior of an \emph{edge} of a rectangle needs the two
  adjacent half-rectangles as a coordinate neighbourhood; each point at a
  \emph{corner} of a rectangle uses the $4\cdot (l+1)$ quarter-rectangles
  attached to it, parametrised in such a way to have the point as a branching
  point of index $l$ when projecting down to a usual union of $4$
  quarter-rectangles; each $P_j$ has several triangles attached: these are
  parametrised by using the logarithm function to form a coordinate
  neighbourhood; and each $Q_i$ has triangles and bigons attached: these are
  parametrised using the logarithm function and the inversion
  $z\mapsto \frac1z$
  to form a coordinate neighbourhood as in
  \cite[§\,4.6]{CFB-90I}: this local parameter $z$ around $Q_i$ also determines
  the non-zero tangent vector $X_i$ via $T_{Q_i}z(X_i)=\partial_x$.\looseness-1
  
  The harmonic potential $u \colon F \to \overline \R$ is defined to be the
  projection of the slit domain to the $x$-axis. Then the critical graph on $F$
  is the union of all slits, and on its complement, the holomorphic map $w$ is just
  the reidentification of $F \setminus \caK$ with the open subset of our $n$
  complex planes with slits removed.
\end{constr}

\begin{tFigure}
  \centering
  \begin{tikzpicture}[scale=2.8]
  \draw[grey!70,very thin] (.75,1) -- (.75,0);
  \draw[grey!70,very thin] (.75,2.2) -- (.75,1.2);
  \draw[grey!70,very thin] (.5,.25) -- (1,.25);
  \draw[grey!70,very thin] (.5,.375) -- (1,.375);
  \draw[grey!70,very thin] (.375,.75) -- (1,.75);
  \draw[grey!70,very thin] (.5,1) -- (.5,0);
  \draw[grey!70,very thin] (.375,1) -- (.375,.382);
  \draw[grey!70,very thin] (.375,.367) -- (.375,.258);
  \draw[grey!70,very thin] (.375,0) -- (.375,.242);
  \draw[grey!70,very thin] (.75,1.45) -- (1,1.45);
  \draw[grey!70,very thin] (.75,1.95) -- (1,1.95);
  \draw[grey!70,very thin] (.375,1.7) -- (1,1.7);
  \draw[grey!70,very thin] (.5,2.2) -- (.5,1.956);
  \draw[grey!70,very thin] (.5,1.944) -- (.5,1.456);
  \draw[grey!70,very thin] (.5,1.444) -- (.5,1.2);
  \draw[grey!70,very thin] (.375,2.2) -- (.375,1.96);
  \draw[grey!70,very thin] (.375,1.94) -- (.375,1.46);
  \draw[grey!70,very thin] (.375,1.44) -- (.375,1.2);
  \draw[dred] (0,.23) -- (0,0) -- (1,0) -- (1,1) -- (0,1) -- (0,.77);
  \draw[dred] (0,.73) -- (0,.395);
  \draw[thin,dgrey] (0,.355) -- (0,.27);
  \draw[semithick,dgreen] (0,.77) -- (.375,.75);
  \draw[semithick,dgreen!50] (.375,.75) -- (0,.73);
  \draw[semithick,dblue!50] (0,.395) -- (.5,.375);
  \draw[semithick,dblue] (.5,.375) -- (0,.355);
  \draw[semithick,dblue] (0,.27) -- (.5,.25);
  \draw[semithick,dblue!50] (.5,.25) -- (0,.23);
  \draw[dyellow] (0,1.43) -- (0,1.2) -- (1,1.2) -- (1,2.2) -- (0,2.2) -- (0,1.97);
  \draw[dred] (0,1.93) -- (0,1.72);
  \draw[dred] (0,1.68) -- (0,1.47);
  \draw[semithick,dlila!50] (0,1.97) -- (.75,1.95);
  \draw[semithick,dlila] (.75,1.95) -- (0,1.93);
  \draw[semithick,dlila] (0,1.47) -- (.75,1.45);
  \draw[semithick,dlila!50] (.75,1.45) -- (0,1.43);
  \draw[semithick,dgreen!50] (0,1.72) -- (.375,1.7);
  \draw[semithick,dgreen] (.375,1.7) -- (0,1.68);
  \node at (.375,.75) {\tiny $\bullet$};
  \node at (.5,.375) {\tiny $\bullet$};
  \node at (.5,.25) {\tiny $\bullet$};
  \node at (.75,1.95) {\tiny $\bullet$};
  \node at (.75,1.45) {\tiny $\bullet$};
  \node at (.375,1.7) {\tiny $\bullet$};
  \node[dred] at (.94,.08) {\scriptsize $1$};
  \node[dyellow] at (.94,1.28) {\scriptsize $2$};
  \node[dgrey] at (.1875,2.28) {\tiny $3$};
  \node[dgrey] at (.4375,2.28) {\tiny $2$};
  \node[dgrey] at (.625,2.28) {\tiny $1$};
  \node[dgrey] at (.875,2.28) {\tiny $0$};
  \node[dgrey] at (1.08,.122) {\tiny $0_1$};
  \node[dgrey] at (1.08,.3095) {\tiny $1_1$};
  \node[dgrey] at (1.08,.5595) {\tiny $2_1$};
  \node[dgrey] at (1.08,.872) {\tiny $3_1$};
  \node[dgrey] at (1.08,1.322) {\tiny $0_2$};
  \node[dgrey] at (1.08,1.572) {\tiny $1_2$};
  \node[dgrey] at (1.08,1.822) {\tiny $2_2$};
  \node[dgrey] at (1.08,2.072) {\tiny $3_2$};
\end{tikzpicture}
  \caption{The visualisation of the slit domain $(\Sigma,a,b)$ with coordinates
    $(a,b_1,b_2)$ in the multisimplex $\Delta^3\times \Delta^3\times \Delta^3$,
    e.g.\ $a=(a_3<a_2<a_1)$, and with
    $\Sigma=(\sigma_3 \!:\! \sigma_2 \!:\! \sigma_1 \!:\! \sigma_0)$ where for example
    \mbox{$\sigma_3=\langle 0_1,2_1,2_2,1_2,3_1 \rangle \langle 1_1 \rangle
      \langle 0_2,3_2 \rangle$}, abbreviating $i_j\coloneqq (i,j)$. Here $g=0$,
    $n=2$, and $m=1$.}\label{fig:SlitPicture-as-Cell}
\end{tFigure}

\subsection{Hilbert uniformisation}
\label{subsec:hilbert}

In the previous subsection, we have constructed a map
$\frP_{g,n}^m\to \frH_{g,n}^m$ by regluing slits according to the combinatorics
of the cell. This construction has a very geometric inverse, which we call
\emph{Hilbert uniformisation}: given an element $[\caF,u,w]\in \frH_{g,n}^m$, the
images of the basins $F_i$ under $w$ are slitted complex planes, so the
collection $\coprod_i w(F_i)$ shows slits on $\C\times\{1,\dotsc,n\}$ as in
\tref{Figure}{fig:graphSlit}. If we additionally track how we have
dissected the surface along the edges of the critical graph, then we receive
exactly the combinatorial gluing information needed to describe an element
$(\Sigma;a,b_1,\dotsc,b_n)$ in $\frP_{g,n}^m$.\looseness-1

The Hilbert uniformisation and the gluing construction are indeed continuous and
inverses of each other,\footnote{In the exceptional case $\frM_{0,1}$, we obtain
  (\ref{eq:PL}) by noting that both $\frM_{0,1}$ and $\frP_{0,1}$ are just a
  point.}  see \cite[Thm.\,5.5.1]{CFB-90I}, and thus $\frP_{g,n}^m$ and
$\frH_{g,n}^m$ are homeomorphic. In particular,
$(|P_{g,n}^m|,|P_{g,n}^{\prime\, m}|)$ is a pair such that the
quotient is compact and the complement
$|P_{g,n}^m|\setminus |P_{g,n}^{\prime\, m}|$ is a $3h$-dimensional open
manifold. By Poincaré–Lefschetz duality, we obtain an isomorphism\looseness-1
\begin{align}\label{eq:PL}
  H_\bullet(\frM_{g,n}^m)\cong H_\bullet(\frP_{g,n}^m)
  \cong H^{3h-\bullet}(P_{g,n}^m,P_{g,n}^{\prime\, m};\scO),
\end{align}
where $\scO$ is the orientation system, which has a simplicial description as
carried out in \cite{Mueller}. We note that $\smash{\frP_{g,n}^m}$ is orientable
if the number $m$ of punctures is $0$ or $1$ (otherwise we have to pass to
$\smash{\widetilde \frP_{g,n}^m}$ where the punctures are ordered), so in these
cases, the orientation system is constant.

For $n=1$, the right side of (\ref{eq:PL}) is the homology of a finite double
complex with $2h$ columns and $h$ rows and can, in principle, be computed by
computer-aided methods. However, the number of cells grows very quickly: for
example, the complex $P_{2,1}$ already has $17{,}136$ non-degenerate cells
\cite[p.\,11]{Abhau-CFB-Ehrenfried}. In his thesis \cite{Visy}, Visy used
intricate combinatorics of the symmetric group $\frS_p$ to show that the
homology of the $p$\textsuperscript{th} column of the double complex is
concentrated in the top degree. Therefore, the first page of the spectral sequence
associated with the column-filtration of the double complex is concentrated in a
single row and hence collapses on the second page. This simplifies the chain
complex drastically and made many of the cited calculations possible at all. 
At the same time, Visy’s result gives a new proof that the homological dimension of the
moduli space $\frM_{g,1}$ is $4g-3$: this was first observed by Harer
\cite{Harer-86}.

\subsection{A computational example}
\label{subsec:compEx}

In order to illustrate how the above semisimplicial model can be used for homology
calculations, we consider the example of $\frM_{1,1}$ as in
\cite[§\,6]{Abhau-CFB-Ehrenfried}: the relative complex
$(P_{\smash{1,1}}^{\vphantom\prime},P_{\smash{1,1}}')$ has eight non-degenerate
cells and the incidence graph that underlies the simplicial double complex looks
as in \tref{Figure}{fig:P11Cells}. Taking care of all the signs involved, the
total cochain complex of $(P_{\smash{1,1}}^{\vphantom\prime},P_{\smash{1,1}}')$
is of the form
\[\begin{tikzcd}[ampersand replacement=\&,column sep=6em,row sep=-.3em]
    \Z^2\ar[<-]{r}{\begin{psmallmatrix*}[r]0 & -1 & 0 & -1\\0 & -1 & 0 & -1\end{psmallmatrix*}} \&
    \Z^4\ar[<-]{r}{\begin{psmallmatrix*}[r]-1 & 1\\0& 1\\1 & 1\\0&-1\end{psmallmatrix*}} \& \Z^2.\\
    {^{[6]}} \& {^{[5]}} \& {^{[4]}}
  \end{tikzcd}\]
Therefore, we see that
$H_\bullet(\frM_{1,1})\cong
H^{6-\bullet}(P_{\smash{1,1}}^{\vphantom\prime},P_{\smash{1,1}}') \cong
(\Z,\Z,0,\dotsc)$, which is not very surprising, as we are calculating the group
homology of $\Gamma_{1,1}\cong \mathfrak{B\kern-.7px r}_3$.

\begin{tFigure}
  \centering
  \begin{tikzpicture}[xscale=3.2,yscale=1.3]
  \node at (1.5,2) {\tiny $\cyc{0{,}3{,}2{,}1{,}4}\!:\!\cyc{0{,}1{,}4}\cyc{2{,}3}\!:\!\cyc{0{,}1{,}2{,}3{,}4}$};
  \node at (2.8,2) {\tiny $\cyc{0{,}3{,}2{,}1{,}4}\!:\!\cyc{0{,}3{,}4}\cyc{1{,}2}\!:\!\cyc{0{,}1{,}2{,}3{,}4}$};
  \node[black!30!red]  at (1,1)     {\tiny $\cyc{0{,}2{,}1{,}3}\!:\!\cyc{0{,}1{,}3}\cyc{2}\!:\!\cyc{0{,}1{,}2{,}3}$};
  \node[black!30!red]  at (2,1)     {\tiny $\cyc{0{,}2{,}1{,}3}\!:\!\cyc{0{,}3}\cyc{1{,}2}\!:\!\cyc{0{,}1{,}2{,}3}$};
  \node[black!30!red]  at (3,1)     {\tiny $\cyc{0{,}2{,}1{,}3}\!:\!\cyc{0{,}2{,}3}\cyc{1}\!:\!\cyc{0{,}1{,}2{,}3}$};
  \node[black!30!blue] at (4,1)     {\tiny $\cyc{0{,}3{,}2{,}1{,}4}\!:\!\cyc{0{,}1{,}2{,}3{,}4}$};
  \node[black!40!green] at (2,0)    {\tiny $\cyc{0{,}1{,}2}\!:\!\cyc{0{,}2}\cyc{1}\!:\!\cyc{0{,}1{,}2}$};
  \node[black!50!yellow] at (3.5,0) {\tiny $\cyc{0{,}2{,}1{,}3}\!:\!\cyc{0{,}1{,}2{,}3}$};
  \draw (1.3,1.85) -- (.975,1.15);
  \draw (1.35,1.85) -- (1.025,1.15);
  \draw (1.7,1.85) -- (1.95,1.15);
  \draw (2.6,1.85) -- (2.05,1.15);
  \draw (2.95,1.85) -- (2.975,1.15);
  \draw (3,1.85) -- (3.025,1.15);
  \draw[white,line width=1mm] (1.75,1.85) -- (3.975,1.15);
  \draw (1.75,1.85) -- (3.975,1.15);
  \draw (3.1,1.85) -- (4.025,1.15);
  \draw (2,.85) -- (2,.15);
  \draw (3,.85) -- (2.05,.15);
  \draw[white,line width=1mm] (2.05,.85) -- (3.44,.15);
  \draw[white,line width=1mm] (1.05,.85) -- (3.41,.15);
  \draw (1,.85) -- (1.95,.15);
  \draw (3.05,.85) -- (3.47,.15);
  \draw (2.05,.85) -- (3.44,.15);
  \draw (1.05,.85) -- (3.41,.15);
  \draw (3.53,.15) -- (3.97,.85);
  \draw (3.56,.15) -- (4,.85);
  \draw (3.59,.15) -- (4.03,.85);
\end{tikzpicture}
  \caption{The incidence graph for the double complex of
    $(P^{\vphantom\prime}_{\smash{1,1}},P^\prime_{\smash{1,1}})$. The black
    cells are of bidimension $(2,4)$, the red ones of bi-dimension $(2,3)$, the
    blue one of $(1,4)$, the green one of $(2,2)$, and the yellow one of
    $(1,3)$. Each edge stands for one direct face relation.}\label{fig:P11Cells}
\end{tFigure}

Let us finally give a geometric description of $\frP_{1,1}$; for details see
\cite{Dahlmann}: first of all, we have a proper action of $\R_{>0} \times \R^2$ on
$|P_{1,1}|$ by translating and scaling the slit picture, which is free on the
non-degenerate part $\frP_{1,1}$.  The $3$-dimensional quotient is a union of
two prisms $\Delta^2 \times \Delta^1$ with two triangles identified to a middle
triangle. The remaining top and bottom triangles are degenerate and the six
squares are identified in three pairs.  Cutting out the degenerate part gives a
$3$-sphere with a trefoil knot removed.

\section{Homology operations}
\label{sec:operations}
In this section we describe several operations on the collection of homology
groups $H_\bullet(\frM_{g,n}^m)$. We start by recalling the $E_2$-structure on the union
$\coprod_{g,m}\frM_{g,1}^m$.%

\subsection{The \texorpdfstring{$E_2$}{E₂}-algebra structure
  on \texorpdfstring{$\Par_{\bullet,1}^\bullet$}{P.₁}}
\label{subsec:c2}

It is a well-known result \cite{Miller,CFB-90II} that the collection
$\coprod_g\frM_{g,1}$ admits an action of the little $2$-cubes operad $\scC_2$ by
sewing of surfaces. The same works with punctures, so we obtain a
$\scC_2$-algebra $\coprod_{g,m}\frM_{g,1}^m$, with $\coprod_g\frM_{g,1}$ as a
subalgebra.

This operadic action can be expressed in terms of parallel slit domains (this
was actually the model used in \cite[§\,3]{CFB-90II} to establish the
$E_2$-structure), see \tref{Figure}{fig:E2-on-P}: given a $k$-ary operation in
$\scC_2$, i.e.\ an ordered collection of $k$ disjoint squares
$B_1,\dotsc,B_k$ in the plane, and, for each $1\le i\le k$, a parallel slit
configuration $S_i=(\Sigma_i,a_i,b_i)\in \Par_{\smash{g_i,1}}^{m_i}$, then we can insert
$S_i$ into $B_i$. For this purpose, we regard a slit domain as being supported
on a square; then the insertion starts with the
right-most square and proceeds to the left: here it may happen that we have
to ‘weave’ some squares through slits leaving another one: if the squares
$B_i$ and $B_j$ have overlapping projections to the $y$-axis, then their
projections to the $x$-axis are disjoint, so we can assume that $B_j$ is on the
left side of $B_i$. The slits leaving $B_i$ then cut $B_j$ into several
horizontal sub-rectangles which are reglued in accordance to the left-most
permutation $\sigma_{q_i}$ of $\Sigma_i$, and this is the area into which we
insert $S_j$, see \tref{Figure}{fig:E2-on-P}. After applying the gluing
recipe from \tref{Construction}{constr:gluingRec}, the insertion can be
visualised as in \tref{Figure}{fig:E2-on-M}.\looseness-1

\begin{figure}[h]
  \centering
  \begin{tikzpicture}[scale=1.2]
  \draw (0,0) rectangle (2,2);
  \draw[semithick] (.2,.6) rectangle (.8,1.2);
  \draw[semithick] (1,.8) rectangle (1.8,1.6);
  \node at (.5,.9) {\small $1$};
  \node at (1.4,1.2) {\small $2$};
  \draw[black!30!red,fill=black!30!red!7] (2.5,.2) rectangle (4.1,1.8);
  \draw[black!40!green,fill=black!40!green!7] (4.4,.2) rectangle (6,1.8);
  \draw (7,0) rectangle (9,2);
  \node at (2.3,1) {$\left(\vbox to 13mm{}\right.$};
  \node at (6.2,1) {$\left.\vbox to 13mm{}\right)$};
  \node at (6.6,1) {$=$};
  \draw[thin,black!30!red!40,fill=black!30!red!7] (7.2,.96) -- (7.2,.6) -- (7.8,.6) -- (7.8,.96);
  \fill[black!30!red!7] (7.2,1.28) rectangle (7.8,1.44);
  \draw[thin,black!30!red!40] (7.2,1.28) -- (7.2,1.44);
  \draw[thin,black!30!red!40] (7.8,1.28) -- (7.8,1.44);
  \draw[thin,black!30!red!40,fill=black!30!red!7] (7.2,1.12) -- (7.2,1.2) -- (7.8,1.2) -- (7.8,1.12);
  \draw[thin,black!40!green!40,fill=black!40!green!7] (8,.8) rectangle (8.8,1.6);
  \draw[black!30!red,semithick] (3,.52) -- (2.5,.52);
  \draw[black!30!red,semithick] (3.6,.84) -- (2.5,.84);
  \draw[black!30!red,semithick] (3,1.032) -- (2.5,1.032);
  \draw[black!30!red,semithick] (3.6,1.352) -- (2.5,1.352);
  \draw[black!40!green,semithick,shift={(1.9,0)}] (3.6,.52) -- (2.5,.52);
  \draw[black!40!green,semithick,shift={(1.9,0)}] (3,.84) -- (2.5,.84);
  \draw[black!40!green,semithick,shift={(1.9,0)}] (3.6,1.16) -- (2.5,1.16);
  \draw[black!40!green,semithick,shift={(1.9,0)}] (3,1.48) -- (2.5,1.48);
  \draw[black!40!green,semithick] (8.55,.96) -- (7,.96);
  \draw[black!40!green,semithick] (8.25,1.12) -- (7,1.12);
  \draw[black!40!green,semithick] (8.55,1.28) -- (7,1.28);
  \draw[black!40!green,semithick] (8.25,1.44) -- (7,1.44);
  \draw[black!30!red,semithick] (7.3875,.72) -- (7,.72);
  \draw[black!30!red,semithick] (7.6125,.84) -- (7,.84);
  \draw[black!30!red,semithick] (7.3875,.92) -- (7,.92);
  \draw[black!30!red,semithick] (7.6125,1.36) -- (7,1.36);
\end{tikzpicture}
  \caption{An instance of
    $\scC_2(2)\times \frP_{1,1}\times \frP_{1,1}\to \frP_{2,1}$}\label{fig:E2-on-P}
\end{figure}

\begin{figure}[h]
  \centering
  \begin{tikzpicture}[scale=1.2]
  \draw (0,0) rectangle (2,2);
  \draw[semithick] (.2,.6) rectangle (.8,1.2);
  \draw[semithick] (1,.8) rectangle (1.8,1.6);
  \node at (.5,.9) {\small $1$};
  \node at (1.4,1.2) {\small $2$};
  \draw (7,0) -- (8.4,0) -- (9,.6) -- (7.6,.6) -- (7,0);
  \draw[black!30!red!20]   (7.32,.18) -- (7.74,.18) -- (7.92,.36) -- (7.5,.36) -- (7.32,.18);
  \draw[black!40!green!20] (7.94,.24) -- (8.5,.24) -- (8.74,.48) -- (8.18,.48) -- (7.94,.24);
  \node at (2.3,1) {$\left(\vbox to 13mm{}\right.$};
  \node at (6.2,1) {$\left.\vbox to 13mm{}\right)$};
  \node at (6.6,1) {$=$};
  \draw[shift={(.1,0)},black!30!red] (2.5,.2) -- (3.7,.2) -- (4.1,.6) -- (2.9,.6) -- (2.5,.2);
  \draw[shift={(.1,0)},black!30!red,fill=white,fill opacity=.9] (2.9,.4) to[out=60,in=-90]
  (2.8,1.4) to[out=90,in=180] (3.3,1.8) to[out=0,in=90] (3.8,1.4) to[out=-90,in=120]
  (3.7,.4);
  \centerarc[shift={(.1,0)},black!30!red](3.3,1.18)(30:150:.3);
  \centerarc[shift={(.1,0)},black!30!red](3.3,1.5)(-27:-153:.23);
  \draw[black!40!green,shift={(1.9,0)}] (2.5,.2) -- (3.7,.2) -- (4.1,.6) -- (2.9,.6) -- (2.5,.2);
  \draw[black!40!green,fill=white,fill opacity=.9,shift={(1.9,0)}] (2.9,.4) to[out=60,in=-90]
  (2.8,1.1) to[out=90,in=180] (3.3,1.5) to[out=0,in=90] (3.8,1.1) to[out=-90,in=120]
  (3.7,.4);
  \centerarc[black!40!green](5.2,.88)(30:150:.3);
  \centerarc[black!40!green](5.2,1.2)(-27:-153:.23);
  \draw[black!30!red,fill=white,fill opacity=.9] (7.46,.27) to[out=60,in=-90]
  (7.41,.645) to[out=90,in=180] (7.62,.83) to[out=0,in=90] (7.83,.645) to[out=-90,in=120]
  (7.78,.27);
  \centerarc[black!30!red](7.62,.58)(30:150:.1125);
  \centerarc[black!30!red](7.62,.7)(-27:-153:.08625);
  \draw[black!40!green,fill=white,fill opacity=.9] (8.14,.36) to[out=60,in=-90]
  (8.09,.76) to[out=90,in=180] (8.34,.96) to[out=0,in=90] (8.59,.76) to[out=-90,in=120]
  (8.54,.36);
  \centerarc[black!40!green](8.34,.65)(30:150:.15);
  \centerarc[black!40!green](8.34,.8)(-27:-153:.115);
\end{tikzpicture}
  \caption{The $E_2$-action after the gluing construction}\label{fig:E2-on-M}
\end{figure}
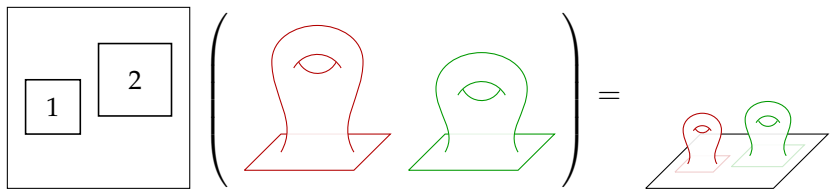

There is a subtlety if some of the $m_i$ is strictly positive, involving a
non-trivial rescaling of the area on the left side of an implanted slit picture;
this has been addressed in \cite[§\,5.2.4]{Kranhold} and does not affect the
general argument.  Thus, we obtain an action of $\scC_2$ on the family
$\Par_{\bullet,1}^\bullet$, which is \emph{bi-graded}, i.e.\ the structure maps
are of the form
\[\lambda\colon\scC_2(k)\times\Par_{g_1,1}^{m_1}\times\dotsb\times\Par_{g_k,1}^{m_k}
  \longrightarrow \Par_{g_1+\dotsb+g_k, 1}^{m_1+\dotsb+m_k}.\]

\subsubsection{Operations for \texorpdfstring{$E_2$}{E₂}-algebras}

The above $\scC_2$-action on $\frP_{\bullet,1}^\bullet$ gives rise to a \emph{unit
  class} $1\in H_0(\frM_{0,1})$ and three (graded) homology
operations,\footnote{Our sign convention for the Browder bracket slightly
  differs from the one in \cite[§\,\textsc{iii}]{Cohen-Lada-May}.} which are
depicted in \tref{Figure}{fig:Pontrjagin-Browder-DL}:
\begin{enumerate}
\item The \emph{Pontrjagin product} is geometrically given by
  by combining two surfaces by a pair of
  pants, or, equivalently, joining two slit domains on a single layer. It
  is a graded-commutative, associative, and unital product
  \[-\cdot - \colon H_i(\fM_{g_1,1}^{m_1}) \otimes H_j(\fM_{g_2,1}^{m_1})
    \longrightarrow H_{i+j}(\fM_{g_1+g_2,1}^{m_1+m_2}).\]
\item The \emph{Browder bracket} of two homology classes $x$ and $y$ geometrically
  corresponds to a full twist of two boxes, ‘filled’ with
  $x$ and $y$, respectively. It is denoted by
  \[[-,-]\colon H_i(\fM_{g_1,1}^{m_1}) \otimes H_j(\fM_{g_2,1}^{m_1})
    \longrightarrow H_{i+j+1}(\fM_{g_1+g_2,1}^{m_1+m_2}).\]
\item If the homological degree of $x$ is even or if we work over $\bbF_2$, then
  we additionally have the \emph{Dyer–Lashof square} $Q(x)$, which geometrically
  corresponds a half-twist of two boxes, both filled with $x$. Thus, $Q$
  is a family of maps
  \[Q \colon H_i(\fM_{g,1}^m) \longrightarrow H_{2i+1}(\fM_{2g,1}^{2m}).\]
\end{enumerate}

\subsubsection{Relations for \texorpdfstring{$E_2$}{E₂}-algebras}
\label{subsubsec:relE2}

There are several ‘universal’ formulæ that involve these operations and which
hold for each $E_2$-algebra. They have been derived in
\cite[§\,\textsc{iii}]{Cohen-Lada-May}; we give a short summary of them
(let $|x|$ be the homological degree of $x$, $(-1)^x\coloneqq (-1)^{|x|}$, and $x'\coloneqq |x|+1$):
\begin{align*}
  x\cdot (y\cdot z) &= (x\cdot y)\cdot z
                      \tag{associativity}\\
  x\cdot y          &= (-1)^{xy}\cdot y\cdot x
                      \tag{graded commutativity}\\
  1\cdot x          &= x\cdot 1 = x\tag{unitality}\\
  {[x,y]}           &= (-1)^{xy}\cdot [y,x]
                      \tag{graded commutativity}\\
  [x,1]             &= [1,x] = 0
                      \tag{annihilation}\\
  0                 &=(-1)^{xz'}\cdot [x,[y,z]] + (-1)^{yx'}\cdot [y,[z,x]]
                      + (-1)^{zy'}\cdot [z,[x,y]]
                      \tag{Jacobi}\\
  {}[x,y\cdot z]    &= [x,y]\cdot z + (-1)^{x'y}\cdot y\cdot [x,z]
                      \tag{Leibniz}
\end{align*}
The above relations between the Pontrjagin product and the Browder bracket are often
summarised in the definition of a \emph{Poisson $2$-algebra} or a
\emph{Gerstenhaber algebra}.

If $x$ is an even class, then $Q(x)$ is also defined
integrally, and we have the additional relation $[x,x]=2\cdot Q(x)$.  If we work
over $\bbF_2$, then we additionally have the following relations involving the
Dyer–Lashof square:
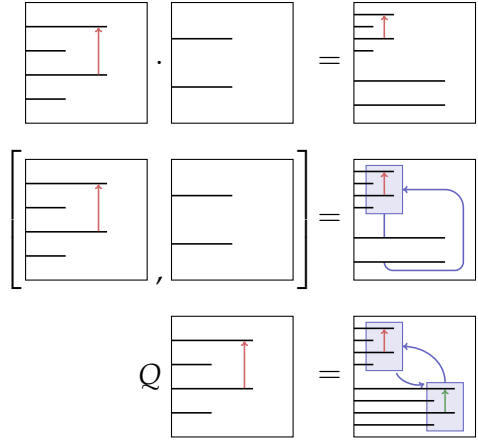
\begin{figure}
  \centering
  \begin{tikzpicture}[scale=1.6]
  \draw (-.1,0) rectangle (.9,1);
  \node at (1,.5) {$\cdot$};
  \draw (1.1,0) rectangle (2.1,1);
  \node at (2.4,.5) {$=$};
  \draw (2.6,0) rectangle (3.6,1);
  \draw[semithick] (.57,.8) -- (-.1,.8);
  \draw[semithick] (.23,.6) -- (-.1,.6);
  \draw[semithick] (.57,.4) -- (-.1,.4);
  \draw[semithick] (.23,.2) -- (-.1,.2);
  \draw[semithick] (1.6,.3) -- (1.1,.3);
  \draw[semithick] (1.6,.7) -- (1.1,.7);
  \draw[-{To[scale=.65,width=1.5mm]},black!30!red!60,semithick] (.5,.405) -- (.5,.795);
  \draw[semithick] (2.93,.9) -- (2.6,.9);
  \draw[semithick] (2.76,.8) -- (2.6,.8);
  \draw[semithick] (2.93,.7) -- (2.6,.7);
  \draw[semithick] (2.76,.6) -- (2.6,.6);
  \draw[-{To[scale=.65,width=1.5mm]},black!30!red!60,semithick] (2.85,.705) -- (2.85,.895);
  \draw[semithick] (3.35,.15) -- (2.6,.15);
  \draw[semithick] (3.35,.35) -- (2.6,.35);
  \node at (-.2,-.8) {$\left[\vbox to 10mm{}\right.$};
  \draw[shift={(0,-1.3)}] (-.1,0) rectangle (.9,1);
  \node at (1,-1.3) {$,$};
  \draw[shift={(0,-1.3)}] (1.1,0) rectangle (2.1,1);
  \node at (2.2,-.8) {$\left.\vbox to 10mm{}\right]$};
  \node at (2.4,-.8) {$=$};
  \draw[shift={(0,-1.3)}] (2.6,0) rectangle (3.6,1);
  \draw[semithick,shift={(-.1,-1.3)}] (.67,.8) -- (0,.8);
  \draw[semithick,shift={(-.1,-1.3)}] (.33,.6) -- (0,.6);
  \draw[semithick,shift={(-.1,-1.3)}] (.67,.4) -- (0,.4);
  \draw[semithick,shift={(-.1,-1.3)}] (.33,.2) -- (0,.2);
  \draw[semithick,shift={(0,-1.3)}] (1.6,.3) -- (1.1,.3);
  \draw[semithick,shift={(0,-1.3)}] (1.6,.7) -- (1.1,.7);
  \draw[thin,black!40!blue!60,fill=black!40!blue!10,shift={(0,-1.3)}] (2.7,.55) rectangle (3,.95);
  \draw[semithick,black!40!blue!60,shift={(0,-1.3)}] (2.85,.55) -- (2.85,.35);
  \draw[semithick,black!40!blue!60,-{To[scale=.65,width=1.5mm]},shift={(0,-1.3)}] (2.85,.15) to[out=-90,in=180] (2.92,.08) -- (3.43,.08) to[out=0,in=-90] (3.5,.15) -- (3.5,.6) to[out=90,in=0] (3.35,.75) -- (3,.75);
  \draw[-{To[scale=.65,width=1.5mm]},black!30!red!60,semithick,shift={(0,-1.3)}] (.5,.405) -- (.5,.795);
  \draw[semithick,shift={(0,-1.3)}] (2.93,.9) -- (2.6,.9);
  \draw[semithick,shift={(0,-1.3)}] (2.76,.8) -- (2.6,.8);
  \draw[semithick,shift={(0,-1.3)}] (2.93,.7) -- (2.6,.7);
  \draw[semithick,shift={(0,-1.3)}] (2.76,.6) -- (2.6,.6);
  \draw[-{To[scale=.65,width=1.5mm]},black!30!red!60,semithick,shift={(0,-1.3)}] (2.85,.705) -- (2.85,.895);
  \draw[semithick,shift={(0,-1.3)}] (3.35,.15) -- (2.6,.15);
  \draw[semithick,shift={(0,-1.3)}] (3.35,.35) -- (2.6,.35);
  \node at (.92,-2.1) {$Q$};
  \draw[shift={(0,-2.6)}] (1.1,0) rectangle (2.1,1);
  \node at (2.4,-2.1) {$=$};
  \draw[shift={(0,-2.6)}] (2.6,0) rectangle (3.6,1);
  \draw[semithick,shift={(1.1,-2.6)}] (.67,.8) -- (0,.8);
  \draw[semithick,shift={(1.1,-2.6)}] (.33,.6) -- (0,.6);
  \draw[semithick,shift={(1.1,-2.6)}] (.67,.4) -- (0,.4);
  \draw[semithick,shift={(1.1,-2.6)}] (.33,.2) -- (0,.2);
  \draw[thin,black!40!blue!60,fill=black!40!blue!10,shift={(0,-2.6)}] (2.7,.55) rectangle (3,.95);
  \draw[thin,black!40!blue!60,fill=black!40!blue!10,shift={(0,-2.6)}] (3.2,.05) rectangle (3.5,.45);
  \draw[semithick,black!40!blue!60,shift={(0,-2.6)},-{To[scale=.65,width=1.5mm]},looseness=.8] (2.95,.53) to[out=-70,in=200] (3.18,.44);
  \draw[semithick,black!40!blue!60,-{To[scale=.65,width=1.5mm]},shift={(0,-2.6)}] (3.35,.45) to[out=90,in=0] (3,.75);
  \draw[-{To[scale=.65,width=1.5mm]},black!30!red!60,semithick,shift={(1.2,-2.6)}] (.5,.405) -- (.5,.795);
  \draw[semithick,shift={(0,-2.6)}] (2.93,.9) -- (2.6,.9);
  \draw[semithick,shift={(0,-2.6)}] (2.76,.8) -- (2.6,.8);
  \draw[semithick,shift={(0,-2.6)}] (2.93,.7) -- (2.6,.7);
  \draw[semithick,shift={(0,-2.6)}] (2.76,.6) -- (2.6,.6);
  \draw[-{To[scale=.65,width=1.5mm]},black!30!red!60,semithick,shift={(0,-2.6)}] (2.85,.705) -- (2.85,.895);
  \draw[-{To[scale=.65,width=1.5mm]},black!60!green!60,semithick,shift={(.5,-3.1)}] (2.85,.705) -- (2.85,.895);
  \draw[semithick,shift={(0,-2.6)}] (3.26,.1) -- (2.6,.1);
  \draw[semithick,shift={(0,-2.6)}] (3.43,.2) -- (2.6,.2);
  \draw[semithick,shift={(0,-2.6)}] (3.26,.3) -- (2.6,.3);
  \draw[semithick,shift={(0,-2.6)}] (3.43,.4) -- (2.6,.4);
\end{tikzpicture}
  \caption{Instances of the Pontrjagin product
    $\smash{H_1(\frM_{1,1})\otimes H_0(\frM_{0,1}^1)\to H_1(\frM_{1,1}^1)}$,
    the Browder bracket
    $\smash{H_1(\frM_{1,1})\otimes H_0(\frM_{0,1}^1)\to H_2(\frM_{1,1}^1)}$,
    and the Dyer–Lashof square \mbox{$H_1(\frM_{1,1})\to
    H_3(\frM_{2,1})$}. In the last example, the depicted class is supported on the mapping
    torus of the twist on $S^1\times S^1$.}\label{fig:Pontrjagin-Browder-DL}
\end{figure}
\begin{align*}
  [x,x]       &= 2\cdot Q(x) = 0     \tag{divisibility of the bracket}\\
  Q(0)        &= Q(1)=0            \tag{nullification}\\
  Q(x+y)      &= Q(x) + [x,y] + Q(y)   \tag{non-linearity}\\
  Q(x\cdot y) &= Q(x)\cdot y^2 + x\cdot [x,y]\cdot y + x^2\cdot Q(y)\tag{Cartan}\\
  [Q(x),y]    &= [x,[x,y]]
\end{align*}
We finally mention the \emph{Nishida relations}, which involve the dual Steenrod
squares in homology $\smash{\mathrm{Sq}{}_t\colon H_\bullet(\frM_{g,1}^m;\bF_2)\to
H_{\bullet-t}(\frM_{g,1}^m;\bF_2)}$, namely%
\begin{align*}
  \mathrm{Sq}{}_t(x\cdot y)= \sum_{i=0}^t\mathrm{Sq}{}_ix\cdot \mathrm{Sq}{}_{t-i}y\qquad\quad
  \mathrm{Sq}{}_t[x,y]     = \sum_{i=0}^t[\mathrm{Sq}{}_ix,\mathrm{Sq}{}_{t-i}y],
\end{align*}
which follow directly from the naturality of the $\mathrm{Sq}_t$ and the
Cartan formula, and
\begin{align*}
  \mathrm{Sq}{}_{2t}(Q(x))   &= Q(\mathrm{Sq}{}_tx) + \sum_{i=0}^{t-1}
                             [\mathrm{Sq}{}_ix,\mathrm{Sq}{}_{2t-i}x],\\
  \mathrm{Sq}{}_{2t+1}(Q(x)) &= (\mathrm{Sq}{}_tx)^2 + \sum_{i=0}^t
                             [\mathrm{Sq}{}_ix,\mathrm{Sq}{}_{2t+1-i}x].
\end{align*}
One particular example of a dual Steenrod square is the Bockstein $\beta$ for
the short exact coefficient sequence $0\to \Z_2\to \Z_4\to \Z_2\to 0$, which
agrees with the Steenrod square $\mathrm{Sq}_1$. In this case, the last Nishida
relations is of the form $\beta Q(x)=x^2+[x,\beta x]$.
\subsubsection{A splitting modulo 2}
\label{subsubsec:splitMod2}

Note that $\coprod_{g,m}\frM_{g,1}^m$ has not only $\coprod_g\frM_{g,1}$ as a
$\scC_2$-subalgebra, but also the collection $\coprod_m\frM_{0,1}^m$. With
coefficients in $\bF_2$, its homology is called the \emph{Dyer–Lashof algebra},
and it has an easy description due to \cite[§\,\textsc{iii}]{Cohen-Lada-May}: if
$\bfa\in H_0(\frM_{0,1}^1)$ is the ground class, then we have an isomorphism
$\bigoplus_m H_\bullet(\frM_{0,1}^m;\bF_2)\cong \bbF_2[Q^j\bfa]_{j\ge 0}$.  Via
the Pontrjagin product, $\bigoplus_{g,m}H_\bullet(\frM_{g,1}^m;\bF_2)$ is a
module over Dyer–Lashof algebra, and based on
\cite{CFB-Cohen-Taylor,CFB-Tillmann-01}, Bianchi \cite[Thm.\,6.5]{Bianchi}
recently has established an isomorphism
\begin{align}\label{eq:Andrea}\bigoplus_{m\ge 0} H_\bullet(\frM_{g,1}^m;\bF_2)
  \cong \bbF_2[Q^j\bfa]_{j\ge 0}\otimes H_\bullet(\frM_{g,1};\on{Sym}\hspace*{1px}\caH)\end{align}
of modules over the Dyer–Lashof algebra, where
$\on{Sym}\hspace*{1px}\caH=\bigoplus_k \mathrm{Sym}^k\caH$ is the symmetric
algebra over the $\Gamma_{g,1}$-representation $H_1(F;\bbF_2)$, with $F$ a
surface of genus $g$. This isomorphism is bigraded by the number of punctures
and by homological degree: here each class in
$H_i(\frM_{g,1};\mathrm{Sym}^k\caH)$ as $k$ punctures and homological degree
$i+k$, and the Dyer–Lashof algebra is bigraded as
usual.\looseness-1 %

\subsection{Rotating the boundary curve}
\label{subsec:Rot}

We have a (generally non-free) action $\rho\colon S^1\times \frM_{\smash{g,1}}^m\to \frM_{\smash{g,1}}^m$ by
rotating the parametrisation of the single boundary curve, or, in other words,
rotating the non-zero tangent vector $X\in T_Q F$ of a conformal class
$[F,\caP,(Q,X)]$. This gives rise to a homology operation\looseness-1
\[R\colon H_i(\frM_{g,1}^m)\longrightarrow H_{i+1}(\frM_{g,1}^m),\quad
R(x)\coloneqq\rho_*([S^1]\times x),\]
called the \emph{rotation}, where $[S^1]$ is the fundamental class of the 
circle. Note that $R\circ R=0$, just because the Pontrjagin product
$[S^1]\cdot [S^1]$ lives in $H_2(S^1)=0$.

This rotation operation is due to the fact that the action of the little
$2$-discs operad $\scC_2$ on $\coprod_{g,m}\frM_{\smash{g,1}}^m$ can be extended
to an action of the \emph{framed} little $2$-discs operad
$\scC_{\smash{2}}^{\text{fr}}$ by allowing parametrisations of the boundary
curves. As a consequence, we obtain several formulæ, which are induced by the
internal structure of $\scC_{\smash{2}}^{\text{fr}}$:
\begin{align*}
  R(1)        &= 0,\\
  R(x\cdot y) &= R(x)\cdot y + [x,y] + (-1)^x\cdot x\cdot R(y).
\end{align*}
Here the first one follows from the fact that $\scC_{\smash 2}^{\text{fr}}(0)$
is contractible, and the second one can be checked by comparing the loop in
$\scC_{\smash 2}^{\text{fr}}(2)$ that describes the rotation with the three
loops describing the summands on the right side. From these two formulæ and from
$R^2=0$, it formally follows that
\begin{align*}
  [R(x),y] = R(R(x)\cdot y) + (-1)^{x}\cdot R(x)\cdot R(y).
\end{align*}
Let us point out that $R(x)$ is very often trivial: recall that
$\frM_{\smash{g,1}}^m$ is a classifying space for the mapping class group
$\Gamma_{\smash{g,1}}^m$, and capping the boundary curve with a disc gives rise
to a map
$\smash{\delta\colon H_\bullet(\Gamma_{g,1}^m)\to H_\bullet(\Gamma_{g,0}^m)}$,
which is surjective for $\smash{\bullet\le \frac23\hspace*{1px}g+1}$ and
injective for $\smash{\bullet\le\frac23\hspace*{1px}g}$ by Harer’s stability
theorem.  On the other hand, the above $S^1$-action clearly becomes homotopy
trivial after capping with a disc, so the image of $R$ lies in the kernel of
$\delta$. It follows that $R(x)=0$ for $x\in H_\bullet(\frM_{g,1}^m)$ with
$\bullet\le \frac23\hspace*{1px}g-1$.  We will encounter several non-trivial
examples of $R(x)$ in \tref{Remark}{rem:Rotation}.

\subsection{Permuting curves and gluing pairs of pants}
\label{subsec:FinB}

If we widen our view for a moment and consider moduli spaces of surfaces with
\emph{multiple} boundary curves, then we obtain, for each $n\ge 1$, an action of
the symmetric group $\frS_n$ on $\frM_{g,n}^m$ by permuting boundary curves, or,
in other words, on $\frP_{g,n}^m$ by exchanging layers of slit pictures.
For a permutation $\sigma\in \frS_n$, we denote the induced homology operation
by $\sigma_*\colon H_\bullet(\frM_{g,n}^m)\to H_\bullet(\frM_{g,n}^m)$.
Secondly, we get, for each $1\le l\le n-1$, a map\looseness-1
\[s^l\colon \frM_{g,n}^m\longrightarrow \frM_{g+1,n-1}^m\]
by gluing a pair of pants to the $l$\textsuperscript{th} and the
$(l+1)$\textsuperscript{st} boundary curve. In terms of slit pictures, this is
the same as joining the $l$\textsuperscript{th} and the
$(l+1)$\textsuperscript{st} layer of the slit picture on a single new layer, see
\tref{Figure}{fig:Codegeneracy}. The codegeneracies constitute one of the three
classes of maps in Harer’s stability theorem: the induced maps\footnote{Again,
  the punctured case follows from the unpunctured one, see \cite{Hanbury}.}
$\smash{s^l_*\colon H_\bullet(\frM_{\smash{g,n}}^m)\to H_\bullet(\frM_{\smash{g+1,n-1}}^m)}$ are
surjective for $\smash{\bullet\le \frac23\hspace*{1px}g+\frac13}$ and isomorphisms for
$\smash{\bullet\le \frac23\hspace*{1px} g-\frac23}$.\looseness-1

In \cite[Rmk.\,4.4.19]{Kranhold}, it is shown that
$s^ks^l\simeq s^ls^{k+1}$ for $l\le k$, whence the induced maps in homology
behave like codegeneracies, justifying our choice of notation. Moreover, we have
$s^l\cyc{l,l{+}1}\simeq s^l$, where $\cyc{l,l{+}1}\in \frS_n$ is the
transposition that exchanges $l$ and $l+1$: this can be seen by comparing $\frP$
to different models, see \cite[Rmk.\,5.2.16]{Kranhold}. Altogether, the
assignment $n\mapsto H_\bullet(\frM_{\bullet,n}^{\bullet})$ is functorial with
respect to surjective maps $\{1,\dotsc,n\}\to \{1,\dotsc,n'\}$.

\begin{figure}
  \centering
  \[s^2\,\,\begin{tikzpicture}[scale=1.35,baseline={([yshift=-.5ex]current bounding box.center)}]
    \draw[thin,black!70](0,2) rectangle (.8,2.8);
    \draw[thin,black!70](0,1) rectangle (.8,1.8);
    \draw[thin,black!70](0,0) rectangle (.8,.8);
    \draw[semithick](0,.16) -- (.16,.16);
    \draw[semithick](0,.32) -- (.32,.32);
    \draw[semithick](0,.48) -- (.48,.48);
    \draw[semithick](0,.64) -- (.64,.64);
    \draw[semithick](0,1.267) -- (.48,1.267);
    \draw[semithick](0,1.533) -- (.64,1.533);
    \draw[semithick](0,2.267) -- (.16,2.267);
    \draw[semithick](0,2.533) -- (.32,2.533);
  \end{tikzpicture}~=~\begin{tikzpicture}[scale=1.35,baseline={([yshift=-.5ex]current bounding box.center)}]
    \fill[black!10] (.2,.2) rectangle (.6,.6);
    \draw[very thin,black!35] (.2,.2) rectangle (.6,.6);
    \fill[black!10] (.2,1) rectangle (.6,1.39);
    \draw[very thin,black!35] (.2,1) rectangle (.6,1.39);
    \fill[black!10] (.2,1.41) rectangle (.6,1.8);
    \draw[very thin,black!35] (.2,1.41) rectangle (.6,1.8);
    \draw[thin,black!70](0,1) rectangle (.8,1.8);
    \draw[thin,black!70](0,0) rectangle (.8,.8);
    \draw[semithick](0,.28) -- (.28,.28);
    \draw[semithick](0,.36) -- (.36,.36);
    \draw[semithick](0,.44) -- (.44,.44);
    \draw[semithick](0,.52) -- (.52,.52);
    \draw[semithick](0,1.133) -- (.44,1.133);
    \draw[semithick](0,1.267) -- (.52,1.267);
    \draw[semithick](0,1.533) -- (.28,1.533);
    \draw[semithick](0,1.667) -- (.36,1.667);
  \end{tikzpicture}\]\vspace*{-12px}
  \caption{An instance of $s^2\colon \frP_{0,3}\to \frP_{1,2}$}
  \label{fig:Codegeneracy}
\end{figure}

\subsection{Transfer operations}

In the following, we construct two operations that, visually speaking, use the
homological transfer maps to choose one or several punctures, endow them with
tangential directions, and regard them as new boundary curves.

\subsubsection{The \texorpdfstring{$T$}{T}-operation}
\label{subsubsec:T}
\enlargethispage{\baselineskip}
Let $g\ge 0$ and $m\ge 1$, which means that we have at least one puncture. Then
there is an $m$-sheeted covering
\mbox{$\alpha\colon \frM_{\smash{g,1}}^{m-1,1}\to \frM_{\smash{g,1}}^m$} that
marks one of the punctures.

We also have a bundle map
$\smash{\frM_{\smash{g,1}}^{m-1,1}\to \frM_{\smash{g,1}}^{m-1}}$ by forgetting
the chosen puncture: this is a subbundle of the universal surface bundle over
$\frM_{\smash{g,1}}^{m-1}$, where points in the fibre have to avoid the boundary
and the given $m-1$ punctures. Over $\frM_{\smash{g,1}}^{m-1,1}$, we consider
the vertical unit tangent bundle
\mbox{$\smash{p\colon W^1\coloneqq UT^\bot\frM_{g,1}^{\smash{m-1,1}}\to
  \frM_{\smash{g,1}}^{m-1,1}}$}, which is a $1$-dimensional sphere bundle. Note
that $W^1$ is oriented, as the structure group of $\smash{\frM_{g,1}^{m-1,1}}$
contains only orientation-preserving diffeomorphisms.

We have an equivalence $\theta\colon W^1\to \frM_{\smash{g,2}}^{m-1}$ by using
that an isolated puncture, together with a unit tangential direction \emph{is}
(by our choice of model) the same as a parametrised boundary curve. Finally, we
use the homological transfer and define
\[\begin{tikzcd}
    \hat T\colon H_i(\frM_{g,1}^m)\ar[r,"\alpha^!"] &
    H_i(\frM_{g,1}^{m-1,1})\ar[r,"p^!"] &
    H_{i+1}(W^1)\ar[r,"\theta_*"] &
    H_{i+1}(\frM_{g,2}^{m-1}).\end{tikzcd}\]
As we want to focus on moduli spaces of surfaces with a single boundary curve,
we glue in a pair of pants and define the \emph{$T$-operation} by
\[T\coloneqq s^1_*\circ \hat T\colon H_i(\frM_{g,1}^m)\longrightarrow H_{i+1}(\frM_{g+1,1}^{m-1}).\]
Under the gluing construction $\frP_{\smash{g,1}}^m\to \frM_{\smash{g,1}}^m$,
the map $\theta$ can be visualised in terms of slit domains: if $W^1_\frP$
denotes the pullback of the bundle $W^1$ along the affine bundle
$\frP_{\smash{g,1}}^{m-1,1}\to \frM_{\smash{g,1}}^{m-1,1}$, then an element in
$W^1_\frP$ is given by a slit domain, together with the choice of one of the cycles
of $\sigma_q$ which does not contain $0$, and, on a  \emph{circle} given by
regluing the upper and lower faces of the rectangles inside the chosen cycle, a marked
point. We can turn this left face into an honest boundary curve by providing a
new layer for it: we draw a small slit at the given position and pair it with a
single slit on a new layer. This gives rise to a map
$\smash{\theta_\frP\colon W^1_\frP\to \frP_{\smash{g,2}}^{m-1}}$, covering
$\theta$ up to homotopy. Finally, we have already seen that gluing in a pair of
pants corresponds to joining both layers of the slit picture on a single layer:
thus, we end up with \mbox{\tref{Figure}{fig:T-and-E}.}\looseness-1

If $x$ lies in a moduli space of a surface without punctures,
then we put $T(x)\coloneqq 0$. It follows immediately from the definition that
$T$ satisfies the graded Leibniz rule
\[T(x\cdot y)=(-1)^y\cdot T(x)\cdot y + x\cdot T(y)\]
with respect to the Pontrjagin product. In particular, if $x$ is a class in
$\frM_{g,1}$, then we have $T(x\cdot y)=x\cdot T(y)$, i.e.\ the operation $T$ is
$\bigoplus_g H_\bullet(\frM_{g,1})$-linear.

\subsubsection{The \texorpdfstring{$E$}{E}-operation}
\label{subsubsec:E}

Let $g\ge 0$ and $m\ge 2$. Then we consider the $\binom{m}{2}$-sheeted covering
$\beta\colon \frM_{g,1}^{m-2,2}\to \frM_{g,1}^m$ where two punctures are
separated from the other ones, but unordered. Over the total space, we consider
the torus bundle $q\colon W^2\to \frM_{\smash{g,1}}^{m-2,2}$ given by the symmetric fibre
product of the two vertical tangent bundles at the two
punctures.\footnote{Formally, we consider the double covering
  $\smash{B\coloneqq \frM_{\smash{g,1}}^{m-2,1,1}\to\frM_{\smash{g,1}}^{m-2,2}}$
  where the two points are ordered. There are two unit vertical tangent bundles
  $L$ and $L'$ over $B$ and the the fibre product
  $\smash{L\times_B L'\to \frM_{\smash{g,1}}^{m-2,1,1}}$ is
  $\frS_2$-equivariant. Then $q$ is the induced map on quotients.}

There is a map $\smash{\eta\colon W^2\to \frM_{\smash{g+1,1}}^{m-1}}$ as
follows: via the exponential map, an element in $W^2$ is given by a conformal
class of a surface $F$, a subset $\smash{\caP\subseteq\mathring F}$ of
cardinality $m-2$, two punctures $x_1,x_2\in F$ with small
disjoint discs $D_1,D_2\subseteq\mathring F\setminus\caP$ around them, and
points $x_i'\in \partial D_i$.  If we cut along the two straight lines from
$x_1$ to $x_1'$ and from $x_2$ to $x_2'$ and reglue, then we have identified the
two punctures and increased the genus by $1$. Again, this construction can be
visualised in terms of slit pictures, see \tref{Figure}{fig:T-and-E}.

The torus bundle $q\colon W^2\to \frM_{\smash{g,1}}^{m-2,2}$ is not orientable, but
if we work over $\bbF_2$, then we still have a homological transfer, which can
be used to define
\[\begin{tikzcd}E\colon H_i(\frM_{g,1}^m)\ar[r,"\beta^!"] &
    H_i(\frM_{g,1}^{m-2,2})\ar[r,"q^!"] &
    H_{i+2}(W^2)\ar[r,"\eta_*"] & H_{i+2}(\frM_{g+1,1}^{m-1}).\end{tikzcd}\] 
Let us point out that the $E$-operation can also be applied integrally to ground
classes $x=[*]\in H_0(\frM_{g,n}^m)$. In this case, $E(x)$ is just pushforward of some
fundamental class of $S^1\times S^1$ along $\smash{S^1\times S^1\to W^2\to \frM_{\smash{g+1,n}}^{m-1}}$,
where the first map is some fibre inclusion. Choosing a path that exchanges the
two punctures, we see that $2\cdot E(x)=0$.\vspace*{4px}

\begin{figure}[h]
  \[T
  \raisebox{-21px}{
    \begin{tikzpicture}[scale=.9]
      \draw[thin,dgrey] (0,.2) rectangle (2,-1.8);
      \draw[thick] (0,-1.2) -- (1.5,-1.2);
      \draw[thick] (0,-1.6) -- (1.5,-1.6);
      \draw [dgrey] (0,-.55) -- (1,-.55);
      \draw [dgrey] (0,-.65) -- (.5,-.65);
      \draw [dgrey] (0,-.75) -- (1,-.75);
      \draw [dgrey] (0,-.85) -- (.5,-.85);
    \end{tikzpicture}}=\raisebox{-21px}{\begin{tikzpicture}[scale=.9]
      \draw[dred!60,semithick,-{To[scale=.65,width=1.5mm]}] (.6,-1.6) -- (.6,-1.21);
      \draw[thin,dgrey] (0,.2) rectangle (2,-1.8);
      \draw[thick] (0,-1.2) -- (1.5,-1.2);
      \draw[thick] (0,-1.6) -- (1.5,-1.6);
      \draw [dgrey] (0,-.55) -- (1,-.55);
      \draw [dgrey] (0,-.65) -- (.5,-.65);
      \draw [dgrey] (0,-.75) -- (1,-.75);
      \draw [dgrey] (0,-.85) -- (.5,-.85);
      \draw [thick] (0,0) -- (.5,0);
      \draw [thick,dred] (0,-1.4) -- (.5,-1.4);
    \end{tikzpicture}}
  \qquad\qquad
  E
  \raisebox{-21px}{
    \begin{tikzpicture}[scale=.9]
      \draw[thin,dgrey] (0,0) rectangle (2,-2);
      \draw[thick] (0,-1.4) -- (1.6,-1.4);
      \draw[thick] (0,-1.8) -- (1.6,-1.8);
      \draw[thick] (0,-.2) -- (.8,-.2);
      \draw[thick] (0,-.6) -- (.8,-.6);
      \draw [dgrey] (0,-.85) -- (1.2,-.85);
      \draw [dgrey] (0,-.95) -- (.6,-.95);
      \draw [dgrey] (0,-1.05) -- (1.2,-1.05);
      \draw [dgrey] (0,-1.15) -- (.6,-1.15);
    \end{tikzpicture}}=\raisebox{-21px}{\begin{tikzpicture}[scale=.9]
      \draw[dgreen!60,semithick,-{To[scale=.65,width=1.5mm]}] (.6,-.6) -- (.6,-.21);
      \draw[dred!60,semithick,-{To[scale=.65,width=1.5mm]}] (.6,-1.8) -- (.6,-1.41);
      \draw[thin,dgrey] (0,0) rectangle (2,-2);
      \draw[thick] (0,-1.4) -- (1.6,-1.4);
      \draw[thick] (0,-1.8) -- (1.6,-1.8);
      \draw[thick] (0,-.2) -- (.8,-.2);
      \draw[thick] (0,-.6) -- (.8,-.6);
      \draw [dgrey] (0,-.85) -- (1.2,-.85);
      \draw [dgrey] (0,-.95) -- (.6,-.95);
      \draw [dgrey] (0,-1.05) -- (1.2,-1.05);
      \draw [dgrey] (0,-1.15) -- (.6,-1.15);
      \draw [thick, dgreen] (0,-.4) -- (.5,-.4);
      \draw [thick, dred] (0,-1.6) -- (.5,-1.6);
    \end{tikzpicture}}
\]    
  \caption{The two operations
    $\smash{T\colon H_i(\frP_{1,1}^1)\to H_{i+1}(\frP_{2,1})}$ and
    $\smash{E\colon H_i(\frP_{1,1}^2)\to H_{i+2}(\frP_{2,1}^1)}$, here applied
    to ground classes.}\label{fig:T-and-E}\vspace*{-5px}
\end{figure}

\subsection{Segal–Tillmann maps}

In \cite{Segal-Tillmann}, Segal and Tillmann study maps
$C_{2g+2}(\mathring D^2)\to \frM_{g,2}$ from the configuration space of $2g+2$
unordered points in an open disc to the moduli space of surfaces of genus $g$
and two boundary curves. Because both spaces are aspherical, the map can likewise
be described on the level of fundamental groups, where it is a homomorphism
$\Br_{2g+2} \to \Gamma_{g,2}$ and has the following description that is due
to \cite{Song-Tillmann}: the braid group $\Br_{2g+2}$ is generated by
$\sigma_1,\dotsc,\sigma_{2g+1}$, where $\sigma_i$ interchanges the
$i$\textsuperscript{th} and the $(i+1)$\textsuperscript{st} string. We choose
$D_1,\dotsc,D_{2g+1}\in\Gamma_{g,2}$ to be the Dehn twists about the curves
$\alpha_1,\dotsc,\alpha_{2g+1}$ as in \tref{Figure}{fig:STDehn}. Then two
consecutive curves intersect exactly once while all other pairs of curves do not
intersect at all. Therefore, the mapping classes $D_1,\dotsc,D_{2g+1}$ satisfy
the braid relations and we obtain the desired homomorphism. Its image is
contained in the symmetric mapping class group, which was studied in \cite{Birman-Hilden}.

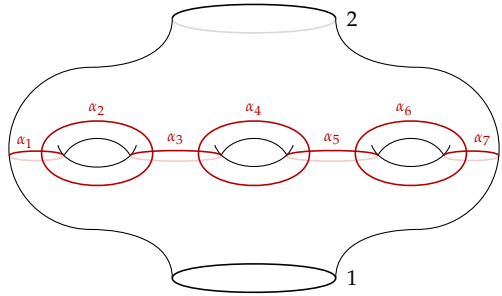
\begin{figure}[h]
  \centering
  \begin{tikzpicture}[scale=4.3]
  \draw[semithick,black!30!red!20,looseness=.3](-1.249,-1.373)
  to[out=-90,in=-90](-1.082,-1.373);
  \draw[semithick,black!30!red!20,looseness=.2](-.878,-1.373)
  to[out=-90,in=-90] (-.6,-1.373);
  \draw[semithick,black!30!red!20,looseness=.2](-.4,-1.373)
  to[out=-90,in=-90] (-.121,-1.373);
  \draw[semithick,black!30!red!20,looseness=.3](.249,-1.373)
  to[out=-90,in=-90] (.081,-1.373);
  \draw[black!15,looseness=.3,semithick] (-.25,-.95) to[out=-90,in=-90] (-.75,-.95);
  \draw (0,-1.1) to[out=180,in=-90] (-.25,-.95);
  \draw[looseness=.3,semithick] (-.25,-.95) to[out=90,in=90] (-.75,-.95);
  \draw (-.75,-.95) to[out=-90,in=0] (-1,-1.1)
  to[out=180,in=0] (-1,-1.1) to[out=180,in=90] (-1.25,-1.35)
  to[out=-90,in=180] (-1,-1.6) to[out=0,in=90] (-.75,-1.75);
  \draw (-.25,-1.75) to[out=90,in=180] (0,-1.6);
  \draw[looseness=.3,semithick] (-.25,-1.75) to[in=90,out=90] (-.75,-1.75)
  to[in=-90,out=-90] (-.25,-1.75);
  \draw (0,-1.6) to[out=0,in=180] (0,-1.6)
  to[out=0,in=-90] (.25,-1.35) to[out=90,in=0] (0,-1.1);
  \draw[shift={(-.48,-.005)}] (-.60,-1.365) to[out=60,in=120] (-.40,-1.365);
  \draw[shift={(-.48,0)}] (-.62,-1.34) to[out=-75,in=-105] (-.38,-1.34);
  \draw[shift={(0,-.005)}] (-.60,-1.365) to[out=60,in=120] (-.40,-1.365);
  \draw[shift={(0,0)}] (-.62,-1.34) to[out=-70,in=-110] (-.38,-1.34);
  \draw[shift={(-.02,-.005)}] (-.1,-1.365) to[out=60,in=120] (.1,-1.365);
  \draw[shift={(-.02,0)}] (-.12,-1.34) to[out=-70,in=-110] (.12,-1.34);
  \draw[semithick,black!30!red,looseness=.3] (-1.249,-1.373) to[out=90,in=90] (-1.082,-1.373);
  \draw[semithick,black!30!red,looseness=.2] (-.878,-1.373) to[out=90,in=90] (-.6,-1.373);
  \draw[semithick,black!30!red,looseness=.2] (-.4,-1.373) to[out=90,in=90] (-.121,-1.373);
  \draw[semithick,black!30!red,looseness=.3] (.249,-1.373) to[out=90,in=90] (.081,-1.373);
  \draw[semithick,black!30!red] (-1.15,-1.365) to[out=90,in=90] (-.81,-1.365)
  to[out=-90,in=-90] (-1.15,-1.365);
  \draw[semithick,black!30!red,shift={(.48,0)}] (-1.15,-1.365)
  to[out=90,in=90] (-.81,-1.365) to[out=-90,in=-90] (-1.15,-1.365);
  \draw[semithick,black!30!red,shift={(.96,0)}] (-1.15,-1.365)
  to[out=90,in=90] (-.81,-1.365) to[out=-90,in=-90] (-1.15,-1.365);
  \node[black!30!red] at (-1.2,-1.33)   {\tiny $\alpha_1$};
  \node[black!30!red] at (-.739,-1.32)  {\tiny $\alpha_3$};
  \node[black!30!red] at (-.2605,-1.32) {\tiny $\alpha_5$};
  \node[black!30!red] at (.2,-1.323)    {\tiny $\alpha_7$};
  \node[black!30!red] at (-.98,-1.23)   {\tiny $\alpha_2$};
  \node[black!30!red] at (-.5,-1.23)    {\tiny $\alpha_4$};
  \node[black!30!red] at (-.04,-1.23)   {\tiny $\alpha_6$};
  \node at (-.2,-1.75) {\scriptsize $1$};
  \node at (-.2,-.95) {\scriptsize $2$};
\end{tikzpicture}
  \caption{The curves $\alpha_1,\dotsc,\alpha_7$ defining the images of the
    elementary braids $\sigma_1,\dotsc,\sigma_7$ along the group homomorphism
    $\Br_8\to\Gamma_{3,2}$.}\label{fig:STDehn}
\end{figure}

For our purposes, it is convenient to cap off the second boundary curves by
gluing a disc, i.e.\ postcomposing with the capping homomorphism
\mbox{$\Gamma_{g,2}\to \Gamma_{g,1}$}. We call the composition
$\mathrm{ST}\colon \Br_{2g+2}\to \Gamma_{g,1}$, as well as its topological
analogue \mbox{$\mathrm{ST}\colon C_{2g+2}(\mathring D^2)\to \frM_{g,1}$}, the
\emph{Segal–Tillmann map}. The main result of \cite{Song-Tillmann} tells us that
the induced map in homology is trivial in the stable range.\looseness-1

\subsection{Vertical Browder brackets}
\label{subsec:vertB}

The boundary permutations and codegeneracies from §\,\ref{subsec:FinB} are part
of a larger operadic action on moduli spaces with multiple boundary curves: in
\cite{Kranhold}, the $\scC_2$-action on $\coprod_{g,m}\frP_{g,1}^m$ has
been generalised to the case of multiple boundary curves, ending up with a
coloured operad, called the \emph{vertical operad} $\scV_{1,1}$, acting on the
sequence $(\coprod_{g,m}\frP_{g,n}^m)_{n\ge 1}$.

Both the internal structure of $\scV_{1,1}$ and the explicit
construction is rather lengthy; we will only use one specific class
of operations that arise from it: it generalises the classical Browder
bracket and is of the form
\[[-,-]\colon H_i(\frM_{g_1,n_1}^{m_1})\otimes H_j(\frM_{g_2,n_2}^{m_2})
  \longrightarrow H_{i+j+1}(\frM_{g_1+g_2,n_1+n_2-1}^{m_1+m_2}).\]
We call it the \emph{vertical Browder bracket}, see
\cite[Def.\,4.4.1\,\&\,Constr.\,5.2.15]{Kranhold}. Pictorially, it takes a slit
domain $x$ on $n_1$ layers and a slit domain $y$ on $n_2$ layers, places the
respective first layers of both arguments on the first layer of a new slit
domain, puts the remaining $n_1-1$ layers of $x$ on the layers $2,\dotsc,n_1$,
and the remaining $n_2-1$ layers of $y$ on the layers $n_1+1,\dotsc,n_1+n_2-1$,
and lets the respective first layers spin around each other, see
\tref{Figure}{fig:Vertical-Browder}. If $n_1=n_2=1$, then we recover
the classical Browder bracket.\looseness-1

As for $E_2$-algebras, there are several universal formulæ involving the
vertical Browder brackets and the boundary permutations and codegeneracies from
§\,\ref{subsec:FinB}, see \cite[Prop.\,4.4.16]{Kranhold}. We will only be using the
following relation\footnote{For general $\scV_{1,1}$-algebras, the second
  summand in formula (\ref{eq:V1}) is
  $s^1_*\tau^{\vphantom{\smash 1}}_*[\tau^{\vphantom{\smash 1}}_*x,y]$. However, for the
  special algebra $\smash{(\coprod_{g,m}\frP_{g,n}^m)_{n\ge 1}}$, codegeneracies and
  permutations satisfy all relations to assemble into a functor as in
  §\,\ref{subsec:FinB}, so we can use that $s^1\circ \tau$ and
  $s^1$ agree \emph{as maps} $\{1,2\}\to \{1\}$.}, with $\tau\coloneqq \cyc{1,2}\in\frS_{n_1+n_2-1}$:
\begin{align}
  [s^1_*x,y] &= s^{\smash 1}_*[x,y]+s^{\smash 1}_*[\tau^{\vphantom{\smash 1}}_*x,y].\label{eq:V1}
\end{align}

\begin{tFigure}
  \centering
  \[
  \mathopen{}\left[\vbox to 1.5cm{}\right.\mathclose{}\begin{tikzpicture}[scale=1.4,baseline={([yshift=-.5ex]current bounding box.center)}]
    \draw[thin,black!30!red!50](0,1) rectangle (.8,1.8);
    \draw[thin,black!30!red!50](0,0) rectangle (.8,.8);
    \draw[semithick,black!30!red](0,.267) -- (.16,.267);
    \draw[semithick,black!30!red](0,.533) -- (.32,.533);
    \draw[semithick,black!30!red](0,1.16) -- (.64,1.16);
    \draw[semithick,black!30!red](0,1.32) -- (.64,1.32);
    \draw[semithick,black!30!red](0,1.48) -- (.16,1.48);
    \draw[semithick,black!30!red](0,1.64) -- (.32,1.64);
  \end{tikzpicture}~\raisebox{-30px}{,}~\begin{tikzpicture}[scale=1.4,baseline={([yshift=-.5ex]current bounding box.center)}]
    \draw[thin,black!40!green!50](0,1) rectangle (.8,1.8);
    \draw[thin,black!40!green!50](0,0) rectangle (.8,.8);
    \draw[semithick,black!40!green](0,.267) -- (.52,.267);
    \draw[semithick,black!40!green](0,.533) -- (.36,.533);
    \draw[semithick,black!40!green](0,1.16) -- (.52,1.16);
    \draw[semithick,black!40!green](0,1.64) -- (.36,1.64);
  \end{tikzpicture}\mathopen{}\left.\vbox to 1.5cm{}\right]\mathclose{} =~ \begin{tikzpicture}[scale=1.4,baseline={([yshift=-.5ex]current bounding box.center)}]
    \draw[black!30!red!45,very thin,fill=black!30!red!10] (0,0) rectangle (.267,.267);
    \draw[black!30!red!45,very thin,fill=black!30!red!10] (0,1) rectangle (.267,1.267);
    \draw[black!40!green!50,very thin,fill=black!40!green!10] (.267,.267) rectangle (.533,.533);
    \draw[black!40!green!50,very thin,fill=black!40!green!10] (.267,2.267) rectangle (.533,2.533);
    \draw[black!30!red!60,semithick] (.267,.133) -- (.4,.133) to[out=0,in=-90] (.667,.4) to[out=90,in=0] (.4,.667)
    to[out=180,in=87] (.138,.44);
    \draw[black!30!red!60,semithick,-{To[scale=.65,width=1.5mm]}] (.267,1.133) -- (.4,1.133) to[out=0,in=-90] (.667,1.4) to[out=90,in=0] (.4,1.667)
    to[out=180,in=90] (.133,1.4) -- (.133,1.267);
    \draw[black!30!red!60,semithick,-{To[scale=.65,width=1.5mm]}] (.133,.355) -- (.133,.267);
    \draw[black!30!red!60,semithick] (.138,2.48) to[out=-93,in=90] (.133,2.32);
    \draw[thin,dgrey](0,2) rectangle (.8,2.8);
    \draw[thin,dgrey](0,1) rectangle (.8,1.8);
    \draw[thin,dgrey](0,0) rectangle (.8,.8);
    \draw[semithick,black!30!red](0,1.0533) -- (.2134,1.0533);
    \draw[semithick,black!30!red](0,1.1067) -- (.2134,1.1067);
    \draw[semithick,black!30!red](0,1.16) -- (.0533,1.16);
    \draw[semithick,black!30!red](0,1.2133) -- (.1067,1.2133);
    \draw[semithick,black!40!green](0,2.32) -- (.44,2.32);
    \draw[semithick,black!40!green](0,2.48) -- (.387,2.48);
    \draw[semithick,black!30!red](0,.08867) -- (.0533,.08867);
    \draw[semithick,black!30!red](0,.1773) -- (.1067,.1773);
    \draw[semithick,black!40!green](0,.355) -- (.4403,.355);
    \draw[semithick,black!40!green](0,.444) -- (.387,.444);
  \end{tikzpicture}
\]\vspace*{-10px}
  \caption{An instance of
    $[-{,}-]\colon \frP_{0,2}^1\times \frP^{\vphantom 1}_{0,2}\to
    \on{map}(S^1,\frP_{0,3}^1)$.}\label{fig:Vertical-Browder}
\end{tFigure}

\section{Generators}
\label{sec:generators}
In this section, we describe several explicit generators of
$H_\bullet(\frM_{\bullet,1}^\bullet)$, which we call $\bfa$, $\bfb$, $\bfc$,
$\bfd$, $\bfe$, $\bff$, $\bfs$, and $\bfv$. Most of them can easily be visualised
in terms of slit pictures. Even though we will show several relations among
these generators in §\,\ref{sec:relations}, we discuss some immediate properties
of these classes directly after having introduced them.\vspace*{2px}

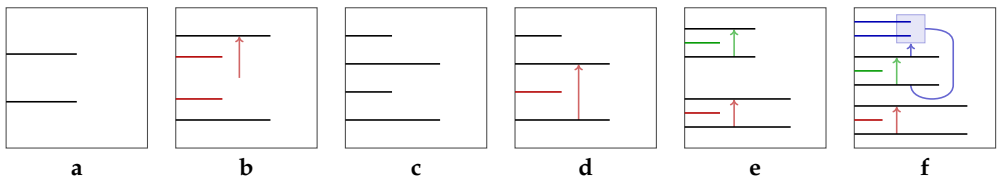
\begin{figure}[h]
  \centering
  \begin{tikzpicture}[scale=1.86]
  \draw[black!70,thin] (0,0) rectangle (1,1);
  \draw[black!70,thin] (1.2,0) rectangle (2.2,1);
  \draw[black!70,thin] (2.4,0) rectangle (3.4,1);
  \draw[black!70,thin] (3.6,0) rectangle (4.6,1);
  \draw[black!70,thin] (4.8,0) rectangle (5.8,1);
  \draw[black!70,thin] (6,0) rectangle (7,1);
  \node at (.5,-.15) {\footnotesize $\bfa\mathstrut$};
  \node at (1.7,-.15) {\footnotesize $\bfb\mathstrut$};
  \node at (2.9,-.15) {\footnotesize $\bfc\mathstrut$};
  \node at (4.1,-.15) {\footnotesize $\bfd\mathstrut$};
  \node at (5.3,-.15) {\footnotesize $\bfe\mathstrut$};
  \node at (6.5,-.15) {\footnotesize $\bff\mathstrut$};
  \draw[black!30!red!60,semithick,-{To[scale=.75,width=1.5mm]}] (1.65,.5) -- (1.65,.795);
  \draw[black!30!red!60,semithick,-{To[scale=.75,width=1.5mm]}] (4.05,.2) -- (4.05,.595);
  \fill[black!30!blue!10] (6.3,.75) rectangle (6.5,.95);
  \draw[black!30!blue!35,very thin] (6.3,.75) rectangle (6.5,.95);
  \draw[black!30!blue!60,semithick] (6.5,.85) to[out=0,in=90] (6.7,.75) -- (6.7,.45) to[out=-90,in=0] (6.55,.35) to[out=180,in=-90] (6.4,.45);
  \draw[black!30!blue!60,semithick,-{To[scale=.75,width=1.5mm]}] (6.4,.65) -- (6.4,.745);
  \draw[black!30!red!60,semithick,-{To[scale=.75,width=1.5mm]}] (6.3,.1) -- (6.3,.295);
  \draw[black!40!green!60,semithick,-{To[scale=.75,width=1.5mm]}] (6.3,.45) -- (6.3,.645);
  \draw[black!30!red!60,semithick,-{To[scale=.75,width=1.5mm]}] (5.15,.15) -- (5.15,.345);
  \draw[black!40!green!60,semithick,-{To[scale=.75,width=1.5mm]}] (5.15,.65) -- (5.15,.845);
  \draw[semithick] (.5,.67) -- (0,.67);
  \draw[semithick] (.5,.33) -- (0,.33);
  \draw[semithick] (1.87,.8) -- (1.2,.8);
  \draw[black!30!red,semithick] (1.53,.65) -- (1.2,.65);
  \draw[black!30!red,semithick] (1.53,.35) -- (1.2,.35);
  \draw[semithick] (1.87,.2) -- (1.2,.2);
  \draw[semithick] (2.73,.8) -- (2.4,.8);
  \draw[semithick] (3.07,.6) -- (2.4,.6);
  \draw[semithick] (2.73,.4) -- (2.4,.4);
  \draw[semithick] (3.07,.2) -- (2.4,.2);
  \draw[semithick] (3.93,.8) -- (3.6,.8);
  \draw[semithick] (4.27,.6) -- (3.6,.6);
  \draw[semithick,black!30!red] (3.93,.4) -- (3.6,.4);
  \draw[semithick] (4.27,.2) -- (3.6,.2);
  \draw[semithick] (6.8,.1) -- (6,.1);
  \draw[semithick] (6.8,.3) -- (6,.3);
  \draw[semithick] (6.6,.45) -- (6,.45);
  \draw[semithick] (6.6,.65) -- (6,.65);
  \draw[semithick,black!30!red] (6.2,.2) -- (6,.2);
  \draw[semithick,black!40!green] (6.2,.55) -- (6,.55);
  \draw[semithick,black!30!blue] (6.4,.8) -- (6,.8);
  \draw[semithick,black!30!blue] (6.4,.9) -- (6,.9);
  \draw[semithick] (5.55,.15) -- (4.8,.15);
  \draw[semithick] (5.55,.35) -- (4.8,.35);
  \draw[semithick] (5.3,.65) -- (4.8,.65);
  \draw[semithick] (5.3,.85) -- (4.8,.85);
  \draw[semithick,black!30!red]   (5.05,.25) -- (4.8,.25);
  \draw[semithick,black!40!green] (5.05,.75) -- (4.8,.75);
\end{tikzpicture}\vspace*{-4px}
  \caption{The generators $\bfa$, $\bfb$, $\bfc$, $\bfd$, $\bfe$, and $\bff$,
    depicted as maps from tori into the space of slit domains: the first
    toric factor is drawn in \textcolor{black!30!red}{red}, the second one in
    \textcolor{black!40!green}{green}, and the third one in
    \textcolor{black!30!blue}{blue}.}\label{fig:Generators-abcdef}
\end{figure}

\subsection{Generator \texorpdfstring{$\bfa$}{a}}

Our first generator is the integral ground class $\bfa \in H_0(\fM_{0,1}^1)$,
which already occured in §\,\ref{subsubsec:splitMod2} during the description of the
Dyer–Lashof algebra. Restricting the Pontrjagin product to $\bfa$ in one
factor defines a homology operation\looseness-1
\[\bfa\cDot-\colon H_\bullet(\fM_{g,1}^m)\longrightarrow H_\bullet(\fM_{g,1}^{m+1}).\]
of degree zero. Since adding a puncture $\frM_{g,1}^m\to \frM_{g,1}^{m+1}$
admits a stable splitting by \cite[Thm.\,1.3]{CFB-Tillmann-01}, this
homology operation is split injective, even integrally.

\subsection{Generator \texorpdfstring{$\bfb$}{b}}\label{subsec:b}
Since $\bfa$ is of even degree, we can consider the Dyer–Lashof square
$\bfb \coloneqq Q\bfa \in H_1(\fM_{0,1}^2)$ integrally. Then $\bfb$ is a
fundamental class of $\frM_{0,1}^2\simeq S^1$, and hence a generator.\looseness-1

Multiplying with $\bfb$ gives rise to a homology operation
$H_\bullet(\frM_{g,1}^m)\to H_{\bullet+1}(\frM_{g,1}^{m+2})$ of degree $1$. In
contrast to multiplying with $\bfa$, this operation is not injective integrally:
we have $2\bfb^2=0$ by graded commutativity, while $\bfb$ has infinite
order. However, for coefficients in $\bbF_2$, this operation is indeed
injective: this follows directly from (\ref{eq:Andrea}).

\subsection{Generator \texorpdfstring{$\bfc$}{c}}
The next generator is the ground class $\bfc \in H_0(\fM_{1,1})$. The
restriction of the multiplication to
$\smash{\bfc\cDot -\colon H_\bullet(\fM_{\smash{g,1}}^m)\to H_\bullet(\fM_{\smash{g+1,1}}^{m})}$
is the classical genus-stabilisation map, which is surjective for
$\bullet\le\frac23\hspace*{1px}g$ and an isomorphism for
$\bullet\le\frac23\hspace*{1px}g-\frac23$ by Harer’s stability theorem, as
already mentioned in the introduction.

One should point out that $\bfc$ lifts to a class in $\frM_{0,2}$: if we
consider the ground class $\bfc_2\in H_0(\frM_{0,2})$, see
\tref{Figure}{fig:Lifted-Generators}, then we clearly have $\bfc = s^1_*\bfc_2$

\subsection{Generator \texorpdfstring{$\bfd$}{d}}
We consider the class $\bfd\coloneqq T\bfa\in H_1(\frM_{1,1};\Z)$, which can be
depicted by the embedded circle shown in \tref{Figure}{fig:Generators-abcdef}.
Even though the computer-aided calculations in \cite[p.\,133]{Mehner} exhibit
$\bfd$ as a generator, this is one of the very few cases where it is doable ‘by
hand’: recall the incidence graph from \tref{Figure}{fig:P11Cells}. Then the
embedded circle carrying $\bfd$ intersects a single $5$-cell transversally (and no
cells of lower dimension), namely
$\Sigma\coloneqq(\cyc{0\kern-.1px {,}\kern.3px 2{,}\kern-.2px 1\kern-.4px{,}\kern.4px
  3}\!:\!\cyc{0\kern-.1px {,}\kern.3px 2{,}3}\cyc{1}\!:\!\cyc{0{,}1\kern-.4px{,}\kern.4px 2{,}3})$,
and one easily checks that its dual $\Sigma^*$ generates $H^5$ of the small cochain
complex discussed in §\,\ref{subsec:compEx}.

Again, $\bfd$ lifts to a class $\bfd_2\in H_1(\frM_{0,2})$, which is depicted
in \tref{Figure}{fig:Lifted-Generators}, i.e.\ we have $\bfd=s^1_*\bfd_2$.  Note
that the abelianisation of $\Gamma_{1,1}$ is a free abelian group generated by
the Dehn twist along an arbitrary simple closed and non-separating curve; in
particular, $\bfd$ agrees, up to sign, with such a Dehn twist.%
\subsection{Generator \texorpdfstring{$\bfe$}{e}}
We define $\bfe \coloneqq E\bfa^2 \in H_2(\fM_{1,1}^1;\bZ) \cong \bZ_2$, which
is depicted in \tref{Figure}{fig:Generators-abcdef}. Again $\bfe$ lifts to a class
$\bfe_2\in H_2(\frM_{0,2}^1)$, which is depicted in
\tref{Figure}{fig:Lifted-Generators}, i.e.\ we have $\bfe=s^1_*\bfe_2$.

Via Poincaré–Lefschetz duality, $\bfe$ is represented by a simplicial
$7$-dimensional cocycle in $(P_{1,1}^1,P_{1,1}^{\prime\,1})$ by tracking all
transversal intersections. Now a computer-aided calculation
\cite[p.\,133]{Mehner}, which is very similar to the one for $\bfd$, shows that
this cocycle is not nullhomologous, i.e.\ $\bfe$ is non-trivial and hence a
generator.

\begin{figure}
  \centering
  \begin{tikzpicture}[scale=1.85]
  \draw[dgrey,thin] (2.4,0) rectangle (3.4,.6);
  \draw[dgrey,thin] (2.4,.7) rectangle (3.4,1);
  \draw[dgrey,thin] (3.6,0) rectangle (4.6,.6);
  \draw[dgrey,thin] (3.6,.7) rectangle (4.6,1);
  \draw[dgrey,thin] (4.8,0) rectangle (5.8,.45);
  \draw[dgrey,thin] (4.8,1) rectangle (5.8,.55);
  \node at (2.9,-.162) {\footnotesize $\bfc_2\mathstrut$};
  \node at (4.1,-.162) {\footnotesize $\bfd_2\mathstrut$};
  \node at (5.3,-.162) {\footnotesize $\bfe_2\mathstrut$};
  \draw[dred!60,semithick,-{To[scale=.75,width=1.5mm]}] (4.05,.15) -- (4.05,.445);
  \draw[dred!60,semithick,-{To[scale=.75,width=1.5mm]}] (5.15,.1125) -- (5.15,.3325);
  \draw[dgreen!60,semithick,-{To[scale=.75,width=1.5mm]}] (5.15,.6625) -- (5.15,.8825);
  \draw[semithick] (2.73,.85) -- (2.4,.85);
  \draw[semithick] (3.07,.45) -- (2.4,.45);
  \draw[semithick] (2.73,.3) -- (2.4,.3);
  \draw[semithick] (3.07,.15) -- (2.4,.15);
  \draw[semithick] (3.93,.85) -- (3.6,.85);
  \draw[semithick] (4.27,.45) -- (3.6,.45);
  \draw[semithick,dred] (3.93,.3) -- (3.6,.3);
  \draw[semithick] (4.27,.15) -- (3.6,.15);
  \draw[semithick] (5.55,.1125) -- (4.8,.1125);
  \draw[semithick] (5.55,.3375) -- (4.8,.3375);
  \draw[semithick] (5.3,.6625) -- (4.8,.6625);
  \draw[semithick] (5.3,.8875) -- (4.8,.8875);
  \draw[semithick,dred]   (5.05,.225) -- (4.8,.225);
  \draw[semithick,dgreen] (5.05,.775) -- (4.8,.775);
\end{tikzpicture}
  \caption{The generators $\bfc_2$, $\bfd_2$ and $\bfe_2$, depicted as
    maps from tori into the space of parallel slit domains on two
    layers.}\label{fig:Lifted-Generators}
\end{figure}
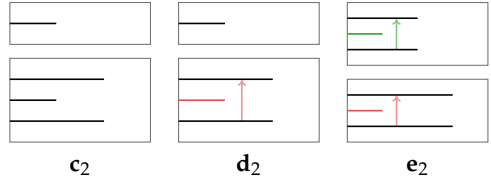

\subsection{Generator \texorpdfstring{$\bff$}{f}}
\label{subsec:f}
Using the vertical Browder bracket from §\,\ref{subsec:vertB}, we let
\mbox{$\bff\coloneqq s^1_*[\bfa,\bfe_2]\in H_3(\frM_{1,1}^2)\cong \Z_2$}. It is
supported on an embedded $3$-torus, as one can see in
\tref{Figure}{fig:Generators-abcdef}.

Again, $\bff$ is, via Poincaré–Lefschetz duality, represented by a simplicial
$9$-cocycle in $(P_{1,1}^2,P_{1,1}^{\prime\,2})$ by tracking all transversal
intersections (with orientation signs), and a computer-aided calculation
\cite[Prop.\,7.3.1]{Boes} shows that this cocycle is not nullhomologous, i.e.\
$\bff$ is non-trivial in $H_3(\frM_{1,1}^2)\cong \Z_2$.

\subsection{Generator \texorpdfstring{$\bfs$}{s}}
The homology class $\bfs\in H_3(\frM_{2,1};\Z)$ is defined in a different way:
fixing a closed Riemann surface $F$ of genus $2$, we have a map
$\imath\colon UTF\to \frM_{2,1}$ from the unit tangent bundle $UTF$ of $F$
to the moduli space $\frM_{2,1}$ by assigning to each pair $(Q,X)$ with
$Q\in F$ and $X\in T_QF$ with norm $1$ the conformal class $[F,(Q,X)]$.
Note that $UTF$ is an orientable closed $3$-manifold, so we can choose a
fundamental class \mbox{$[UTF]\in H_3(UTF;\Z)$} and define
$\bfs\coloneqq\imath_*[UTF]$.\looseness-1

In \tref{Theorem}{thm:s}, we will give the proof from \cite[Prop.\,7.1.2]{Boes}
that $\bfs$ is a \emph{rational} generator (this is already stated in
\cite[p.\,33]{Harer-91}, but without proof); showing that there is a unique
positive natural number $\mu$ such that for $\lambda\coloneqq \frac1\mu$, the
class $\lambda\bfs$ generates the free part of $H_3(\frM_{2,1};\Z)$.

\subsection{Generator \texorpdfstring{$\bfv$}{v}}

Consider the Segal–Tillmann map
$\mathrm{ST}\colon C_6(\mathring D^2)\to \frM_{2,1}$.  We show in
\tref{Theorem}{thm:v} that the induced map on $H_4(-;\Z)$ is injective. On the
other hand, it is a classical calculation \cite[§\,\textsc{iii}]{Cohen-Lada-May}
that $H_4(C_6(\R^2);\Z)$ is isomorphic to $\Z_3$, and a generator is given by
the ptolemaic epicycle $\tilde \bfv$ as in \tref{Figure}{fig:H4C6}. We set
$\bfv\coloneqq \mathrm{ST}_*(\tilde\bfv)\in H_4(\frM_{2,1};\Z)$.

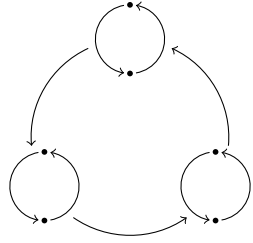
\begin{figure}
  \centering
  \begin{tikzpicture}[scale=1.3]
  \centerarc[-to](0,0)(-5:65:1);
  \centerarc[-to](0,0)(115:185:1);
  \centerarc[-to](0,0)(235:305:1);
  \centerarc[-to](.866,-.5)(-80:80:.35);
  \centerarc[-to](.866,-.5)(100:260:.35);
  \centerarc[-to](0,1)(-80:80:.35);
  \centerarc[-to](0,1)(100:260:.35);
  \centerarc[-to](-.866,-.5)(-80:80:.35);
  \centerarc[-to](-.866,-.5)(100:260:.35);
  \fill (.866,-.15) circle (.029);
  \fill (.866,-.85) circle (.029);
  \fill (-.866,-.15) circle (.029);
  \fill (-.866,-.85) circle (.029);
  \fill (0,1.35) circle (.029);
  \fill (0,.65) circle (.029);
\end{tikzpicture}
  \caption{The ptolemaic epicycle $\tilde\bfv$ generating
    $H_4(C_6(\R^2);\Z)\cong\Z_3$}\label{fig:H4C6}
\end{figure}

\section{Relations}
\label{sec:relations}
In this section, we describe several relations that hold between the generators
and operations from the previous sections. Some of them have already been found:
for example, the relation $Q\bfc=3\cdot \bfc\bfd$ is due to \cite[Ex.\,6]{Godin}
and appears in \cite[§\,1.2]{Mehner}. In particular,
$[\bfc,\bfc]=2\cdot Q\bfc=6\cdot\bfc\bfd\ne 0$, showing that the $E_2$-structure
on $\coprod_g\frM_{g,1}$ cannot be enhanced to an $E_3$-structure
\cite[Thm.\,2.5]{Fiedorowicz-Song}.

In the following, we contribute to this list of relations. Our first
observations consider the genus-stabilisation, i.e.\ multiplication with
$\bfc$. The following proposition claims that a single genus stabilisation step
cancels the Browder bracket. Our proof is essentially the same as in
\cite[Prop.\,5.3.3]{Kranhold}, using a different model. Let us point out that
the very same argument works for each algebra over Tillmann’s surface
operad.\looseness-1

\enlargethispage{\baselineskip}
\begin{prop}\label{prop:ckill}
  For each $x\in H_\bullet(\frM_{g,1}^m)$ and $x'\in H_\bullet(\frM_{g',1}^{m'})$,
  we have $\bfc\cdot [x,x']=0$.
\end{prop}
\begin{proof}
  We fix a slit picture $S\in \Par_{1,1}$ and let $\scC_2^S(2)$ be the space of
  two small numbered cubes inside this slit picture, see
  \tref{Figure}{fig:Loop-with-Genus}. By the same implanting procedure as
  before, we obtain a map
  $\smash{\lambda^S\colon \scC_2^S(2)\times \Par_{\smash{g\kern-1px,1}}^m\times
    \Par_{\smash{g',1}}^{m'}\longrightarrow \Par_{\smash{1+g+g'\kern-1px,1}}^{m+m'}}$.
  
  Now let \mbox{$\gamma\colon S^1\to \scC^S_2(2)$} be the loop that is depicted
  on the left side of \tref{Figure}{fig:Loop-with-Genus}. Then we clearly have
  $\bfc\cdot [x,x'] = \lambda^S_*([\gamma]\otimes x\otimes x')$, so it suffices to
  show that $\gamma$ is a boundary.  To do so, let $I\subseteq S^1$ be a small
  closed interval on the standard circle. Then we define a map
  $\tilde\gamma\colon (S^1\times S^1)\setminus \mathring I^2\to \scC_2^S(2)$ whose
  source is a $2$-torus with a small square removed and which is depicted on the right
  side of \tref{Figure}{fig:Loop-with-Genus}: the positions of the boxes on
  the red and on green circles parametrised by a torus, and the excluded square
  $\mathring{I}^2$ parametrises the situation in which both boxes lie
  inside the blue region—this is the only case where the disjointness
  condition can be violated. Now one readily checks that, after identifying
  $\partial I^2\cong S^1$, the boundary of $\tilde\gamma$ is homotopic to
  $\gamma$.
\end{proof}

\begin{figure}[ht]
  \centering
  \begin{tikzpicture}[scale=2]
    \draw (0,0) rectangle (1.9,1);
    \draw (0,.2) -- (.8,.2);
    \draw (0,.4) -- (.5,.4);
    \draw (0,.6) -- (.8,.6);
    \draw (0,.8) -- (.5,.8);
    \draw (.9,.7) rectangle (1.1,.9);
    \draw (1.2,.6) rectangle (1.4,.4);
    \centerarc[semithick,black!30!blue,-{To[scale=.8,width=1.6mm]}](1.3,.5)(110:-200:.42);
    \node at (1.3,.5) {\small $2$};
    \node at (1,.8) {\small $1$};
  \end{tikzpicture}
\qquad\quad
\begin{tikzpicture}[scale=2]
  \centerarc[semithick,black!35!green,{To[scale=.8,width=1.6mm]}-](1,.4)(-20:290:.35);
  \fill[white] (0,.2) rectangle (.8,.6);
  \centerarc[semithick,black!30!red,{To[scale=.8,width=1.6mm]}-](.7,.6)(-20:290:.35);
  \fill[white] (0,.4) rectangle (.5,.8);
  \draw (0,.2) -- (.8,.2);
  \draw (0,.4) -- (.5,.4);
  \draw (0,.6) -- (.8,.6);
  \draw (0,.8) -- (.5,.8);
  \draw[black!30!blue!20,fill=black!30!blue!10] (1.01641,.74962) circle (.12);
  \centerarc[black!30!red,semithick](1.01641,.74962)(115.178:135.178:.12);
  \centerarc[black!30!red,semithick](1.01641,.74962)(275.436:296.436:.12);
  \centerarc[black!35!green,semithick](1.01641,.74962)(-22.558:-2.558:.12);
  \centerarc[black!35!green,semithick](1.01641,.74962)(177.184:197.184:.12);
  \centerarc[thin,black!35!green](1,.4)(-10:120:.35);
  \centerarc[thin,black!30!red](.7,.6)(-10:90:.35);
  \draw[fill=white] (.847,.253) rectangle (1.047,.453);
  \draw[fill=white,shift={(.2,-.2)}] (.947,.253) rectangle (1.147,.453);
  \node at (.947,.353) {\small $1$};
  \node at (1.247,.153) {\small $2$};
  \draw (0,0) rectangle (1.9,1);
\end{tikzpicture}
  \caption{%
    \emph{Left:} The loop $\gamma\colon S^1\to\scC_2^S(2)$.\quad
    \emph{Right:} The map
    $\tilde\gamma\colon (S^1\times S^1)\setminus\mathring I^2\to\scC_2^S(2)$.}
  \label{fig:Loop-with-Genus}
\end{figure}

The stabilisation step can also cancel other unstable classes, which do not
decompose into a Browder bracket, as \tref{Proposition}{prop:ce0} shows.

\begin{prop}\label{prop:ce0}
  $\bfc\bfe=0$.
\end{prop}

This relation has already been claimed in \cite[p.\,14]{Mehner}, but without a
proof.

\begin{proof}
  We have $\bfc\bfe\in H_2(\frM_{2,1}^1;\Z)\cong\Z\oplus\Z_2$ and
  $2\cdot \bfc\bfe=0$, so it suffices to show that $\bfc\bfe$ is not the
  generator of the single $\Z_2$-summand.
  Here we use that $H_2(\frM_{2,1};\Z)\cong\Z_2$, and we call its generator $x$
  (we will show later in \tref{Theorem}{thm:tableG2} that it is $\bfd^2$, but this is
  not necessary for the argument). As the Pontrjagin product with $\bfa$ is
  split injective, the class $\bfa x$ generates the aforementioned
  $\Z_2$-summand of $H_2(\frM_{2,1}^1;\Z)$. We show that $\bfa x$ differs from
  $\bfc\bfe$ by inspecting their mod-2 reductions, using the isomorphism
  (\ref{eq:Andrea}) from \cite{Bianchi}: here $\bfa x$ is a non-trivial element in
  the direct summand
  $\bbF_2\gen{\bfa}\otimes H_2(\Gamma_{2,1};\mathrm{Sym}^0\caH)$, while
  $\bfc\bfe$ lies in the direct summand
  $\bbF_2\gen{1}\otimes H_2(\Gamma_{2,1};\mathrm{Sym}^1\caH)$.\looseness-1
\end{proof}

Using the vertical Browder brackets from §\,\ref{subsec:vertB} and exploiting
the fact that $\bfc$, $\bfd$, and $\bfe$ can be lifted to classes of
$\smash{\frM_{g,2}^m}$, we can show that certain classical Browder brackets are
divisible by $2$ or even vanish. This is essentially the proof from
\cite[Prop.\,5.3.7]{Kranhold}.

\begin{prop}\label{prop:bdiv}
  For $x\in H_\bullet(\frM_{g,1}^m)$, $[\bfc,x]$ and $[\bfd,x]$
  are divisible by $2$, and $[\bfe,x]=0$.
\end{prop}
\begin{proof}
  Let $\tau\coloneqq \cyc{1,\kern-.5px 2}\in \frS_2$
  be the transposition that exchanges $1$ and $2$.
  Now we note that $s^1_*\bfc_2=\bfc$ and
  $\tau_*\bfc_2=\bfc_2$, so by employing formula (\ref{eq:V1}), we get
  \[[\bfc,x] = [s^1_*\bfc_2,x]
    = s^1_*[\bfc_2,x] + s^1_*[\tau^{\vphantom 1}_*\bfc_2,x]
    = 2\cdot s^1_*[\bfc_2,x].\]
  The very same argument works for
  $\bfd$, and for $\bfe$, we note that
  $\tau_*\bfe_2=-\bfe_2$, whence we get
  $[\bfe,x] = [s^1_*\bfe_2,x]
  = s^1_*[\bfe_2,x] + s^1_*[\tau^{\vphantom 1}_*\bfe_2,x] = 0$.
\end{proof}

Note that the proof says a little bit more:
the same argument works for each class \mbox{$x\in H_\bullet(\frM_{g,n}^m)$},
where the number $n$ of boundary curves is arbitrary and
where we use the vertical Browder bracket from §\,\ref{subsec:vertB}.
This shows \mbox{$[\bfc,\bfc] = 2\cdot s^1_*[\bfc_2,\bfc] = 4\cdot s^1_*s^1_*[\bfc_2,\bfc_2]$},
i.e.\ $[\bfc,\bfc]$ is divisible by $4$, which we already
know since $[\bfc,\bfc]=2\cdot Q\bfc = 6\cdot\bfc\bfd = -4\cdot \bfc\bfd$.

We also conclude that
$[\bfc,\bfd]\in H_2(\frM_{2,1})\cong\Z_2$ and
$[\bfa,\bfd]\in H_2(\frM_{1,1}^1)\cong \Z_2$ vanish.\footnote{In
  \cite[p.\,14]{Mehner}, one finds $[\bfc,\bfd]=\bfd^2$; this, however, is
  disproven by the above argument.} Even though such an argument cannot be
applied to the brackets $[\bfa,\bfc]\in H_2(\frM_{1,1}^1)\cong \Z$ or to
$[\bfd,\bfd]\in H_3(\frM_{2,1})\cong \Z\oplus \Z_2$, these Browder brackets
vanish as well:

\begin{prop}\label{prop:ac0}
  $[\bfa,\bfc]=0$ and $[\bfd,\bfd]=0$.
\end{prop}
\begin{proof}
  The homomorphism $\delta\colon \Gamma_{1,1}^1\to \Gamma_{1,1}$ that forgets
  the single puncture is surjective, as it is part of the Birman exact sequence
  $\pi_1(F_{1,1})\to \Gamma_{1,1}^1\to \Gamma_{1,1}$. Since abelianisation is
  right exact, the induced map on first homology
  $\delta_*\colon H_1(\frM_{1,1}^1)\to H_1(\frM_{1,1})$ is epic as well. Since
  both homology groups are (abstractly) isomorphic to $\Z$, it follows that
  $\delta_*$ is an isomorphism. However, we clearly have
  $\delta_*[\bfa,\bfc] = [1,\bfc]=0$.

  We know that $[\bfd,\bfd]=-[\bfd,\bfd]$ by the graded commutativity of the
  Browder bracket, i.e.\ $[\bfd,\bfd]$ has to lie in the
  $\Z_2$-summand. However, we also know that $[\bfd,\bfd]$ is divisible by $2$,
  so it has to vanish.
\end{proof}

The following proposition shows that the $T$-operation often behaves like a
differential. More precisely, an expression
$(T\circ T)(x)\in H_{i+2}(\frM_{g+2,1}^{m-2})$ can only be non-trivial %
if it lies in a direct summand of
the form $\Z_{\smash{2^k}}$ with $k\ge 2$.

\begin{prop}\label{prop:TT}
  $T\circ T$ is divisible by $2$ and of order $2$, i.e.\ over
  $\bbF_2$ or $\Q$, we have $T\circ T=0$.
\end{prop}
\begin{proof}
  Recall the $\binom{m}{2}$-sheeted covering
  $\smash{\beta\colon \frM_{\smash{g,1}}^{m-2,2}\to \frM_{g,1}^m}$ where \emph{two} of
  the punctures are separated from the other ones and indistinguishable, and let
  $\smash{\gamma\colon \frM_{\smash{g,1}}^{m-2,1,1}\to \frM_{\smash{g,1}}^{m-2,2}}$ be the
  $2$-sheeted covering in which these two punctures can be distinguished.

  Let \mbox{$p\colon W^{1,1}\to \frM_{\smash{g,1}}^{m-2,1,1}$} be the fibre
  product of the two vertical unit tangent bundles: it is an orientable
  $(S^1\times S^1)$-bundle and we have an equivalence
  \mbox{$\vartheta^{1,1}\colon W^{1,1}\to \frM_{g,3}^{m-2}$} by regarding the
  two isolated punctures, together with their tangential directions,
  as boundary curves as in §\,\ref{subsubsec:T}. Gluing pairs of pants, we obtain
  $T\circ T=s^1_*s^2_*\vartheta^{1,1}_*p^!\gamma^!\beta^!$.

  Note that both spaces $W^{1,1}$ and $\frM_{\smash {g,1}}^{m-2,1,1}$ carry a
  free $\frS_2$-action by interchanging the two isolated punctures (and their
  tangential directions), and the bundle map $p$ is
  $\frS_2$-equivariant. Quotienting out this symmetry, we obtain the bundle
  \mbox{$\smash{q\colon W^2\to \frM_{\smash{g,1}}^{m-2,2}}$} from
  §\,\ref{subsubsec:E}. Moreover, the map $\smash{\vartheta^{1,1}}$ is equivariant
  with respect to the $\frS_2$-action on $\smash{\frM_{\smash{g,3}}^{m-2}}$
  interchanging the boundary curves $2$ and $3$. Now recall that
  \mbox{$s^2\colon \frM_{\smash{g,3}}^{m-2}\to \frM_{\smash{g+1,2}}^{m-2}$} is
  homotopic to $s^2\circ \cyc{2,\kern-.5px 3}$, i.e.\ $s^2\circ \vartheta^{1,1}$
  is homotopy $\frS_2$-invariant. Because the projection
  $\tilde\gamma\colon W^{1,1}\to W^2$ is a covering and hence a homotopy
  quotient, we obtain a map
  $\vartheta^2\colon W^2\to \frM_{\smash{g+1,2}}^{m-2}$ such that
  $s^2\vartheta^{1,1}$ is homotopic to $\vartheta^2\circ \tilde\gamma$, i.e.\ we
  have a diagram (where the right square commutes up to homotopy):\looseness-1
  \[
    \begin{tikzcd}[row sep=1.5em]
       \frM_{g,1}^{m-2,1,1}\ar[d,"\gamma"]  & W^{1,1}\ar[r,"\vartheta^{1,1}"]\ar[l,swap,"p"]\ar[d,"\tilde\gamma"] & \frM_{g,3}^{m-2}\ar[d,"s^2"]\\
      \frM_{g,1}^{m-2,2}\ar[d,"\beta"]      & W^2\ar[r,"\vartheta^2"]\ar[l,swap,"q"]                              & \frM_{g+1,2}^{m-2}\ar[d,"s^1"]\\
      \frM_{g,1}^m                          &                                                                     & \frM_{g+2,1}^{m-2}.   
    \end{tikzcd}
  \]
  Now we use that $\tilde\gamma$ is a $2$-sheeted covering, i.e.\
  $\tilde\gamma_*\tilde\gamma^!$ is the same as multiplication with $2$, whence
  $T\circ T = s^1_*s^2_*\vartheta^{1,1}_*p^!\gamma^!\beta^!
  = s^1_*\vartheta^2_*\tilde\gamma_*p^!\gamma^!\beta^!
  = s^1_*\vartheta^2_*\tilde\gamma_*\tilde\gamma^!q^!\beta^!
  = 2\cdot s^1_*\vartheta^2_*q^!\beta^!$.

  To see that $2\cdot (T\circ T)=0$, we write again
  $\tau\coloneqq \cyc{1,\kern-.5px 2}$.  Then we have $\gamma^!=\tau_*\gamma^!$,
  already on the level of singular chains, and secondly, we have
  $\tau_* p^!=-p^!\tau_*$, since $\tau$ reverses the orientations of the
  fibres. Using that $s^2_*\vartheta^{1,1}_*\tau = s^2_*\vartheta^{1,1}_*$, this
  shows that
  \begin{align*}
    2\cdot (T\circ T)(x) &=(s^1_*s^2_*\vartheta^{1,1}_*)p^!(2\gamma^!)\beta^!x
    \\ &=(s^1_*s^2_*\vartheta^{1,1}_*)p^!(\gamma^!+\tau_*\gamma^!)\beta^!x
    \\ &=(s^1_*s^2_*\vartheta^{1,1}_*)(p^!-\tau_*p^!)\gamma^!\beta^!x
    \\ &=s^1_*(s^2_*-s^2_*\tau_*)\vartheta^{1,1}_*p^!\gamma^!\beta^!x
    \\ &=0\qedhere
  \end{align*}
\end{proof}

\begin{rem}\label{rem:Rotation}
  Let us close this section with some remarks on how the rotation $R$ from
  §\,\ref{subsec:Rot} acts on some of our homology classes:
  \begin{enumerate}
  \item We have $R\bfa=0$ since $H_1(\frM_{0,1}^1)=0$. This shows that
    $R(\bfa x)= [\bfa, x] + \bfa\cdot Rx$; for example, we get
    $R(\bfa^2)=[\bfa,\bfa]=2\cdot\bfb$.
  \item We have $R\bfc=12\cdot\bfd$, see the proof below.
  \item Combining 1 and 2 with
    \tref{Proposition}{prop:ac0}, we get \mbox{$R(\bfa\bfc) = \bfa\cdot R\bfc = 12\cdot\bfa\bfd$}.
  \end{enumerate}
\end{rem}
\begin{proof}
  We are left to show that $R\bfc = 12\cdot \bfd$: we know that the mapping
  class group $\Gamma_{1,1}$ is generated by the two Dehn twists $T_a$ and $T_b$
  about the two non-separating simple closed curves whose homotopy classes are the standard
  generators of $\pi_1$ of the torus. They satisfy the braid relation
  $T_aT_bT_a=T_bT_aT_b$, so when passing to
  $H_1(\frM_{1,1})=\Gamma_{1,1}^{\text{ab}}$, both represent the generator
  $\bfd$. The Dehn twist $T_\partial$ about the single boundary curve represents
  the homology class $R\bfc\in H_1(\frM_{1,1})$. Moreover, we have
  $(T_aT_b)^6=T_\partial$ by the chain relation
  \cite[Prop.\,4.12]{Farb-Margalit}, so after abelianising, we get
  $12\cdot\bfd=R\bfc$ as claimed.\looseness-1
\end{proof}

\section{Proofs for the tables}
\label{sec:computations}
In this section, we check that our tables from the introduction are
correct. Since the two classes $\bfs$ and $\bfv$ need a more subtle treatment,
we will deal with them first.

\subsection{Proof: Generator \texorpdfstring{$\bfs$}{s}}

Recall that $\bfs$ is the image of a fundamental class under
$\imath\colon UT{F}\to \frM_{2,1}$ where $UT{F}$ is the unit tangent bundle of a
closed genus-$2$ surface ${F}$.

The map $\imath$ is a homotopy fibre: let $\mathrm{Diff}({F})$ be the group of
orientation-preserving diffeomorphisms on ${F}$: it acts transitively on the
unit tangent bundle $UT{F}$. By a continuous version of Shapiro’s lemma, the
Borel construction \mbox{$E\mathrm{Diff}({F})\times_{\mathrm{Diff}({F})} UT{F}$}
is a classifying space for the subgroup $\mathrm{Diff}_{(Q,X)}({F})$ of
diffeomorphisms fixing a point $(Q,X)\in UT{F}$. Since ${F}$ is hyperbolic, both
groups $\mathrm{Diff}_{(Q,X)}({F})$ and $\mathrm{Diff}({F})$ have contractible
components and hence are equivalent to their group of path components, which are
the mapping class groups $\Gamma_{2,1}$ and $\Gamma_2$, respectively. Using that
$\frM_{2,1}$ is a classifying space for $\Gamma_{2,1}$, the above Borel
construction gives rise to a homotopy fibre sequence
$UT{F}\to \frM_{2,1}\to B\Gamma_2$, where $\imath$ occurs as the fibre
inclusion.\looseness-1

\enlargethispage{\baselineskip}
\begin{theo}\label{thm:s}
  The class $\bfs=\imath_*[UT{F}]$ is a generator of $H_3(\frM_{2,1};\Q)$.
\end{theo}
\begin{proof}
  In a first step, we determine the rational homology of $UT{F}$ by inspecting
  the Serre spectral sequence of the $S^1$-bundle $UT{F}\to {F}$: note that
  $\pi_1({F})$ acts trivially on the homology of the fibre $H_\bullet(S^1;\Q)$
  since $UT{F}$ is orientable. Therefore, the second page of the spectral
  sequence is given as in \tref{Figure}{fig:SSeqTF}. The only interesting
  differential is given by multiplication with the Euler number $\chi({F})=-2$,
  and hence is an isomorphism. Therefore, we obtain
  $H_\bullet(UT{F};\Q)=(\Q,\Q^4,\Q^4,\Q)$. We note that the identifications
  $H_1(UT{F};\Q)\cong H_1({F};\Q)\cong H_2(UT{F};\Q)=:A$ are even isomorphisms of
  $\Gamma_2$-representations: each mapping class $\phi\in \Gamma_2$ induces a
  bundle automorphism of $UT{F}\to {F}$ and hence automorphism of spectral
  sequences converging to the respective automorphism of
  $H_\bullet(UT{F};\Q)$.%

  \begin{figure}[ht]
    \centering
    \begin{tikzpicture}[yscale=.9]
  \draw[black!20,thick,-to] (0,-.5) -- (0,1.5);
  \draw[black!20,thick,-to] (-.5,0) -- (2.5,0);
  \node at (0,0) {$\Q$};
  \node at (1,0) {$\Q^4$};
  \node at (2,0) {$\Q$};
  \node at (0,1) {$\Q$};
  \node at (1,1) {$\Q^4$};
  \node at (2,1) {$\Q$};
  \draw[-to] (1.8,.2) -- (.2,.8);
\end{tikzpicture}\vspace*{-2px}
    \caption{The second page of the rational Serre spectral sequence for
      $S^1\to UT{F}\to {F}$}\label{fig:SSeqTF}
  \end{figure}
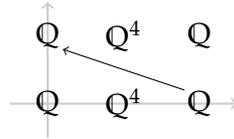

  Now we look at the Serre spectral sequence for the fibration
  \mbox{$UT{F}\to \frM_{2,1}\to B\Gamma_2$}. Here we use that we already know
  the rational homology of $\frM_{2,1}$: it is $(\Q,0,0,\Q)$, see
  \cite[Lem.\,1.3]{Harer-91}. Moreover, the group $\Gamma_2$ is rationally
  acyclic \cite{Igusa-60}, so if we put $M_k\coloneqq H_k(\Gamma_2;A)$, then the
  second page of the Serre spectral sequence is given as in
  \tref{Figure}{fig:SSeqM211}. The map of interest, namely
  \mbox{$\imath_*\colon \Q=H_3(UT{F};\Q)\to H_3(\frM_{2,1})$}, is given by the
  edge morphism
  \mbox{$E^2_{0,3}\to E^\infty_{0,3}\hookrightarrow H_3(\frM_{2,1};\Q)$}; hence
  we only have to show that all differentials that reach $E^2_{0,3}$ are
  trivial.

  \begin{figure}[h]
    \centering
    \begin{tikzpicture}[yscale=.82,xscale=1.1]
  \draw[black!20,thick,-to] (0,-.5) -- (0,3.5);
  \draw[black!20,thick,-to] (-.5,0) -- (4.5,0);
  \node at (0,0) {$\Q$};
  \node[black!30] at (1,0) {$0$};
  \node[black!30] at (2,0) {$0$};
  \node[black!30] at (3,0) {$0$};
  \node[black!30] at (4,0) {$0$};
  \node at (0,1) {$M_0$};
  \node at (1,1) {$M_1$};
  \node at (2,1) {$M_2$};
  \node at (3,1) {$M_3$};
  \node at (4,1) {$\dotsb$};
  \node at (0,2) {$M_0$};
  \node at (1,2) {$M_1$};
  \node at (2,2) {$M_2$};
  \node at (3,2) {$M_3$};
  \node at (4,2) {$\dotsb$};
  \node at (0,3) {$\Q$};
  \node[black!30] at (1,3) {$0$};
  \node[black!30] at (2,3) {$0$};
  \node[black!30] at (3,3) {$0$};
  \node[black!30] at (4,3) {$0$};
  \draw[-to] (1.7,2.25) -- (.35,2.75);
  \draw[-to] (2.7,1.25) -- (.2,2.75);
\end{tikzpicture}\vspace*{-2px}
    \caption{The second page of the rational Serre spectral sequence for
      \mbox{$UT{F}\to \frM_{2,1}\to B\Gamma_{2}$}, together with the only two
      differentials that could possibly kill $E^2_{0,3}$}\label{fig:SSeqM211}
  \end{figure}
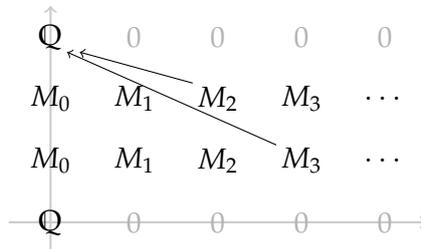
 
  To this aim, it is enough to show that $E^2_{2,2}=M_2$ and $E^2_{3,1}=M_3$ are
  trivial; we even show that \emph{all} $M_k$ are trivial: we start by noting
  that $M_0=E^\infty_{0,1}$ and $M_1=E^\infty_{1,1}$ are trivial. Now we use
  from \cite[p.\,33]{Harer-91} that $H_\bullet(\Gamma_2^1;\Q)\cong
  (\Q,0,\Q)$. If we consider the Serre spectral sequence $\bar E$ associated
  with the fibre sequence ${F}\to B\Gamma_2^1\to B\Gamma_2$, then
  $\smash{\bar E^2_{k,1}\cong M_k}$, and since $\Gamma_2$ is rationally acyclic
  and $M_0=M_1=0$, the spectral sequence looks as in \tref{Figure}{fig:SSeqM21}
  and converges to $\smash{H_\bullet(\Gamma_2^1;\Q)}$.  This already shows that
  $M_k=\bar E^\infty_{k,1}=0$ for $k\ge 3$. Finally, since $\bar E^2_{0,2}=\Q$
  has to survive, the differential $d_{2,1}$ has to be trivial; and thus,
  $M_2=\bar E^\infty_{2,1}=0$ as desired.
\end{proof}

\begin{figure}[ht]
  \centering
  \begin{tikzpicture}[yscale=.82,xscale=1.1]
  \draw[black!20,thick,-to] (0,-.5) -- (0,2.5);
  \draw[black!20,thick,-to] (-.5,0) -- (4.5,0);
  \node at (0,0) {$\Q$};
  \node[black!30] at (1,0) {$0$};
  \node[black!30] at (2,0) {$0$};
  \node[black!30] at (3,0) {$0$};
  \node[black!30] at (4,0) {$0$};
  \node[black!30] at (0,1) {$0$};
  \node[black!30] at (1,1.04) {$0$};
  \node at (2,1) {$M_2$};
  \node at (3,1) {$M_3$};
  \node at (4,1) {$\dotsb$};
  \node at (0,2) {$\Q$};
  \node[black!30] at (1,2) {$0$};
  \node[black!30] at (2,2) {$0$};
  \node[black!30] at (3,2) {$0$};
  \node[black!30] at (4,2) {$0$};
  \draw[-to] (1.7,1.2) -- (.3,1.8);
  \node at (1.4,1.6) {\scriptsize $d_{2,1}$};
\end{tikzpicture}
  \caption{The second page of the rational Serre spectral sequence $\bar E$ for
    ${F}\to B\Gamma_2^1\to B\Gamma_{2}$; here the differential $d_{2,1}$ has to be
    trivial, as the limit $H_2(B\Gamma_2^1;\Q)$ has dimension 1.}\label{fig:SSeqM21}
\end{figure}
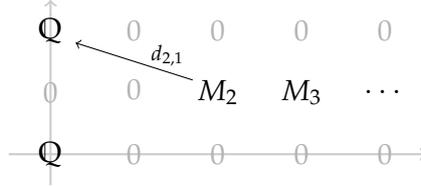

\subsection{Proof: Generator \texorpdfstring{$\bfv$}{v}}

In this subsection, we summarise the proof of \cite[Prop.\,7.2.2]{Boes} that the
homology class $\bfv$ generates the $\Z_3$-summand of
$H_4(\frM_{2,1};\Z)\cong \Z_2\oplus \Z_3$.  Recall that $\bfv$ was defined using
the Segal–Tillmann map $\mathrm{ST}\colon C_6(\R^2)\to \frM_{2,1}$, by taking
the ptolemaic epicycle $\tilde\bfv$, which generates $H_4(\Br_6;\Z)\cong \Z_3$ and is depicted in
\tref{Figure}{fig:H4C6}. It hence suffices to show:\looseness-1

\begin{theo}\label{thm:v}
  The Segal–Tillmann map $\mathrm{ST}\colon C_6(\R^2)\to \frM_{2,1}$ induces an
  injection on the $3$-torsion of $H_4$, and hence, $\bfv$ generates the
  $\Z_3$-summand of $H_4(\frM_{2,1};\Z)$.
\end{theo}

The proof needs a little preparation: first of all, recall that the
Segal–Tillmann map has an algebraic counterpart on the level of mapping class
groups, which is a group homomorphism $\mathrm{ST}\colon \Br_6\to \Gamma_{2,1}$.
Now note that $\Br_6$ is the same as the mapping class group
$\Gamma_{0,1}^6$ of a disc with six punctures. Here we can cap the boundary with
a disc, resulting in a homomorphism $\vartheta\colon \Gamma_{0,1}^6\to \Gamma^6_0$. The
main ingredient of the proof of \tref{Theorem}{thm:v} is the following
observation from \cite[Prop.\,7.2.3]{Boes}:\looseness-1

\begin{lem}\label{lem:ST3tor}
  The map $\vartheta\colon \Gamma_{0,1}^6\to \Gamma^6_0$ induces an isomorphism on
  the $3$-torsion of $H_4$.
\end{lem}

The proof of \tref{Lemma}{lem:ST3tor} relies on the observation that $\vartheta$
is a composition of two projection maps in short exact sequences: the capping
sequence $\Z\to \Gamma_{0,1}^6\to \Gamma_0^{6,1}$, whose target is the mapping
class group of a sphere with \emph{seven} punctures, where the first six
punctures are allowed to be permuted and the last one is isolated; and the
Birman sequence
$\pi_1(S^2\setminus\{p_1,\dotsc,p_6\})\to \Gamma_0^{\smash{6,1}}\to \Gamma_0^6$ that
forgets the isolated puncture. Both projections induce isomorphisms on the
$3$-torsion of $H_4$, as Boes shows by considering the respective
Lyndon–Hochschild–Serre spectral sequences with coefficients in the ring
$\Z_{(3)}\subseteq\Q$ where all primes except for $3$ have been inverted. As a
composition of these two maps, the same holds for $\vartheta$: this proves the
Lemma. Now we are ready to give the main proof of \tref{Theorem}{thm:v}:\vspace*{.5\baselineskip}

\begin{tFigure}
  \centering
  \includegraphics[width=.7\columnwidth]{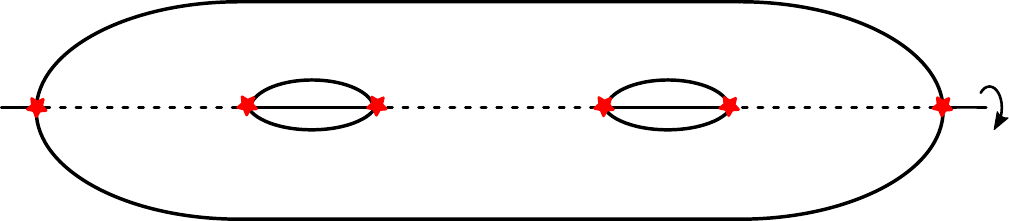}
  \caption{The hyperelliptic involution of the closed surface of genus
    2 and its 6 fixed points}
  \label{fig:HyperellipticInvolutionF2}
\end{tFigure}

\begin{proof}[Proof of \tref{Theorem}{thm:v}]
  The Birman–Hilden theorem \cite{Birman-Hilden} (see also
  \cite[§\,9.4.2]{Farb-Margalit}) gives rise to a map
  $\gamma\colon \Gamma_{2}\to \Gamma^6_0$ as follows: let ${F}$ be a closed
  surface of genus $2$ and consider the hyperelliptic involution
  $\tau\colon {F}\to {F}$ as in
  \tref{Figure}{fig:HyperellipticInvolutionF2}. Let $\caP\subseteq {F}$ be the
  set of fixed points: it has exactly six elements. Then $\tau^2=\mathrm{id}$
  and ${F}\hspace*{.5px}/\hspace*{-1px}\langle\tau\rangle$ is identified with a
  sphere $S^2$; we call the image of the fixed points $\caP'\subseteq S^2$. Now
  let $\mathrm{SHomeo}({F})\subseteq\mathrm{Homeo}({F})$ be the subgroup of
  those orientation-preserving homeomorphisms that commute with $\tau$, called
  \emph{symmetric homeomorphisms}. Since each symmetric homeomorphism has to fix
  $\caP$ as a set, we obtain a map
  $\mathrm{SHomeo}({F})\to \mathrm{Homeo}_{\caP'}(S^2)$ to the group of
  orientation-preserving homeomorphisms on $S^2$ that fix $\caP'$ as a set, by
  quotienting out $\tau$. By passing to $\pi_0$, we obtain a zig-zag of
  group homomorphisms
  \[\begin{tikzcd}\Gamma_{2} &
      \pi_0(\mathrm{SHomeo}({F}))\ar[l,swap,"\imath"]\ar[r,"\jmath"] &
      \Gamma_0^6.\end{tikzcd}\]
  Since all standard generators of $\Gamma_2$ can easily be represented by
  symmetric homeomorphisms, $\imath$ is surjective. Secondly, one ingredient for
  the Birman–Hilden theorem states that if two symmetric homeomorphisms $\phi$
  and $\psi$ are isotopic, then they are even isotopic through \emph{symmetric}
  homeomorphisms, see \cite[Prop.\,9.4]{Farb-Margalit}. This shows that $\imath$
  is also injective, and hence an isomorphism of groups: we put
  $\gamma\coloneqq \jmath\circ\imath^{-1}$.

  If we consider the capping morphism $\beta\colon \Gamma_{2,1}\to \Gamma_2$ and
  the Segal–Tillmann map $\mathrm{ST}\colon \Br_6\to \Gamma_{2,1}$, then
  the composition $\beta\circ\mathrm{ST}\colon \Br_6\to \Gamma_2$ can,
  without any choice of isotopy, be described in terms of symmetric
  homeomorphisms, as the Dehn twists about the curves $\alpha_1,\dotsc,\alpha_5$
  in \tref{Figure}{fig:STDehn} (now for genus $2$) are already
  symmetric. Quotienting out $\tau$, we recover the standard generators of
  $\Gamma_0^6=\Br_6(S^2)$, the sixth braid group of the sphere.  This
  shows that $\gamma\circ\beta\circ\mathrm{ST}$ agrees with the capping map
  $\vartheta$, so since $\vartheta$ induces an isomorphism on the $3$-torsion of
  $H_4$, it follows that $\mathrm{ST}$ induces an injection on the $3$-torsion
  of $H_4$, which finishes the proof.
\end{proof}

\subsection{Proof: Rest of the tables}

Now we have everything at hand to verify the remaining entries in our three
tables of generators from the introduction. Recall that we already know the
isomorphism types of the abelian groups $H_i(\frM_{g,1}^m)$ for small $g$ and
$m$; we only have to argue why our classes are generators.

\begin{theo}\label{thm:tableG0}
  The generators for $H_\bullet(\frM_{0,1}^\bullet;\Z)$ are given as in \tref{Table}{tab:g0}.
\end{theo}
\begin{proof}
  Both $\frM_{0,1}$ and $\frM_{0,1}^1$ are contractible, with ground classes $1$
  and $\bfa$. Moreover, $\frM_{0,1}^m$ is a classifying space for the braid
  group $\Br_m$, so $H_1(\frM_{0,1}^m)$ is generated by a loop corresponding to
  an elementary braid: this is the class $\bfa^{m-2}\bfb$.  Finally, the
  Pontrjagin product with $\bfb$ is injective over $\bbF_2$, as discussed in
  §\,\ref{subsec:b}, and hence $\bfb^2$ is non-trivial in
  $H_2(\frM_{0,1}^4;\bF_2)$. This shows that $\bfb^2$ cannot vanish integrally,
  and hence generates $H_2(\frM_{0,1}^4)=\Z_2$.  Finally, $\bfa\bfb^2\ne 0$ by
  the injectivity of $\bfa\cdot-$.
\end{proof}

\begin{theo}\label{thm:tableG1}
  The generators for $H_\bullet(\frM_{1,1}^\bullet;\Z)$ are given as in \tref{Table}{tab:g1}.
\end{theo}
\begin{proof}
  The row for the $0$\textsuperscript{th} homology is obvious, as $\bfa^m\bfc$ is
  the ground class of $\frM_{1,1}^m$. Moreover, we have already seen that $\bfd$
  and $\bfe$ generate their respective homology groups. By the injectivity of
  $\bfa\cdot -$, we conclude that all $\bfa^m\bfd$ are of infinite order, and hence
  generate the free part of $H_1(\frM_{1,1}^m)$ for $m\le 4$. As before,
  $\bfb\bfc$ is non-trivial modulo $2$ and hence non-trivial integrally. On the
  other hand, we see $2\cdot\bfb\bfc=[\bfa,\bfa]\cdot \bfc=0$ by
  \tref{Proposition}{prop:ckill}. Thus, $\bfb\bfc$ generates the $\Z_2$-summand
  of $H_1(\frM_{1,1}^2)$. Again, by the injectivity of $\bfa\cdot -$, the same
  holds for $\bfa\bfb\bfc$ and $\bfa^2\bfb\bfc$. Having found all generators
  for $H_0$ and $H_1$, we now consider $H_2$: since $\bfe\ne 0$, we conclude
  that $\bfa^m\bfe\ne 0$. As before, the mod-$2$ reduction of $\bfb\bfd$ is
  non-trivial, and hence, $\bfb\bfd$ itself is non-trivial. Under the
  isomorphism (\ref{eq:Andrea}), we see that the mod-$2$ reduction of $\bfa\bfe$
  lies in the direct summand
  $\bbF_2\gen{\bfa}\otimes H_1(\frM_{1,1};\on{Sym}^1\caH)$, while $\bfb\bfd$
  lies in the direct summand
  $\bbF_2\gen{Q\bfa}\otimes H_1(\frM_{1,1};\on{Sym}^0\caH)$, so since they are
  non-trivial, they cannot agree. By the injectivity of adding a puncture, we
  get that $\bfa^2\bfe$ and $\bfa\bfb\bfd$, as well as $\bfa^3\bfe$ and
  $\bfa^2\bfb\bfd$ are different generators. Finally, $\bfb^2\bfc$ is
  non-trivial and differs from $\bfa^3\bfe$ and $\bfa^2\bfb\bfd$ by the very
  same isomorphism (\ref{eq:Andrea}). For $H_3$, we already know that $\bff$ is
  a generator, so the same applies to $\bfa\bff$ and $\bfa^2\bff$. Moreover, the
  mod-$2$ reduction of $\bfb\bfe$ is non-trivial, and under the isomorphism
  (\ref{eq:Andrea}), it lies in the summand
  $\bbF_2\gen{Q\bfa}\otimes H_1(\frM_{1,1};\on{Sym}^1\caH)$, while $\bfa\bff$
  lies in the summand
  $\bbF_2\gen{\bfa}\otimes H_1(\frM_{1,1};\on{Sym}^2\caH)$: this shows that
  $\bfb\bfe$ and $\bfa\bff$ are linearly independent. As both classes are
  of order $2$, they generate the $2$-torsion part of $H_3(\frM_{1,1}^3;\Z)$. By the
  injectivity of adding a puncture, $\bfa^2\bff$ and $\bfa\bfb\bfe$ are linearly
  independent. Moreover, $\bfb^2\bfd$ differs from $\bfa^2\bff$ and
  $\bfa\bfb\bfe$ by the same argument as before, and $\bfb\bff$ is non-trivial
  and of order $2$.
\end{proof}

\begin{theo}\label{thm:tableG2}
  The generators for $H_\bullet(\frM_{2,1}^\bullet;\Z)$ are given as in \tref{Table}{tab:g2}.
\end{theo}
\begin{proof}
  Again, $H_0$ is obvious.  Since $\bfd$ is represented by a single Dehn twist
  about a non-separating curve in a surface of genus $1$ and one boundary curve,
  the same applies to $\bfc\bfd$, and we know that each such Dehn twist
  generates $H_1(\frM_{2,1})=\Gamma_{2,1}^{\text{ab}}$. Using again that adding
  a puncture is injective, we find that $\bfa\bfc\bfd$ and $\bfa^2\bfc\bfd$ are
  generators.  Moreover, $\bfb\bfc^2$ is non-trivial modulo $2$, and hence
  non-trivial integrally, and it is of order $2$, as already $\bfb\bfc$ is.
  The non-triviality of $\bfd^2$ is due to \cite[Ex.\,4]{Godin}, and the
  non-triviality of the mod-$2$ reduction of $T\bfe$ has been shown by a
  computer-aided calculation in \cite[p.\,133]{Mehner}. Since $T\bfe$ is of order $2$
  and $\lambda\bfs$ is a free generator by \tref{Theorem}{thm:s}, the
  classes $\lambda\bfs$ and $T\bfe$ generate
  $H_3(\frM_{2,1})$. The very same arguments as before justify the entries
  $\lambda\bfa\bfs$, $\bfa\cDot T\bfe$, $\lambda\bfa^2\bfs$,
  $\bfa^2\cDot T\bfe$, and $\bfb\bfd^2$. By \tref{Theorem}{thm:v}, we know that
  $\bfv$ generates the $\Z_3$-summand in $H_4(\frM_{2,1})$. This justifies the
  entries $\bfa\bfv$ and $\bfa^2\bfv$, and again, $\bfb\cDot T\bfe$ is
  non-trivial and, using the isomorphism (\ref{eq:Andrea}), different from
  $\bfa^2\bfv$.
\end{proof}

\appendix

\section{Calculations modulo 2}
\label{sec:mod2}
In this short Appendix, we discuss analogous tables for homology with coefficients in
$\bbF_2$. The mere $\bF_2$-Betti numbers can easily be computed via the
universal coefficient theorem from our integral Tables \ref{tab:g0},
\ref{tab:g1}, and \ref{tab:g2}. Moreover, the mod-$2$ reductions of our integral
generators (for the $\Z$- and the $\Z_2$-summands) appear as generators for the
$\bbF_2$-homology. However, several further summands arise from the torsion
part in the universal coefficient theorem, and we can identify some further
generators of these summands.\looseness-1

First of all, note that \tref{Table}{tab:m2g0} for $g=0$ only shows the first
grading components of the Dyer–Lashof algebra $\bbF_2[Q^j\bfa]_{j\ge 0}$. For
the other two Tables \ref{tab:m2g1} and \ref{tab:m2g2}, we extensively use the
decomposition (\ref{eq:Andrea}) from \cite{Bianchi}; and for
\tref{Table}{tab:m2g1}, we additionally use from \cite[p.\,13]{Mehner} that
$E\bfb$ is non-trivial in $H_3(\frM_{1,1}^1)$, so it must generate the
$\on{Tor}$-summand. By the naturality of the short exact sequence from the
universal coefficient theorem, the same applies to $\bfa^m\cDot E\bfb$.
For \tref{Table}{tab:m2g2}, the last Nishida relation from §\,\ref{subsubsec:relE2}
tells us that for the mod-$2$ Bockstein morphism $\beta$, we have
$\beta Q\bfd = \bfd^2 + [\bfd,\beta\bfd]$. The second summand vanishes by
\tref{Proposition}{prop:bdiv}, and hence $\beta Q\bfd = \bfd^2\ne 0$. Thus,
$Q\bfd$ is not the mod-$2$ reduction of an integral class, and hence generates
the single $\mathrm{Tor}$-summand of $H_3(\frM_{2,1};\bF_2)$.  Moreover,
\cite[p.\,14]{Mehner} shows that $TE\bfb$, $\bfc\cDot E\bfb$, $\bfd\bfe$, and
$\bfd\cDot E\bfb$ are non-trivial and that $\bfc\cDot E\bfb$ and $\bfd\bfe$ are
independent.\vspace*{.5\baselineskip}

\begin{hTable}
  \footnotesize
  \begin{tabular}{rllllll}    
  \toprule
  ~    & $\fM^{\vphantom 1}_{0,1}$ & $\fM_{0,1}^1$ & $\fM_{0,1}^2$ & $\fM_{0,1}^3$ & $\fM_{0,1}^4$ & $\fM_{0,1}^5$      \\
  \hline
  $0$~ & $1$         & $\bfa$        & $\bfa^2$      & $\bfa^3$      & $\bfa^4$      & $\bfa^5$           \\ 
  $1$~ &             &               & $\bfb$        & $\bfa\bfb$    & $\bfa^2\bfb$  & $\bfa^3\bfb$       \\
  $2$~ &             &               &               &               & $\bfb^2$      &  $\bfa\bfb^2$      \\
  $3$~ &             &               &               &               & $Q\bfb$       &  $\bfa\cdot Q\bfb$ \\
  \bottomrule
\end{tabular}

  \caption{Homology groups over $\bbF_2$ for $g=0$ and $m = 0,\dotsc,5$}\label{tab:m2g0}
\end{hTable}

\begin{tTable}
  \footnotesize
  \begin{tabular}{rlllll} 
  \toprule
  ~    & $\fM^{\vphantom 1}_{1,1}$ & $\fM_{1,1}^1$ & $\fM_{1,1}^2$                & $\fM_{1,1}^3$                                         & $\fM_{1,1}^4$                                                               \\
  \hline
  $0$~ & $\bfc$      & $\bfa\bfc$    & $\bfa^2\bfc$                 & $\bfa^3\bfc$                                          & $\bfa^4\bfc$                                                                              \\
  $1$~ & $\bfd$      & $\bfa\bfd$    & $\bfa^2\bfd,\bfb\bfc$        & $\bfa^3\bfd,\bfa\bfb\bfc$                             & $\bfa^4\bfd,\bfa^2\bfb\bfc$                                                               \\
  $2$~ &             & $\bfe$        & $\bfa\bfe,\bfb\bfd,?$        & $\bfa^2\bfe,\bfa\bfb\bfd,?$                           & $\bfa^3\bfe,\bfa^2\bfb\bfd,\bfb^2\bfc,?$                                                  \\
  $3$~ &             & $E\bfb$       & $\bff,\bfa\!\cdot\! E\bfb,?$~~ & $\bfa\bff,\bfb\bfe,\bfa^2\!\cdot\!E\bfb,?^{\oplus 2}$~~ & $\bfa^2\bff,\bfa\bfb\bfe,\bfb^2\bfd,Q\bfb\!\cdot\!\bfc,\bfa^3\!\cdot\!E\bfb,?^{\oplus 3}$ \\
  $4$~ &             &               & $?$                          & $\bfb\!\cdot\!E\bfb,?^{\oplus 3}$                     & $\bfb\bff,Q\bfb\!\cdot\!\bfd,\bfa\bfb\!\cdot\!E\bfb,?^{\oplus 5}$                         \\
  $5$~ &             &               &                              & $?$                                                   & $?^{\oplus 5}$                                                                            \\
  $6$~ &             &               &                              &                                                       & $?^{\oplus 2}$                                                                            \\
  \bottomrule
\end{tabular}

  \caption{Homology groups over $\bbF_2$ for $g=1$ and $m = 0,\dotsc,4$;
    here the symbol $?^{\oplus k}$ means that $k$ generators have not yet been found.}\vspace*{.8\baselineskip}\label{tab:m2g1}
\end{tTable}

\begin{tTable}
  \footnotesize
  \begin{tabular}{rlll}
  \toprule
  ~    & $\fM^{\vphantom 1}_{2,1}$               & $\fM_{2,1}^1$                                                               & $\fM_{2,1}^2$                                                                                                    \\
  \hline
  $0$~ & $\bfc^2$                  & $\bfa\bfc^2$                                                                & $\bfa^2\bfc^2$                                                                                                   \\
  $1$~ & $\bfc\bfd$                & $\bfa\bfc\bfd$                                                              & $\bfa^2\bfc\bfd,\bfb\bfc^2$                                                                                      \\
  $2$~ & $\bfd^2,?$                & $\bfa\bfd^2,?$                                                              & $\bfa^2\bfd^2,\bfb\bfc\bfd,?^{\oplus 3}$                                                                         \\
  $3$~ & $\lambda\bfs,T\bfe,Q\bfd$~~ & $\lambda\bfa\bfs,\bfa\cDot T\bfe,\bfa\cDot Q\bfd,\bfc\cDot E\bfb,\bfd\bfe$~~  & $\lambda\bfa^2\bfs,\bfa^2\cDot T\bfe,\bfa^2\cDot Q\bfd,\bfa\bfc\cDot E\bfb,\bfa\bfd\bfe,\bfb\bfd^2,?^{\oplus 3}$ \\
  $4$~ & $TE\bfb,?$                & $\bfa\cDot TE\bfb,\bfd\cDot E\bfb,?^{\oplus 2}$                             & $\bfa^2\cDot TE\bfb,\bfa\bfd\cDot E\bfb,\bfb\cDot T\bfe,?^{\oplus 7}$                                            \\
  $5$~ & $?$                       & $?^{\oplus 3}$                                                              & $\bfb\cDot TE\bfb,?^{\oplus 10}$                                                                                 \\
  $6$~ &                           & $?$                                                                         & $?^{\oplus 9}$                                                                                                   \\
  $7$~ &                           &                                                                             & $?^{\oplus 4}$                                                                                                   \\
  $8$~ &                           &                                                                             & $?$                                                                                                              \\
  \bottomrule
\end{tabular}

  \caption{Homology groups over $\bbF_2$ for $g=2$ and $m = 0,1,2$}\label{tab:m2g2}
\end{tTable}

\printbibliography[heading=bibintoc]

\small

\noindent\textbf{Carl-Friedrich Bödigheimer}\quad
Mathematisches Institut,
Universität Bonn,
Endenicher Allee 60,
53115 Bonn,
Germany,
\mail{cfb@math.uni-bonn.de}
\bigskip

\noindent\textbf{Felix Boes}\quad
Institute of Computer Science,
Universität Bonn,
Friedrich-Hirzebruch-Allee 8,
53115 Bonn,
Germany,
\mail{boes@cs.uni-bonn.de}
\bigskip

\noindent\textbf{Florian Kranhold}\quad
Mathematisches Institut,
Universität Bonn,
Endenicher Allee 60,
53115 Bonn,
Germany,
\mail{kranhold@math.uni-bonn.de}

\end{document}